\documentclass[a4paper]{amsart}
\pdfoutput=1
\usepackage[all,hyperref]{paper_diening}
\renewcommand{\norm}[1]{\left\lVert#1\right\rVert}
\renewcommand{\abs}[1]{\left\lvert#1\right\rvert}

\usepackage[backend=biber,maxbibnames=5,sorting=nyt,maxalphanames=5,style=alphabetic,bibencoding=utf8,giveninits,url=false,isbn=false]{biblatex}
\AtBeginBibliography{\scriptsize}

\usepackage{xparse}
\let\realItem\item % save a copy of the original item
\makeatletter
\NewDocumentCommand\myItem{ o }{%
   \IfNoValueTF{#1}%
      {\realItem}% add an item
      {\realItem[#1]\def\@currentlabel{#1}}% add an item and update label
}
\makeatother
\usepackage{enumitem}
\setlist[enumerate]{
    before=\let\item\myItem,       % use \myItem in enumerate
    label=\textnormal{(\arabic*)}, % format the label
    widest=(2')                    % set the widest label
}

\usepackage{amsmath,amsthm,amssymb}

\usepackage[]{algorithm2e}
\usepackage{framed}

\usepackage{booktabs}

\usepackage{epsfig}
\usepackage[utf8]{inputenc}
\usepackage[T1]{fontenc}
\usepackage{dsfont}
\usepackage{faktor}
\usepackage{tikz}
\usepackage{pgfplots}
\pgfplotsset{
  width=.65\linewidth,
  axis background/.style={fill=black!5!white},
  grid style={densely dotted,semithick},
  legend style={
    legend columns=1,
    legend pos=outer north east
  },
  compat=newest % compatibility for old pgfplots versions
}
 
\pgfplotscreateplotcyclelist{MyColors}{%
    {red,mark = *,every mark/.append style={solid,scale=0.4,fill=red}},
    {teal,mark = x,every mark/.append style={solid,scale=0.4,fill=teal}},
    {blue,mark = square*,every mark/.append style={solid,scale=0.4,fill=blue}},
{dotted,red,mark = *,every mark/.append style={solid,scale=0.4,fill=red}},
    {dotted,teal,mark = x,every mark/.append style={solid,scale=0.4,fill=teal}},
    {dotted,blue,mark = square*,every mark/.append style={solid,scale=0.4,fill=blue}},
{dash dot,red,mark = *,every mark/.append style={solid,scale=0.4,fill=red}},
    {dash dot,teal,mark = x,every mark/.append style={solid,scale=0.4,fill=teal}},
    {dash dot,blue,mark = square*,every mark/.append style={solid,scale=0.4,fill=blue}},}

\pgfplotscreateplotcyclelist{MyColors4}{
    {red,mark = *,line width=1.2,every mark/.append style={solid,scale=0.4,fill=red}},
    {teal,mark = x,line width=1.2,every mark/.append style={solid,scale=0.4,fill=teal}},
    {dotted,blue,mark = square*,line width=1.2,every mark/.append style={solid,scale=0.4,fill=blue}},
    {dotted,green,mark = o,line width=1.2,every mark/.append style={solid,scale=0.4,fill=green}},
	{dotted,red,line width=1.2,mark = *,every mark/.append style={solid,scale=0.4,fill=red}},
    {dotted,teal,line width=1.2,mark = x,every mark/.append style={solid,scale=0.4,fill=teal}},
    {dotted,blue,line width=1.2,mark = square*,every mark/.append style={solid,scale=0.4,fill=blue}},
    {dotted,green,line width=1.2,mark = o,every mark/.append style={solid,scale=0.4,fill=green}},
	{dash dot,red,line width=1.2,mark = *,every mark/.append style={solid,scale=0.4,fill=red}},
    {dash dot,teal,line width=1.2,mark = x,every mark/.append style={solid,scale=0.4,fill=teal}},
    {dash dot,blue,line width=1.2,mark = square*,every mark/.append style={solid,scale=0.4,fill=blue}},
    {dash dot,green,line width=1.2,mark = o,every mark/.append style={solid,scale=0.4,fill=green}},}

\usepackage{mathrsfs}
\numberwithin{equation}{section}

\newcommand{\Div}{\divergence}

\newcommand{\dd}{\,\mathrm{d}}
\newcommand{\ds}{\dd s}
\newcommand{\dt}{\dd t}

\providecommand{\seminormtmp}[2]{{#1[{#2}#1]}}
\providecommand{\seminorm}[1]{\seminormtmp{}{#1}}

\begin{document}

\title[Numerics for $p$-Stokes]{A class of space-time discretizations for the stochastic $p$-Stokes system}

%A class of space-time discretizations for the stochastic $p$-Stokes system: Stability; Velocity approximation 

\author{Kim-Ngan Le, J\"{o}rn Wichmann}%
\address[N. Le, J. Wichmann]{School of Mathematics, Monash University, Australia}%
\email{ngan.le@monash.edu,joern.wichmann@monash.edu}%
\thanks{ This work was partially supported by the Australian Government through the Australian Research Council’s Discovery
Projects funding scheme (grant number DP220100937). }

\begin{abstract}
The main objective of the present paper is to construct a new class of space-time discretizations for the stochastic $p$-Stokes system and analyze its stability and convergence properties. 

We derive regularity results for the approximation that are similar to the natural regularity of solutions. One of the key arguments relies on discrete extrapolation that allows us to relate lower moments of discrete maximal processes.

We show that, if the generic spatial discretization is constraint conforming, then the velocity approximation satisfies a best-approximation property in the natural distance. 

Moreover, we present an example such that the resulting velocity approximation converges with rate $1/2$ in time and $1$ in space towards the (unknown) target velocity with respect to the natural distance. The theory is corroborated by numerical experiments.
\end{abstract}

    % MSC

\maketitle

\tableofcontents
\section{Introduction}
The $p$-Stokes system models the evolution of power-law fluids in the regime of laminar flow. It neglects additional difficulties that may arise through convective effects present for general flows. More details on the physical interpretation can be found e.g. in~\cite{MR2182831} and the references therein. Additionally, the stochastic forcing accounts for fluctuations on the microscopic description of the fluid. A derivation of the random fluctuations for the most famous Navier-Stokes equations can be found e.g. in~\cite{MR2050201} (see also~\cite{MR2459085,MR3607454} and the references therein).

In this paper, we consider the stochastic $p$-Stokes system with multiplicative noise: Given some bounded domain $\mathcal{O} \subset \mathbb{R}^n$ and a time horizon $T> 0$, we seek a velocity field $u$ and a pressure~$\pi$ such that the momentum equations and the incompressibility condition 
\begin{subequations} \label{eq:intro_pStokes}
\begin{alignat}{2} \label{eq:intro_pStokes01}
\dd u  - \Div S(\varepsilon u) \dt + \nabla \dd \pi &= G(u) \dd W(t) \quad &&\text{ on } (0,T)\times \mathcal{O}, \\ \label{eq:intro_pStokes02}
\Div u &= 0 \quad &&\text{ on } (0,T)\times \mathcal{O},
\end{alignat}
\end{subequations}
are satisfied, respectively. Here, $\varepsilon u = 2^{-1}( \nabla u + (\nabla u)^T )$ denotes the symmetric velocity gradient and $S$ is the viscous stress and has the form
\begin{align} \label{eq:intro_defS}
S(A) = (\kappa + \abs{A})^{p-2} A \in \mathbb{R}^{n \times n},
\end{align}
for $p \in (1,\infty)$ and $\kappa \geq 0$. Moreover, $G(\cdot) \dd W$ denotes a stochastic forcing driven by a cylindrical Wiener process~$W$ and a suitable noise coefficient~$G$. 
The system couples with incorporating no-slip boundary conditions and an initial state for the velocity field
\begin{subequations}
\begin{alignat}{2}
u &= 0 \quad &&\text{ on } (0,T)\times \partial \mathcal{O}, \\
u(0) &= u_0 \quad &&\text{ on }  \mathcal{O},
\end{alignat}
\end{subequations} 
as well as a mean-free and an initial condition for the pressure
\begin{subequations}
\begin{alignat}{2}
\int_{\mathcal{O}} \pi(x) \dd x &= 0 \quad &&\text{ on } (0,T), \\ \label{intro:PressureInit}
\nabla \pi(0) &= 0 \quad &&\text{ on }  \mathcal{O}.
\end{alignat}
\end{subequations} 

The global well-posedness of~\eqref{eq:intro_pStokes} is studied in a recent paper of the second author~\cite{2022arXiv220902796W}. He shows that in the context of analytically weak solutions, the stochastic pressure -- related to non-divergence free stochastic forces -- enjoys almost $-\frac{1}{2}$ temporal derivatives on a Besov scale. The velocity~$u$ of strong solutions obeys $\frac12$ temporal derivatives in an exponential Nikolskii space. Moreover,  the non-linear symmetric gradient $V(\varepsilon u):= (\kappa + \abs{\varepsilon u})^{(p-2)/2} \varepsilon u$ 
has $\frac12$ temporal derivatives in a Nikolskii space. These temporal regularity results on the natural scale for the stochastic $p$-Stokes system will be used to estimate approximation errors in this paper.

Many authors have contributed towards the development of numerical algorithms for the deterministic counterpart of~\eqref{eq:intro_defS} and related equations, see e.g.~\cite{MR1069652,MR1301740,MR1860723,MR2511695,MR2519599,CP,
BelenkiBerselliDieningRozicka2012,MR3035482,MR3020906,MR3010181,
MR3066804,MR3831242,MR4048427,
MR4111847,MR4092271,MR4344263,MR4244254,
MR4319602,MR4565577,MR4537562,MR4612144}. However, to the best of the authors knowledge there is no numerical algorithm for the approximation of the stochastic $p$-Stokes system driven by a general multiplicative stochastic forcing.

The numerical approximation of stochastic partial differential equations has received much attention in recent years, see e.g.~\cite{MR2139212,MR2465711,MR2578878,
BrezezniakProhlCarelli2013,MR3843574,MR4179715,MR4029845,
MR4261330,MR4298537,MR4243433,
Ondrejt2022,MR4410739,MR4565984,2023arXiv230300411K}. 
A major obstacle for the development of numerical algorithms of stochastic equations is the low temporal regularity of solutions. Generally, solutions will not exceed the temporal regularity of its driving stochastic process. In our case, we enforce the system by a cylindrical Wiener process for which sharp regularity results are available, cf.~\cite{VerHyt08}.

It has been observed that especially the temporal regularity of the pressure deteriorates for stochastic fluid models, cf.~\cite{MR4286261,MR4274685}. This leads to severe difficulties in the error analysis for non-exactly divergence free velocity approximations, since the velocity approximation depends on the pressure approximation. This strongly motivates the use of exactly divergence-free approximations from a theoretical perspective.

\textbf{Our contributions:}
We construct a new class of space-time discretizations for the stochastic $p$-Stokes system and analyze its stability and convergence properties. Our construction relies on the following three steps:
\begin{enumerate}
\item Data approximation; at some point one needs to approximate the hyper parameters of the model. Instead of prescribing an explicit approximation rule, we propose a distance measure of abstract data approximations. In this way a certain flexibility in the construction of the load data is introduced. We show that all data approximations lead to the same speed of convergence.
\item Generic spatial discretization; we use abstract approximate spaces of the natural regularity class of solutions for the spatial discretization. This enables a unified stability theory for a broad class of classical spatial discretizations.
\item Robust temporal discretization; in order to resolve the lack of regularity of solutions we approximate time-averaged values of the unknown solution. This technique has already been used in~\cite{MR4286257} and~\cite{Diening2022}. It allows us to establish optimal convergence rates with minor regularity requirements. 
\end{enumerate}

\textbf{Outline:}
The paper is organized as follows:

Section~\ref{sec:NotaResu} introduces the mathematical framework and summarizes the main results of the paper. 

Section~\ref{sec:FullyDisc} addresses the definition of the abstract space-time discretizations. 

Section~\ref{sec:Stability} contains stability results for the abstract velocity and pressure approximations. 

Section~\ref{sec:ErrorDecomp} deals with the derivation of the error decomposition.

Section~\ref{sec:Convergence} provides an example of a discretization that convergence with optimal rates.

Section~\ref{sec:Numerical-simulation} presents numerical simulations that corroborate the theory.

Appendix~\ref{app:Appendix} contains further details on pressure norms, the relation of~$S$ and~$V$ and discrete extrapolation.

\section{Notations and Main results}\label{sec:NotaResu}
Let $\mathcal{O} \subset \mathbb{R}^n$ be a bounded Lipschitz domain. For a given $T>0$ we denote by $I := (0,T)$ the time interval and write $\mathcal{O}_T := I \times \mathcal{O}$ for the time space cylinder. Moreover, let $\left(\Omega,\mathcal{F}, (\mathcal{F}_t)_{t\in I}, \mathbb{P} \right)$ denote a stochastic basis, i.e., a probability space with a complete and right continuous filtration $(\mathcal{F}_t)_{t\in I}$. We write $f \lesssim g$ for two non-negative quantities $f$ and $g$ if $f$ is bounded by $g$ up to a multiplicative constant. Accordingly we define $\gtrsim$ and $\eqsim$. For $N \in \mathbb{N}$, we abbreviate $[N] := \{ 1, \ldots, N\}$ and $[N_0] := [N] \cup \{0\}$. We denote by $c$ a generic constant which can change its value from line to line. 

\subsection{Function spaces} \label{sec:Function spaces}
As usual, for $1\leq q < \infty$ we denote by $L^q(\mathcal{O})$ the Lebesgue space and $W^{1,q}(\mathcal{O})$ the Sobolev space. Moreover, $W^{1,q}_0(\mathcal{O})$ denotes the Sobolev spaces with zero boundary values. It is the closure of $C^\infty_0(\mathcal{O})$ (smooth functions with compact support) in the $W^{1,q}(\mathcal{O})$-norm. We denote by $W^{-1,q'}(\mathcal{O})$ the dual of $W^{1,q}_0(\mathcal{O})$. The space of mean-value free Lebesgue functions is denoted by $L^q_0(\mathcal{O})$. The space of smooth, compactly supported and divergence-free vector fields is called $C^\infty_{0,\Div}(\mathcal{O})$ and its closure within the $W^{1,p}$-norm is abbreviated by $W^{1,p}_{0,\Div}(\mathcal{O})$. Let $H$ be a Hilbert space. We denote its inner product by $\left( \cdot, \cdot \right)_H$. If the Hilbert space is $L^2(\mathcal{O})$ we simply write $\left( \cdot, \cdot \right)$. We do not distinguish notation between scalar-, vector- and matrix-valued functions.

For a Banach space $\left(X, \norm{\cdot}_X \right)$ let $L^q(I;X)$ be the Bochner space of Bochner-measurable functions $u: I \to X$ satisfying $t \mapsto \norm{u(t)}_X \in L^q(I)$. Moreover, $C^0(\overline{I};X)$ is the space of continuous functions with respect to the norm-topology. We also use $C^{0,\alpha}(\overline{I};X)$ for the space of $\alpha$-H\"{o}lder continuous functions. Given an Orlicz-function $\Phi: [0,\infty] \to [0,\infty]$, i.e. a convex function satisfying $ \lim_{t \to 0} \Phi(t)/t = 0$ and $\lim_{t \to \infty} \Phi(t)/t = \infty$ we define the Luxemburg-norm 
\begin{align*}
\norm{u}_{L^\Phi(I;X)} := \inf \left\{ \lambda > 0 : \int_I \Phi \left( \frac{\norm{u}_X}{\lambda} \right) \ds \leq 1 \right\}.
\end{align*}
The Orlicz space $L^\Phi(I;X)$ is the space of all Bochner-measurable functions with finite Luxemburg-norm. For more details on Orlicz-spaces we refer to \cite{DiHaHaRu}. Given $h \in I$ and $u :I \to X$ we define the difference operator $\tau_h: \set{u: I \to X} \to \set{u: I\cap I - \set{h} \to X} $ via $\tau_h(u) (s) := u(s+h) - u(s)$. The Besov-Orlicz space $B^\alpha_{\Phi,r}(I;X)$ with differentiability $\alpha \in (0,1)$, integrability $\Phi$ and fine index $r \in [1,\infty]$ is defined as the space of Bochner-measurable functions with finite Besov-Orlicz norm $\norm{\cdot}_{B^\alpha_{\Phi,r}(I;X)}$, where
\begin{align} \label{def:Besov01}
\begin{aligned}
\norm{u}_{B^\alpha_{\Phi,r}(I;X)} &:= \norm{u}_{L^{\Phi}(I ;X)} + \seminorm{u}_{B^{\alpha}_{\Phi,r}(I;X)}, \\
\seminorm{u}_{B^\alpha_{\Phi,r}(I;X)} &:= \begin{cases} \left( \int_{I} h^{-r\alpha} \norm{\tau_h u}_{L^\Phi(I\cap I - \set{h};X)}^r \frac{\dd h}{h} \right)^\frac{1}{r} & \text{ if } r \in [1,\infty), \\
\esssup_{h \in I} h^{-\alpha} \norm{\tau_h u}_{L^\Phi(I\cap I - \set{h};X)} & \text{ if } r = \infty.
\end{cases} 
\end{aligned}
\end{align} 
The case $r = \infty$ is commonly called Nikolskii-Orlicz space and abbreviated by $N^{\alpha,\Phi} = B^{\alpha}_{\Phi,\infty}$. When $\Phi(t) = t^p$ we write $B^\alpha_{p,r}(I;X)$ and call it Besov space.
% If $X$ is reflexive, we denote by $B^{-\alpha}_{p',q'}(I;X') = \big( B^{\alpha}_{p,q}(I;X) \big)'$.

Similarly, given a Banach space $\left(Y, \norm{\cdot}_Y \right)$, we define $L^q(\Omega;Y)$ as the Bochner space of Bochner-measurable functions $u: \Omega \to Y$ satisfying $\omega \mapsto \norm{u(\omega)}_Y \in L^q(\Omega)$. The space $L^q_{\mathcal{F}}(\Omega \times I;X)$ denotes the subspace of $X$-valued progressively measurable processes. We abbreviate the notation $L^q_\omega L^r_t L^s_x := L^q(\Omega;L^r(I;L^s(\mathcal{O}))) $ and $L^{q-} := \bigcap_{r< q} L^r$.

Let $\mathfrak{U}$ be a separable Hilbert space and $W$ be an $\mathfrak{U}$-valued cylindrical Wiener process with respect to $(\mathcal{F}_t)$. We denote by $L_2(\mathfrak{U};L^2_x)$ the space of Hilbert-Schmidt operators.

\subsection{Assumptions and Concept of solutions}
We assume that $G$ is progressively measurable  
\begin{align*}
G : \Omega \times I \times L^2_x  \to L_2(\mathfrak{U};L^2_x)
\end{align*}
and satisfies the following:
\begin{enumerate}
\item (sublinear growth) For all $(\omega,t) \in \Omega \times I$ and $v \in L^2_x$
\begin{align} \label{ex:Sublinear}
\norm{G(\omega,t, v)}_{L_2(\mathfrak{U};L^2_x)}^2 \lesssim 1+ \norm{v}_{L^2_x}^2.
\end{align}
\item (Lipschitz) For all $(\omega,t) \in \Omega \times I$ and $v,w \in L^2_x$
\begin{align}\label{ex:Lipschitz}
\norm{G(\omega,t, v) - G(\omega,t, w)}_{L_2(\mathfrak{U};L^2_x)}^2 \lesssim  \norm{v - w}_{L^2_x}^2.
\end{align}
\end{enumerate}

The concept of weak solutions is defined as follows.
\begin{definition} \label{def:weakCouple}
A vector field~$u$ and scalar function~$\pi$ are called weak solution to \eqref{eq:intro_pStokes} if 
\begin{enumerate}
\item $u \in L^2_\omega C_t^0 L^2_{\Div} \cap L^p_\omega L^p_t W^{1,p}_{0,x}$ is $(\mathcal{F}_t)$-adapted,
\item $\pi = \pi^{1} + \pi^{2} \in  L^{p'}_\omega W^{1,p'}_t L^{p'}_{0,x} + L^{2}_\omega C^0_t \big( W^{1,2}_{x} \cap L^2_{0,x}  \big)$ is $(\mathcal{F}_t)$-adapted,
\item $\pi^2$ is a centered $(\mathcal{F}_t)$-martingale,
\item  for all $t \in I$, $\xi \in L^2_x \cap W^{1,p}_{0,x}$ and $\mathbb{P}$-a.s. it holds
\begin{align} \label{eq:WeakSolutionCoupled}
\begin{aligned}
&\int_{\mathcal{O}} \left( u(t) - u_0 - \int_0^t G\big(s, u(s) \big) \dd W(s) \right) \cdot \xi \dd x + \int_{\mathcal{O}}  \int_0^t S( \varepsilon u) \ds :\nabla \xi \dd x\\
&\hspace{4em}  - \int_{\mathcal{O}}  \pi^{1}(t) \Div \xi \dd x + \int_{\mathcal{O}}  \nabla \pi^{2}(t) \cdot \xi \dd x =0.
\end{aligned}
\end{align}
\end{enumerate}
\end{definition}
The existence of a weak solution is settled in~\cite{2022arXiv220902796W} (see also~\cite{Breit2015} for the construction of weak solutions to~\eqref{eq:intro_pStokes} including convection).

\subsection{Summary of the main results}
We shortly summarize the main findings of the paper. The results are three-fold. Let~$u_{h,\tau}$ and~$\pi_{h,\tau}$ denote velocity and pressure approximation, respectively. For the exact definition we refer to Section~\ref{sec:FullyDisc}.

First, we carefully trace the assumptions on the data approximation and its influence on the stability of the algorithm. We propose two scales for the data approximation. On the first scale, we derive uniform estimates
\begin{subequations}
\begin{align} \label{eq:introReg}
u_{h,\tau} &\in L^2_\omega L^\infty_t L^2_x \cap B^{1/2}_{2,\infty} L^2_\omega L^2_x \cap L^p_\omega L^p_t W^{1,p}_{0,x}, \\
\pi_{h,\tau} &\in  L^{p'}_\omega W^{1,p'}_t L^{p'}_x + \big( L^2_\omega L^\infty_t W^{1,2}_x \cap B^{1/2}_{r,\infty} L^2_\omega W^{1,2}_x \big),
\end{align}
\end{subequations}
for any $r \in [2,\infty)$, cf. Theorem~\ref{thm:EnergyBoundsVelocity} and Theorem~\ref{thm:PressureTimeStepping}.

The second scale strengthens the assumptions on the data approximation. This leads to the improved stability result
\begin{subequations}
\begin{align}
u_{h,\tau} &\in L^2_\omega \big( L^\infty_t L^2_x \cap B^{1/2}_{2,\infty} L^2_x \big) \cap L^p_\omega L^p_t W^{1,p}_{0,x}, \\
\pi_{h,\tau} &\in  L^{p'}_\omega W^{1,p'}_t L^{p'}_x +  L^2_\omega \big( L^\infty_t W^{1,2}_x \cap B^{1/2}_{r,\infty} W^{1,2}_x \big),
\end{align}
\end{subequations}
for any $r \in [2,\infty)$, cf. Lemma~\ref{lem:VelocityReg} and Lemma~\ref{lem:PressureStrong}. In other words, the second scale allows us to measure regularity of time increments as a stochastic process.

Second, we derive an error decomposition, cf. Theorem~\ref{thm:ErrorEstimates} and Theorem~\ref{lem:ErrorEstimateBesov}, for the velocity error of exactly divergence-free approximations on the natural scale~\eqref{eq:introReg}, i.e.,

\begin{subequations}\label{eq:introEstFull}
\begin{align}
\begin{aligned} \label{eq:introEst01}
\norm{ u - u_{h,\tau} }_{L^2_\omega L^\infty_t L^2_x} &+ \norm{ V(\varepsilon u) - V(\varepsilon u_{h,\tau}) }_{L^2_\omega L^2_t L^2_x}  \\
&\lesssim C_{\mathrm{init}} + C_{L^\infty} +  C_{\mathrm{best}} + C_{G},
\end{aligned}
\end{align}
and
\begin{align}\label{eq:introEst02}
\norm{u - u_{h,\tau} }_{B^{1/2}_{2,\infty} L^2_\omega L^2_x} \lesssim C_{\mathrm{init}} + C_{B^{1/2}_{2,\infty}}+ C_{\mathrm{best}} + C_{G}  + C_{V}.
\end{align}
\end{subequations}
Here, $C_{\mathrm{init}}$ and $C_{G}$ correspond to data approximation errors; $C_{L^\infty}$ and $C_{B^{1/2}_{2,\infty}}$ are errors induced by the spatial discretization measured on respective time scales; $C_V$ measures temporal oscillation of the non-linear gradient of the target velocity and $C_{\mathrm{best}}$ is the distance of the best-approximation and the target velocity in a non-linear distance. To the best of our knowledge~\eqref{eq:introEst02} is even new for the deterministic $p$-Laplace system.

Third, we present an example for the data approximation and the spatial discretization such that the error terms in~\eqref{eq:introEstFull} can be controlled. The resulting velocity approximation converges with optimal rates towards the target velocity with respect to the natural distance, i.e.,
\begin{align*}
\norm{ u - u_{h,\tau} }_{L^2_\omega L^\infty_t L^2_x} &+ \norm{ V(\varepsilon u) - V(\varepsilon u_{h,\tau}) }_{L^2_\omega L^2_t L^2_x} + \norm{u - u_{h,\tau} }_{B^{1/2}_{2,\infty} L^2_\omega L^2_x} \lesssim h + \sqrt{\tau}.
\end{align*}
The result is presented in Theorem~\ref{thm:Convergence}.

\section{Fully discrete ansatz} \label{sec:FullyDisc}
This section contains the fully discrete abstract discretization. We split it into a generic spatial discretization, a robust time discretization based on the approximation of averaged values and an abstract data approximation. Ultimately, we present a numerical scheme for the approximation of~\eqref{eq:intro_pStokes}.

\subsection{Generic spatial discretization}
Let $h \in (0,1]$ denote the spatial discretization parameter and
\begin{align}
V_h &\subset W^{1,\infty}_x \cap W^{1,1}_{0,x} \quad \text{ and } \quad 
Q_h \subset L^\infty_x \cap L^1_{0,x},
\end{align}
be finite dimensional velocity and pressure approximate spaces, respectively. The subspace of discretely divergence free velocities is denoted by
\begin{align*}
V_{h,\Div} := \left\{ v \in V_h | \, \forall q \in Q_h: \left( q, \Div v \right) = 0 \right\}
\end{align*} 
and its orthogonal complement (with respect to $L^2_x$)
\begin{align*}
V_{h,\Div}^\perp := \left\{v \in V_h | \, \forall w \in V_{h,\Div}: \left( v,w \right) = 0 \right\}.
\end{align*}
 In this way $V_h$ splits into
\begin{align*}
V_h = V_{h,\Div} \oplus V_{h,\Div}^\perp.
\end{align*}
Moreover, we denote by $\Pi_{\Div}$ and $\Pi_{\Div}^\perp$ the $L^2_x$-projections on $V_{h,\Div}$ and $V_{h,\Div}^\perp$, respectively.

\subsection{Robust temporal discretization} \label{sec:time-disc}
Let $N \in \mathbb{N}$. Denote by $\tau = T/(N+1)$ the temporal step size, $t_n = n\tau$ the equidistant grid points and $J_n = [t_n - \tau/2, t_n + \tau/2]$ the intervals of size~$\tau$ centered at~$t_n$. We set $J_0 := [0,\tau/2]$ and denote by $\mean{\cdot}_n = \dashint_{J_n} \cdot \dd t$ the mean integral on the interval $J_n$. 

Formally, the analytic solution satisfies 
\begin{align*}
\mean{u}_1 - u_0 - \dashint_{J_1} \int_0^t \Div S\big(\varepsilon u(s)\big) \dd s \dd t + \nabla \mean{\pi}_1 - \nabla \pi_0 = \dashint_{J_1} \int_0^t G\big(s,u(s)\big) \dd W(s) \dd t
\end{align*}
and
\begin{align*}
\mean{u}_n - \mean{u}_{n-1} - &\dashint_{J_n} \int_{t-\tau}^t \Div S \big(\varepsilon u(s)\big) \dd s \dd t + \nabla \mean{\pi}_n - \nabla \mean{\pi}_{n-1}  \\
&\quad = \dashint_{J_n} \int_{t-\tau}^t G\big(s,u(s) \big) \dd W(s) \dd t.
\end{align*}
Due to Fubini's theorem this is equivalent to 
\begin{align*}
\mean{u}_1 - u_0 - \int a_1(t) \Div S\big(\varepsilon u(t) \big) \dd t + \nabla \mean{\pi}_1  - \nabla \pi_0 =  \int a_1(t) G\big(t,u(t)\big) \dd W(t)
\end{align*}
and
\begin{align}\label{eq:FormalAnalytic}
\begin{aligned}
\mean{u}_{n} - \mean{u}_{n-1} - &\int a_n(t) \Div S\big(\varepsilon u(t) \big) \dd t + \nabla \mean{\pi}_n  - \nabla \mean{\pi}_{n-1}  \\
&\quad =  \int a_n(t) G\big(t,u(t)\big) \dd W(t),
\end{aligned}
\end{align}
where
\begin{subequations} \label{eq:Weights}
\begin{align} \label{eq:WeightInit}
a_1(t) &=  1_{J_0}(t) + \frac{t_1 + \tau/2 - t}{\tau} 1_{J_1}(t), \\\label{eq:WeightN}
a_n(t) &= \frac{t - (t_{n-1} - \tau/2)}{\tau} 1_{J_{n-1}}(t)  + \frac{t_{n} + \tau/2 - t}{\tau} 1_{J_n}(t).
\end{align}
\end{subequations}
Notice that $\int a_n(t) \dd t = \tau$. Here $1_{J}$ denotes the indicator function on~$J$.

Moreover, we define
\begin{subequations}\label{eq:Waver}
\begin{align} \label{eq:WaverInit}
\Delta_1 \mathbb{W} &:=  \mean{W}_{J_1} = \dashint_{J_1} W(t) \dd t  = \dashint_{J_1} \int_0^t \dd W(s) \dd t = \int a_1(t) \dd W(t),\\ \label{eq:WaverN}
\Delta_n \mathbb{W} &:= \mean{W}_{J_n} - \mean{W}_{J_{n-1}} = \dashint_{J_n} W(t) - W(t-\tau) \dd t = \int a_n(t) \dd W(t).
\end{align}
\end{subequations}
Since~$\Delta_n \mathbb{W}$ relies on values of~$W$ up to~$t_n + \tau/2$ a natural choice for a discrete filtration is
\begin{align}
\mathcal{F}^N_n := \mathcal{F}_{t_{n} + \tau/2}.
\end{align}
Additionally, we set $\mathcal{F}_0^N := \mathcal{F}_0$.

The distribution of the random vector~$\Delta \mathbb{W}:= (\Delta_1 \mathbb{W}, \ldots, \Delta_N \mathbb{W})$ can be computed as in~\cite[Lemma~34]{Diening2022}. It is a centered Gaussian random variable with variance
\begin{align*}
\mathbb{E}\left[ \left( \Delta_n \mathbb{W}, \mathfrak{u}_1 \right)_{\mathfrak{U}} \left( \Delta_m \mathbb{W}, \mathfrak{u}_2 \right)_{\mathfrak{U}} \right] = \left( \mathfrak{u}_1, \mathfrak{u}_2 \right)_{\mathfrak{U}} \int a_n(t) a_m(t) \dd t,
\end{align*}
for all $\mathfrak{u}_1, \mathfrak{u}_2 \in \mathfrak{U}$. A simple sampling algorithm can also be found in~\cite[Section~4.3]{Diening2022}. 

The naive evaluation of the local integrals via a one point quadrature rule results in 
\begin{align} \label{eq:QuadratureDet}
\int a_n(t) \Div S\big(\varepsilon u(t)\big) \dd t &\approx \Div S\big(\varepsilon u(r_n) \big) \tau, \\ \label{eq:QuadratureSto}
  \int a_n(t) G\big(t, u(t)\big) \dd W(s) &\approx G\big(r_n, u(r_n)\big) \Delta_n \mathbb{W}.
\end{align}
We want to stress that stochastic integration is highly sensitive with respect to the time instance~$r_n$. This is in sharp contrast to deterministic integration. 

\subsection{Data approximation}
We also need to discretize the noise coefficient in time. Instead of prescribing an explicit discretization rule we allow for a class of possible discretizations. 

\begin{assumption}[Data stability] \label{ass:DataApprox}
We assume there exists a constant $C>0$ (independent of $n$ and $N$) such that for all $n \in \{1, \ldots,N\}$ the following conditions are satisfied:
\begin{enumerate}
\item (measurability) $G_n:\Omega \times L^2_x \to L_2(\mathfrak{U};L^2_x)$ is $\mathcal{F}_{(n-2) \vee 0}^N$-measurable.
\item (sublinear growth) For all $\mathcal{F}_{(n-2) \vee 0}^N$-measurable $v_n \in L^2_{\omega} L^2_x$ 
\begin{align}\label{ass:Sublinear}
\mathbb{E} \left[ \norm{G_n(v_n)}_{L_2(\mathfrak{U};L^2_x)}^2 \right] \leq C(1+ \norm{v_n}_{L^2_\omega L^2_x}^2).
\end{align}
\item (Lipschitz) For all $\mathcal{F}_{(n-2) \vee 0}^N$-measurable $v_n,w_n \in L^2_\omega L^2_x$
\begin{align} \label{ass:Lipschitz}
\mathbb{E} \left[ \norm{G_n(v_n) - G_n(w_n)}_{L_2(\mathfrak{U};L^2_x)}^2 \right] \leq C \norm{v_n-w_n}_{L^2_\omega L^2_x}^2.
\end{align}
\end{enumerate}
\end{assumption}
For a concrete choice of $G_n$ we refer to Example~\ref{ex:GandGn}.

\subsection{The algorithm}
Motivated by the formal derivation~\eqref{eq:FormalAnalytic} and the techniques used in~\cite{MR4286257,Diening2022}, we propose a new algorithm based on time-averaged values of the analytic solution. Time averages relax the regularity requirements for the analytic solution and still allow for optimal convergence results. 

\subsubsection{Initial condition}
We can only prescribe the initial velocity $u_0^h \in V_h$. The initial pressure is defined as the solution $\pi_0^h \in Q_h$ to 
\begin{align} \label{eq:InitPressure01}
\forall\xi \in V_{h,\Div}^{\perp}: \quad \left(\pi_0^h ,\Div \xi \right) = \left( u_0^h, \xi \right)
\end{align}
or equivalently
\begin{align}\label{eq:InitPressure02}
\pi_0^h = \argmin_{q \in Q_h} \norm{\nabla^h q + \Pi_{\Div}^\perp u_0^h}_{L^2_x}.
\end{align}
Here, the discrete gradient~$\nabla^h : L^1_{0,x} \to V_h$ is defined by
\begin{align} \label{eq:DiscreteGrad}
\forall v \in V_h:\quad \left(\nabla^h q, v \right) = - \left(q, \Div v \right).
\end{align}
Note that $\nabla^h Q_h \subset V_{h,\Div}^\perp$.

\subsubsection{Time stepping}
We want to find $(u, \pi) \in \left[ V_h \times Q_h \right]^N$ such that for all $(\xi, q) \in V_h \times Q_h$, $n \in [N]$ and $\mathbb{P}$-a.s.
\begin{subequations} \label{eq:CoupledAlgo}
\begin{align}\label{eq:CoupledAlgo01}
\left(d_n u, \xi \right) + \tau \left( S(\varepsilon u_n), \nabla \xi \right) - \left(d_n \pi, \Div \xi \right) &= \left( G_n(u_{(n-2)\vee 0}) \Delta_n \mathbb{W}, \xi \right), \\ \label{eq:CoupledAlgo02}
\left(\Div u_n, q \right) &= 0,
\end{align}
\end{subequations}
where $d_n v:= v_{n} - v_{n-1}$ and $\Delta_n \mathbb{W}$ is given by~\eqref{eq:Waver}.

\subsubsection{Decoupled time stepping}
Equation~\eqref{eq:CoupledAlgo} decouples into a velocity equation that is independent of the pressure and a complementary equation that allows for the reconstruction of the pressure.

Indeed, we first find $u \in [V_{h,\Div}]^N$ such that for all $\xi \in V_{h,\Div}$, $n \in [N]$ and $\mathbb{P}$-a.s.
\begin{align}\label{eq:velocity001}
\left(d_n u, \xi \right) + \tau \left( S(\varepsilon u_n), \nabla \xi \right) &= \left( G_n(u_{(n-2)\vee 0}) \Delta_n \mathbb{W}, \xi \right).
\end{align}

Afterwards the pressure $\pi \in [Q_h]^N$ is reconstructed such that for all $\xi \in V_{h,\Div}^\perp$, $n \in [N]$ and $\mathbb{P}$-a.s.
\begin{align} \label{eq:pressure001}
\left( d_n \pi, \Div \xi \right) &= \left(d_n u, \xi \right) + \tau \left( S(\varepsilon u_n), \nabla \xi \right)- \left( G_n(u_{(n-2)\vee 0}) \Delta_n \mathbb{W}, \xi \right).
\end{align}

\section{Stability} \label{sec:Stability}
This section contains stability results for the abstract velocity and pressure approximations. 

Since regularity moderates the speed of convergence, it is desirable to find norms as strong as possible that are uniformly bounded in terms of the data. We specifically trace the growth conditions on~$G_n$ and its influence on the a priori estimates. It turns out that the time integrability is a crucial ingredient.  

If~$G_n$ satisfies a growth condition on~$L^2_t$, cf.~\eqref{ass:Sublinear}, we verify an a priori estimates for the velocity  
\begin{align*}
u &\in L^2_\omega L^\infty_t L^2_x \cap L^p_\omega L^p_t W^{1,p}_{0,x} \cap B^{1/2}_{2,\infty} L^2_\omega L^2_x
\end{align*}
and the pressure
\begin{align*}
\pi &\in L^{p'}_\omega W^{1,p'}_t L^{p'}_x + ( L^2_\omega L^\infty_t W^{1,2}_x \cap B^{1/2}_{\infty-,\infty} L^2_\omega W^{1,2}_x ).
\end{align*}
Assuming merely a control on the $L^2_t$ scale, it is only possible to measure time differences of expectations rather than expectations of time differences. 

Strengthening the control to $L^\infty_t$, cf. Assumption~\ref{ass:StrongerCondition}, additionally enables an a priori result for the expectation of time differences
\begin{align*}
u &\in L^2_\omega L^\infty_t L^2_x \cap L^p_\omega L^p_t W^{1,p}_{0,x} \cap  L^2_\omega B^{1/2}_{2,\infty} L^2_x, \\
\pi &\in L^{p'}_\omega W^{1,p'}_t L^{p'}_x + ( L^2_\omega L^\infty_t W^{1,2}_x \cap  L^2_\omega B^{1/2}_{\infty-,\infty} W^{1,2}_x ).
\end{align*}

For the sake of presentation we restrict to the case~$\kappa = 0$. The results carry over to~$\kappa > 0$ with appropriate changes.

\subsection{A first scale of a priori estimates}
Motivated by the stability of continuous stochastic integration, see e.g.~\cite[Theorem~3.2]{MR4116708}, it is possible to derive estimates for discrete stochastic integration. We want to stress again that the temporal regularity of the integrand essentially determines the temporal regularity of the stochastic integral.

If $G_n(v_n) \in L^2_\omega L_2(\mathfrak{U};L^2_x)$ is $\mathcal{F}^N_{(n-2)\vee 0}$-measurable uniformly in $n$, then the associated stochastic integral
\begin{align*}
[N] \ni M \mapsto \sum_{n=1}^M G_n(v_n) \Delta_n \mathbb{W}
\end{align*}
is bounded on $L^2_\omega L^\infty_t L^2_x$. This is one of the main ingredients towards an a priori estimate for solutions to~\eqref{eq:CoupledAlgo}.

We separate the discussion on the velocity and the pressure.

\subsubsection{Velocity estimate on $L^2_\omega L^\infty_t L^2_x \cap L^p_\omega L^p_t W^{1,p}_{0,x} \cap B^{1/2}_{2,\infty} L^2_\omega L^2_x$}

The next theorem addresses well posedness of the time stepping~\eqref{eq:velocity001} and suitable stability results.

\begin{theorem} \label{thm:EnergyBoundsVelocity}
Let $u_0 \in L^2_{\omega} V_{h,\Div}$ be $\mathcal{F}_0$-measurable and Assumption~\ref{ass:DataApprox} be satisfied. Then the following hold true:
\begin{enumerate}
\item (Existence) There exists a random variable $u \in [V_{h,\Div}]^N$ solving~\eqref{eq:velocity001}.
\item (Measurability) $u_n$ is $\mathcal{F}_{n}^N$-measurable.
\item (Stability)
It holds
\begin{align} \label{eq:AprioriFirst}
\begin{aligned}
&\mathbb{E} \left[ \max_{n \in [N]} \norm{u_n}_{L^2_x}^2 + \sum_{n \in [N]} \tau \norm{\varepsilon u_n}_{L^p_x}^p \right]  + \max_{k \in [N]} \sum_{n = k}^N \mathbb{E}\left[ \frac{\norm{ u_n - u_{n-k}}_{L^2_x}^2}{k} \right]  \\
&\hspace{5em} \lesssim \mathbb{E}\left[ \norm{u_0}_{L^2_x}^2 \right] + 1.
\end{aligned}
\end{align}
\end{enumerate}
\end{theorem}
\begin{proof}
The existence and uniqueness of the solution~$u_n$ to~\eqref{eq:velocity001} is a standard result in the theory of monotone operators. Moreover, the $\mathcal{F}_n^N$-measurability of~$u_n$ immediately follows from the measurability of $u_{k}$, $k < n$, $G_n$ and $\Delta_n \mathbb{W}$. For more details we refer to~\cite[Section~3]{MR3607728} where the construction of the solution to a related problem is discussed.

We split the derivation of~\eqref{eq:AprioriFirst} into two parts.

\textbf{Part~A -- $L^2_\omega L^\infty_t L^2_x \cap L^p_\omega L^p_t W^{1,p}_{0,x}$:} 
Within this part we aim to verify
\begin{align} \label{eq:Estimate01}
&\mathbb{E} \left[ \max_{n \in [N]} \norm{u_n}_{L^2_x}^2  +  \sum_{n \in [N]} \tau \norm{\varepsilon u_n}_{L^p_x}^p \right]  \lesssim \mathbb{E}\left[ \norm{u_0}_{L^2_x}^2 \right] + 1.
\end{align}
%Before we start with the verification of~\eqref{eq:Estimate01} we shortly describe the main steps.
%
%We test the local equation~\eqref{eq:velocity001} at time instance $n$ with the solution $u_n$ itself. Afterwards the local contributions are summed up for $n \in [N]$ to obtain a time discrete counter part of the norm in~$L^2_\omega L^\infty_t L^2_x \cap L^p_\omega L^p_t W^{1,p}_{0,x}$. We want to point out that the summation is necessary in order to get a stable and non-trivial estimate with respect to the time discretization parameter~$N$. The non-linear tensor~$S$ is treated by monotonicity arguments and the stochastic term~$G_n(\cdot) \Delta_n \mathbb{W}$ requires probabilistic arguments. One of the main difficulties arises due to the correlation of adjacent~$\Delta_n \mathbb{W}$ which prohibits to apply martingale inequalities immediately. We overcome this difficulty by introducing a martingale decomposition as done in~\cite{Diening2022}.

Choose~$\xi = u_n \in V_{h,\Div}$ in~\eqref{eq:velocity001}, use $(a-b)\cdot a = 2(\abs{a}^2 +\abs{a-b}^2 - \abs{b}^2)$ and $S(A):A = \abs{A}^p$ to obtain
\begin{align*}
&\frac{1}{2} \left( \norm{u_n}_{L^2_x} -\norm{u_{n-1}}_{L^2_x}^2 + \norm{u_n - u_{n-1}}_{L^2_x}^2 \right) + \tau \norm{\varepsilon u_n}_{L^p_x}^p \\
&\quad = \left( G_n( u_{(n-2)\vee 0}) \Delta_n \mathbb{W}, u_n \right).
\end{align*}
Summation over $n \in \{1,\ldots,M\}$ for some $M \leq N$ yields
\begin{align*}
\norm{u_M}_{L^2_x} &+ \sum_{n=1}^M \norm{u_n - u_{n-1}}_{L^2_x}^2 + 2\sum_{n=1}^M \tau \norm{\varepsilon u_n}_{L^p_x}^p \\
&\quad = \norm{u_0}_{L^2_x} + 2 \sum_{n=1}^M \left( G_n( u_{(n-2)\vee 0}) \Delta_n \mathbb{W}, u_n \right).
\end{align*}
Take the maximum over $M \in [N]$ and expectation
\begin{align} \label{eq:EstimateMAX}
\begin{aligned}
&\mathbb{E}\left[ \max_{M \in [N]} \norm{u_M}_{L^2_x}^2 + \sum_{n=1}^N \norm{u_n - u_{n-1}}_{L^2_x}^2 + \sum_{n=1}^N \tau \norm{\varepsilon u_n}_{L^p_x}^p \right] \\
&\hspace{3em} \lesssim \mathbb{E} \left[ \norm{u_0}_{L^2_x}^2 + \max_{M \in [N]} \sum_{n=1}^M \left( G_n( u_{(n-2)\vee 0}) \Delta_n \mathbb{W}, u_n \right) \right].
\end{aligned}
\end{align}
We need to estimate the maximum of a discrete stochastic integral. Usually the Burkholder-Davis-Gundy inequality allows us to bound the maximal process of a martingale in terms of its quadratic variation. However, since $u_n$ is $\mathcal{F}_{n}^N$- measurable and $\Delta_n \mathbb{W}$ relies on $W(t)$ for $t \in J_{n-1} \cup J_n$, cf.~\eqref{eq:WaverN},
\begin{align*}
M \mapsto \sum_{n=1}^M \left( G_n( u_{(n-2)\vee 0}) \Delta_n \mathbb{W}, u_n \right)
\end{align*}
does not define a discrete martingale. We artificially delay the test function so that it is uncorrelated to the Gaussian increment~$\Delta_n \mathbb{W}$,
\begin{align*}
&\mathbb{E}\left[ \max_{M \in [N]} \sum_{n=1}^M \left( G_n( u_{(n-2)\vee 0}) \Delta_n \mathbb{W}, u_n \right) \right] \\
&\leq \mathbb{E}\left[ \max_{M \in [N]} \sum_{n=1}^M \left( G_n( u_{(n-2)\vee 0}) \Delta_n \mathbb{W}, u_n - u_{(n-2) \vee 0} \right) \right] \\
&\quad + \mathbb{E}\left[ \max_{M \in [N]} \sum_{n=1}^M \left( G_n( u_{(n-2)\vee 0}) \Delta_n \mathbb{W}, u_{(n-2) \vee 0} \right) \right] \\
&= \mathrm{I} + \mathrm{II}.
\end{align*}

H\"older's and Young's inequalities together with $u_{n} - u_{(n-2)\vee 0}  = \sum_{l=1}^2 u_{(n-l+1) \vee 0} - u_{(n-l)\vee 0}$ ensure
\begin{align*}
\mathrm{I} & \lesssim \delta \mathbb{E} \left[ \sum_{n=1}^N \norm{u_n - u_{(n-2)\vee 0} }_{L^2_x}^2 \right] + c_\delta \mathbb{E} \left[\sum_{n=1}^N \norm{G_n( u_{(n-2)\vee 0}) \Delta_n \mathbb{W}}_{L^2_x}^2 \right]\\
&\lesssim \delta \mathbb{E} \left[ \sum_{n=1}^N \norm{u_n - u_{n-1} }_{L^2_x}^2 \right] + c_\delta \mathbb{E} \left[\sum_{n=1}^N \norm{G_n( u_{(n-2)\vee 0}) \Delta_n \mathbb{W}}_{L^2_x}^2 \right].
\end{align*}
The first term can be absorbed by the left hand side of~\eqref{eq:EstimateMAX}. The second term is controlled by It\^o's isometry,~$a_n \leq 1$ and~\eqref{ass:Sublinear},
\begin{align} \label{eq:EstimateStochastic001}
\begin{aligned}
&\mathbb{E} \left[\sum_{n=1}^N \norm{G_n( u_{(n-2)\vee 0}) \Delta_n \mathbb{W}}_{L^2_x}^2 \right] \leq  \mathbb{E} \left[ \sum_{n=1}^N  \tau \norm{G_n( u_{(n-2)\vee 0}) }_{L_2(\mathfrak{U};L^2_x)}^2 \right]  \\
&\hspace{4em} \leq C\left( 1+ \mathbb{E}\left[ \sum_{n=1}^N \tau \norm{u_{(n-2)\vee 0}}_{L^2_x}^2 \right] \right) \\
&\hspace{4em} \lesssim \sum_{n=1}^N \tau  \mathbb{E}\left[ \max_{m \in [n-2]} \norm{u_{m}}_{L^2_x}^2 \right] + \mathbb{E}\left[ \norm{u_0}_{L^2_x}^2 \right] + 1.
\end{aligned}
\end{align}

Next, we turn our attention to the discrete stochastic process
\begin{align} \label{eq:DiscreteStochasticIntegral}
Y_M &:= \sum_{n=1}^M \left( G_n( u_{(n-2)\vee 0}) \Delta_n \mathbb{W}, u_{(n-2) \vee 0} \right).
\end{align}
Note that~$Y$ is almost a centred martingale with respect to~$\mathcal{F}^N_n$. Indeed, similar to~\cite[(3.26)]{Diening2022} we can decompose 
\begin{align} \label{eq:LocalDecomposition}
Y_M = \mathbb{E}\left[Y_N \big| \mathcal{F}_M^N \right] -\mathrm{E}_M,
\end{align}
where the compensator is given by
\begin{align} \label{eq:LocalMartingaleError}
\begin{aligned}
\mathrm{E}_M &:=  \left(\int_{ J_{M} } a_{M+1}(t)  G_{M+1}( u_{(M-1)\vee 0}) \dd W(t) , u_{(M-1) \vee 0} \right)1_{\{ M +1\leq N\}}.
\end{aligned}
\end{align}

Using~\eqref{eq:LocalDecomposition} we decompose the maximum of~$Y$,
\begin{align*}
\mathrm{II} &= \mathbb{E}\left[ \max_{M \in [N]} Y_M \right] \\
&\leq \mathbb{E}\left[ \max_{M \in [N]} \abs{\mathrm{E}_M} \right] + \mathbb{E}\left[ \max_{M \in [N]} \mathbb{E}\left[Y_N \big| \mathcal{F}_{M}^N \right] \right]=: \mathrm{II}_1 + \mathrm{II}_2.
\end{align*}

Estimating the maximum by a sum and applying H\"older's and Young's inequalities 
\begin{align} \label{eq:ArtiEstim}
\begin{aligned}
\mathrm{II}_1 &\leq \mathbb{E}\left[ \left( \sum_{n=1}^N \abs{\mathrm{E}_n}^2 \right)^{1/2} \right] \\
&\leq \mathbb{E}\left[ \left( \sum_{n=1}^N \norm{\int_{ J_n } a_{n}(t)  G_{n}( u_{(n-2)\vee 0}) \dd W(t)}_{L^2_x}^2\norm{ u_{(n-2) \vee 0} }_{L^2_x}^2 \right)^{1/2} \right] \\
&\leq \delta \mathbb{E} \left[ \max_{M \in [N]} \norm{u_M}_{L^2_x}^2 \right] + c_\delta \mathbb{E}\left[ \sum_{n=1}^N \norm{\int_{J_{n}} a_n(t)  G_n( u_{(n-2)\vee 0}) \dd W(t)}_{L^2_x}^2 \right].
\end{aligned}
\end{align}

The martingale part is estimated by the Burkholder-Davis-Gundy inequality and~\eqref{eq:LocalDecomposition},
\begin{align*}
\mathrm{II}_2 &\lesssim \mathbb{E} \left[ \left( \sum_{n =1}^N \abs{ \mathbb{E}\left[Y_N \big| \mathcal{F}_{n}^N\right] - \mathbb{E}\left[Y_N \big| \mathcal{F}_{n-1}^N \right]  }^2 \right)^{1/2} \right] \\
&\lesssim \mathbb{E} \left[ \left( \sum_{n=1}^{N} \abs{Y_n - Y_{n-1}}^2 \right)^{1/2} \right] + \mathbb{E} \left[ \left( \sum_{n=1}^N \abs{\mathrm{E}_n}^2 \right)^{1/2} \right].
\end{align*}

Analogously to~\eqref{eq:ArtiEstim} one can show
\begin{align*}
&\mathbb{E} \left[ \left( \sum_{n=1}^{N} \abs{Y_n - Y_{n-1}}^2 \right)^{1/2} \right] \\
&\hspace{3em} \leq \delta \mathbb{E} \left[ \max_{M \in [N]} \norm{u_M}_{L^2_x}^2 \right] + c_\delta \mathbb{E}\left[ \sum_{n=1}^N \norm{\int a_n(t)  G_n( u_{(n-2)\vee 0}) \dd W(t)}_{L^2_x}^2 \right].
\end{align*}

It remains to use the It\^o isometry and the growth assumption~\eqref{ass:Sublinear} (as done in~\eqref{eq:EstimateStochastic001}) to find
\begin{align*}
\mathrm{II} \lesssim \delta \mathbb{E} \left[ \max_{M \in [N]} \norm{u_M}_{L^2_x}^2 \right] +  c_\delta\left( \sum_{n=1}^N \tau  \mathbb{E}\left[ \max_{m \in [n-2]} \norm{u_{m}}_{L^2_x}^2 \right] + \mathbb{E}\left[ \norm{u_0}_{L^2_x}^2 \right] + 1 \right).
\end{align*}

Overall, we have proved
\begin{align} \label{eq:Estimate0011}
\begin{aligned}
&\mathbb{E}\left[ \max_{M \in [N]} \norm{u_M}_{L^2_x}^2 + \sum_{n=1}^N \norm{u_n - u_{n-1}}_{L^2_x}^2 + \sum_{n=1}^N \tau \norm{\varepsilon u_n}_{L^p_x}^p \right] \\
&\hspace{3em} \lesssim  \delta \left( \mathbb{E} \left[ \sum_{n=1}^N \norm{u_n - u_{n-1} }_{L^2_x}^2 \right] + \mathbb{E} \left[ \max_{M \in [N]} \norm{u_M}_{L^2_x}^2 \right] \right) \\
&\hspace{4em} +  c_\delta\left( \sum_{n=1}^N \tau  \mathbb{E}\left[ \max_{m \in [n-2]} \norm{u_{m}}_{L^2_x}^2 \right] + \mathbb{E}\left[ \norm{u_0}_{L^2_x}^2 \right] + 1 \right).
\end{aligned}
\end{align}

Choosing $\delta > 0$ sufficiently small allows for an application of Gronwall's lemma. Thus,
\begin{align} \label{eq:GronwallResult}
\begin{aligned}
&\mathbb{E}\left[ \max_{M \in [N]} \norm{u_M}_{L^2_x}^2 \right] \lesssim \mathbb{E}\left[ \norm{u_0}_{L^2_x}^2 \right] + 1.
\end{aligned}
\end{align} 

The estimate for the symmetric gradient and increments of distance one follows from~\eqref{eq:Estimate0011} and~\eqref{eq:GronwallResult}
\begin{align} \label{eq:GradientEstimate}
\mathbb{E}\left[ \sum_{n=1}^N \norm{u_n - u_{n-1}}_{L^2_x}^2 + \sum_{n=1}^N \tau \norm{\varepsilon u_n}_{L^p_x}^p \right] \lesssim \mathbb{E}\left[ \norm{u_0}_{L^2_x}^2 \right] + 1.
\end{align}
This finishes the proof of~\eqref{eq:Estimate01}.

\textbf{Part~B -- $B^{1/2}_{2,\infty} L^2_\omega L^2_x$:} Within this part we aim to verify
\begin{align}
& \max_{k \in [N]} \sum_{n = k}^N \mathbb{E}\left[ \frac{\norm{ u_n - u_{n-k}}_{L^2_x}^2}{k} \right]  \lesssim \mathbb{E}\left[ \norm{u_0}_{L^2_x}^2 \right] + 1.
\end{align}

Write~\eqref{eq:velocity001} for $n = l$, sum up over $l \in \{n-k+1,\ldots,n\}$ and use the symmetry of $S(\varepsilon u)$
\begin{align} \label{eq:Global}
\begin{aligned}
&\left( u_n - u_{n-k},\xi \right)  \\
&\quad =- \left(\sum_{l=n-k+1}^n \tau S(\varepsilon u_l), \varepsilon \xi \right) +  \left(\sum_{l=n-k+1}^n G_l(u_{(l-2)\vee 0}) \Delta_l \mathbb{W}, \xi \right).
\end{aligned}
\end{align}
Choose~$\xi = u_{n}$, sum up for $n\in\{k,\ldots,N\}$ and take expectation
\begin{align} \label{eq:Temporaer001}
\begin{aligned}
&\mathbb{E}\left[  \sum_{n=N-k+1}^N \norm{u_n}_{L^2_x}^2  + \sum_{n=k}^N \norm{u_n - u_{n-k}}_{L^2_x}^2 \right]  =\mathbb{E} \left[ \sum_{n=0}^{k-1} \norm{u_n}_{L^2_x}^2 \right] \\
&\quad -2 \sum_{n=k}^N \mathbb{E}\left[ \left(\sum_{l=n-k+1}^n \tau S(\varepsilon u_l), \varepsilon u_{n} \right) - \left(\sum_{l=n-k+1}^n G_l(u_{(l-2)\vee 0}) \Delta_l \mathbb{W}, u_n \right) \right].
\end{aligned}
\end{align}

Clearly,
\begin{align*}
\mathbb{E} \left[ \sum_{n=0}^{k-1} \norm{u_n}_{L^2_x}^2 \right] \leq k \mathbb{E}\left[\max_{n \in [N]} \norm{u_n}_{L^2_x}^2 \right].
\end{align*}

The second term in the right hand side of~\eqref{eq:Temporaer001} is bounded by H\"older's and Young's inequalities and discrete Fubini's theorem
\begin{align} \label{eq:Sestimate}
\begin{aligned}
&\mathbb{E}\left[ \sum_{n=k}^N \left(\sum_{l=n-k+1}^n \tau S(\varepsilon u_l), \varepsilon u_{n} \right) \right] \\
&\hspace{3em} \leq \mathbb{E}\left[ \sum_{n=k}^N \sum_{l=n-k+1}^n \tau \norm{S(\varepsilon u_l)}_{L^{p'}_x} \norm{ \varepsilon u_{n} }_{L^p_x} \right] \\
&\hspace{3em}  \leq \mathbb{E}\left[ \sum_{n=k}^N \sum_{l=n-k+1}^n \tau \left( \norm{S(\varepsilon u_l)}_{L^{p'}_x}^{p'} +  \norm{ \varepsilon u_{n} }_{L^p_x}^p \right) \right] \\
&\hspace{3em}  \leq k \mathbb{E}\left[ \sum_{n=1}^N \tau \left( \norm{S(\varepsilon u_n)}_{L^{p'}_x}^{p'} +  \norm{ \varepsilon u_{n} }_{L^p_x}^p \right) \right].
\end{aligned}
\end{align}

Using that $\sum_{l=n-k+1}^n G_l(u_{(l-2)\vee 0}) \Delta_l \mathbb{W}$ is uncorrelated of~$u_{(n-k - 1) \vee 0}$ for all $k \leq n$, we artificially introduce increments in the third term in~\eqref{eq:Temporaer001}
\begin{align*}
&\sum_{n=k}^N  \mathbb{E}\left[\left(\sum_{l=n-k+1}^n G_l(u_{(l-2)\vee 0}) \Delta_l \mathbb{W}, u_n \right) \right] \\
&\quad =\sum_{n=k}^N  \mathbb{E}\left[\left(\sum_{l=n-k+1}^n G_l(u_{(l-2)\vee 0}) \Delta_l \mathbb{W}, u_n - u_{(n-k - 1) \vee 0} \right) \right].
\end{align*}

H\"older's and Young's inequalities show
\begin{align*}
&\sum_{n=k}^N  \mathbb{E}\left[\left(\sum_{l=n-k+1}^n G_l(u_{(l-2)\vee 0}) \Delta_l \mathbb{W}, u_n - u_{(n-k - 1) \vee 0} \right) \right] \\
&= \sum_{n=k}^N  \mathbb{E}\left[\left(\sum_{l=n-k+1}^n G_l(u_{(l-2)\vee 0}) \Delta_l \mathbb{W}, u_n - u_{n-k } \right) \right] \\
&\quad + \sum_{n=k}^N  \mathbb{E}\left[\left(\sum_{l=n-k+1}^n G_l(u_{(l-2)\vee 0}) \Delta_l \mathbb{W}, u_{n-k} - u_{(n-k -1) \vee 0} \right) \right] \\
& \leq \frac{1}{4} \sum_{n=k}^N \mathbb{E} \left[ \norm{u_n - u_{n-k}}_{L^2_x}^2 \right] + \mathbb{E} \left[ \norm{ u_{n-k} - u_{(n-k - 1) \vee 0}}_{L^2_x}^2 \right] \\
&\quad + \sum_{n=k}^N \mathbb{E} \left[ \norm{\sum_{l=n-k+1}^n G_l(u_{(l-2)\vee 0}) \Delta_l \mathbb{W}}_{L^2_x}^2 \right]
\end{align*}

Notice that, due to an index shift,
\begin{align*}
&\sum_{n=k}^N \mathbb{E} \left[ \norm{ u_{n-k} - u_{(n-k - 1) \vee 0}}_{L^2_x}^2 \right] \leq  \sum_{n=1}^N \mathbb{E} \left[ \norm{ u_{n} - u_{n - 1}}_{L^2_x}^2 \right].
\end{align*}

Recall~\eqref{eq:Waver}. Using It\^o isometry,~$a_n^2 \leq 1$ for all $n$, the locality of~$a_n$ and discrete Fubini's theorem,
\begin{align*}
&\sum_{n=k}^N \mathbb{E} \left[ \norm{\sum_{l=n-k+1}^n G_l(u_{(l-2)\vee 0}) \Delta_l \mathbb{W}}_{L^2_x}^2 \right] \\
&\quad = \sum_{n=k}^N \mathbb{E} \left[  \int  \norm{ \sum_{l=n-k+1}^n a_l(t) G_l(u_{(l-2)\vee 0})}_{L_2(\mathfrak{U}; L^2_x)}^2 \dd t \right] \\
&\quad \lesssim \sum_{n=k}^N \mathbb{E} \left[ \sum_{l=n-k+1}^n \tau \norm{ G_l(u_{(l-2)\vee 0})}_{L_2(\mathfrak{U}; L^2_x)}^2 \right] \\
&\quad \leq k \sum_{n=1}^N \tau \mathbb{E} \left[  \norm{ G_n(u_{(n-2)\vee 0})}_{L_2(\mathfrak{U}; L^2_x)}^2 \right].
\end{align*}

Divide~\eqref{eq:Temporaer001} by $k$ and take the maximum over $k \in [N]$
\begin{align*}
&\max_{k \in [N]} \sum_{n = k}^N \mathbb{E}\left[ \frac{\norm{ u_n - u_{n-k}}_{L^2_x}^2}{k} \right] \\
&\hspace{3em} \lesssim \mathbb{E}\left[\max_{n \in [N]} \norm{u_n}_{L^2_x}^2 \right] + \mathbb{E}\left[ \sum_{n=1}^N \tau \left( \norm{S(\varepsilon u_n)}_{L^{p'}_x}^{p'} +  \norm{ \varepsilon u_{n} }_{L^p_x}^p \right) \right] \\
&\hspace{3em} \quad +  \sum_{n=1}^N \mathbb{E} \left[ \norm{ u_{n} - u_{n - 1}}_{L^2_x}^2 \right] + \sum_{n=1}^N \tau \mathbb{E} \left[  \norm{ G_n(u_{(n-2)\vee 0})}_{L_2(\mathfrak{U}; L^2_x)}^2 \right].
\end{align*}
The assertion follows by using~\eqref{ass:Sublinear}, the identity $\norm{S(\varepsilon u)}_{L^{p'}_x}^{p'} = \norm{\varepsilon u}_{L^p_x}^p$ and the a priori bounds~\eqref{eq:GronwallResult} and~\eqref{eq:GradientEstimate}.
\end{proof}

\subsubsection{Pressure reconstruction}
The pressure naturally belongs to the sum of two vector spaces and thus can be written as $\pi = \pi^1 + \pi^2$. In fact,~$\pi^1$ corresponds to the deterministic pressure and~$\pi^2$ is the stochastic pressure. The decomposition allows us to measure regularity of the individual components. This is important, since~$\pi^1$ and $\pi^2$ behave substantially different. While the first one enjoys higher temporal regularity compared to the second one, it lacks spatial regularity. Conversely, the spatial regularity of~$\pi^2$ is superior than the one of~$\pi^1$, but~$\pi^2$ behaves irregular in time. 

We aim to reconstruct the pressure by the formula~\eqref{eq:pressure001}. Since we are interested in a decomposition of the pressure into its individual components, we need to make sure that we can solve the following problem:

Given $\ell \in ( V_{h,\Div}^\perp)^*$ we need to find $q \in Q_h$ such that for all $v \in V_{h,\Div}^\perp$
\begin{align} \label{eq:PressureSolve}
\left(q, \Div v \right) &= \ell(v).
\end{align}
Solvability of~\eqref{eq:PressureSolve} is equivalent by Babuska's lemma to ask for
\begin{itemize}
\item an inf-sup stability
\begin{align} \label{eq:InfSupAbst}
\inf_{q \in Q_h} \sup_{v \in V_{h,\Div}^\perp } \frac{\left(q, \Div v \right)}{\norm{q}_Q \norm{v}_V } =: \beta_h > 0
\end{align}
for appropriate norms on $Q_h$ and $V_h$, respectively.
\item a non-degeneracy condition, i.e., for all $v \in V_{h,\Div}^\perp \backslash \{ 0\}$ exists $q \in Q_h$ such that
\begin{align} \label{eq:NonDeg}
\left(q, \Div v \right) &\neq 0.
\end{align} 
\end{itemize}

Keep in mind that the non-degeneracy condition is always satisfied by the construction of $V_{h,\Div}^\perp$, cf. Lemma~\ref{lem:NonDegenerate}.

It is a design question on the approximate spaces for the velocity and the pressure and its corresponding norms whether the inf-sup constant~$\beta_h$ is uniformly non-degenerate in the discretization parameter~$h$.

Naturally, the pressure (semi-)norms are given by
\begin{subequations} \label{eq:StochPresBoth}
\begin{align} \label{eq:StochPresNorm}
\norm{q}_{Q_{\mathrm{sto}}} &:= \sup_{v \in V_{h,\Div}^\perp } \frac{\left( q,\Div v\right)}{\norm{v}_{L^2_x}}, \\ \label{eq:DetPresNorm}
\norm{q}_{Q_{\mathrm{det}}} &:= \sup_{v \in V_{h,\Div}^\perp } \frac{\left( q,\Div v\right)}{\norm{\nabla v}_{L^p_x}}.
\end{align}
\end{subequations} 
They become norms if we restrict to the equivalence classes~$Q_h^{\sim}:= Q_h \slash \sim$ of the relation
\begin{align}
q_1 \sim q_2 \in Q_h \quad \Leftrightarrow \quad \forall v \in V_{h,\Div}^\perp: (q_1 - q_2, \Div v) =0.
\end{align}
More details can be found in Appendix~\ref{sec:PressureNorms}.

\subsubsection{Pressure estimate on $L^{p'}_\omega W^{1,p'}_t L^{p'}_x + ( L^2_\omega L^\infty_t W^{1,2}_x \cap B^{1/2}_{\infty-,\infty} L^2_\omega W^{1,2}_x )$}

We are ready to state the well posedness result of the pressure reconstruction.

\begin{theorem}\label{thm:PressureTimeStepping}
Let the assumptions of Theorem~\ref{thm:EnergyBoundsVelocity} be satisfied and~$\pi^{\mathrm{init}}$ be given by~\eqref{eq:InitPressure02}. Then the following hold true:
\begin{enumerate}
\item (Existence) There exists a unique $\pi \in [Q_h^{\sim}]^{N+1}$ solving~\eqref{eq:pressure001}.
\item (Measurability) $\pi_n$ is $\mathcal{F}_{n}^N$-measurable.
\item (Decomposition) There exist $(\pi^{\mathrm{det}}, \pi^{\mathrm{sto}}) \in [Q_h^{\sim} \times Q_h^{\sim}]^N$ such that
\begin{subequations}
\begin{align}
\pi_0 &= \pi^{\mathrm{init}}, \\
\forall n \in [N]: \quad \pi_n &= \pi^{\mathrm{init}}+ \pi^{\mathrm{det}}_n + \pi^{\mathrm{sto}}_n.
\end{align}
\end{subequations}

\item (Explicit formula) For all $\xi \in V_{h,\Div}^\perp$, $n \in [N]$ and $\mathbb{P}$-a.s.
\begin{subequations}
\begin{align} \label{eq:PressureA}
\left( \pi^{\mathrm{det}}_n, \Div \xi \right) &=   \left(\sum_{l=1}^n \tau S(\varepsilon u_l), \nabla \xi \right), \\ \label{eq:PressureB}
\left(\pi^{\mathrm{sto}}_n , \Div \xi \right) &=  -\left(\sum_{l=1}^n G_l(u_{(l-2)\vee 0}) \Delta_l \mathbb{W}, \xi \right).
\end{align}
\end{subequations}
\item (Stability) It holds
\begin{subequations}
\begin{align} \label{eq:BesovFirst02}
 \mathbb{E} \left[ \sum_{n=1}^N \tau \norm{\frac{d_n \pi^{\mathrm{det}} }{\tau}}_{Q_{\mathrm{det}}}^{p'} \right] 
 + \mathbb{E}\left[ \max_{n \in [N]} \norm{\pi_n^{\mathrm{sto}}}_{Q_{\mathrm{sto}}}^2 \right] &\lesssim \mathbb{E}\left[ \norm{ u_0}_{L^2_x}^2 \right] + 1,
\end{align}
and for all $r \in [2,\infty)$ 
\begin{align} \label{eq:BesovFirst01}
 \left( \max_{k \in [N]}\sum_{n=k}^N\tau \left( \frac{ \mathbb{E}\left[ \norm{\pi_n^{\mathrm{sto}} - \pi_{n-k}^{\mathrm{sto}}}_{Q_{\mathrm{sto}}}^2 \right] }{\tau k} \right)^{r/2} \right)^{2/r}  &\lesssim \mathbb{E}\left[ \norm{ u_0}_{L^2_x}^2 \right] + 1.
\end{align}
\end{subequations}

\end{enumerate}
\end{theorem}

\begin{proof}
The existence and uniqueness of $\pi^{\mathrm{det}}_n$ and $\pi^{\mathrm{sto}}_n$ that solve~\eqref{eq:PressureA} and~\eqref{eq:PressureB}, respectively, follows immediately from Babuska's lemma. The $\mathcal{F}_{n}^N$-measurability follows from~\eqref{eq:PressureA} and~\eqref{eq:PressureB} and the $\mathcal{F}_{n}^N$-measurability of $u_l$, $G_l$ and $\Delta_l \mathbb{W}$, $l \leq n$.

Let $\pi^{\mathrm{det}}_0 := 0$ and~$n \in [N]$. Subtracting~\eqref{eq:PressureA} for two consecutive time instances yields
\begin{align*}
\left(d_n \pi^{\mathrm{det}}, \Div \xi   \right) &= \left( \pi^{\mathrm{det}}_n - \pi^{\mathrm{det}}_{n-1}, \Div \xi   \right) =\left( \tau S(\varepsilon u_n), \nabla \xi   \right).
\end{align*}

By H\"older's inequality and definition of the deterministic pressure norm~\eqref{eq:DetPresNorm}
\begin{align*}
\norm{\frac{d_n \pi^{\mathrm{det}}}{\tau}}_{Q_{\mathrm{det}}} \leq \norm{S(\varepsilon u_n)}_{L^{p'}_x}.
\end{align*}

Finally, take the $p'$-th power, multiply by~$\tau$, sum up for $n\in [N]$ and take expectation
\begin{align*}
\mathbb{E} \left[ \sum_{n=1}^N \tau \norm{\frac{d_n \pi^{\mathrm{det}} }{\tau}}_{Q_{\mathrm{det}}}^{p'} \right] \leq \mathbb{E} \left[ \sum_{n=1}^N \tau \norm{S(\varepsilon u_n)}_{L^{p'}_x}^{p'} \right].
\end{align*}

Since $\norm{S(\varepsilon u)}_{L^{p'}_x}^{p'} = \norm{\varepsilon u}_{L^p_x}^p$ we can use the a priori bound~\eqref{eq:AprioriFirst} to conclude
\begin{align*}
\mathbb{E} \left[ \sum_{n=1}^N \tau \norm{S(\varepsilon u_n)}_{L^{p'}_x}^{p'} \right] \lesssim \mathbb{E}\left[ \norm{u_0}_{L^2_x}^2 \right] + 1.
\end{align*}

Similarly to~\eqref{eq:InitialPressureEstimate} and using $\norm{\Pi_{\Div}^\perp}_{ \mathcal{L}( L^2_x; L^2_x)} \leq 1$, one derives
\begin{align*}
\norm{\pi_{n}^{\mathrm{sto}}}_{Q_{\mathrm{sto}}}^2 &= \norm{\Pi_{\Div}^\perp \left( \sum_{l=1}^n G_l(u_{(l-2)\vee 0}) \Delta_l \mathbb{W} \right) }_{L^2_x}^2\\
&\leq \norm{\sum_{l=1}^n G_l(u_{(l-2)\vee 0}) \Delta_l \mathbb{W} }_{L^2_x}^2.
\end{align*}
Take the maximum over $n \in [N]$ and expectation
\begin{align*}
\mathbb{E}\left[ \max_{n\in [N]} \norm{\pi_{n}^{\mathrm{sto}}}_{Q_{\mathrm{sto}}}^2 \right] \leq \mathbb{E}\left[ \max_{n\in [N]} \norm{\sum_{l=1}^n G_l(u_{(l-2)\vee 0}) \Delta_l \mathbb{W} }_{L^2_x}^2 \right]
\end{align*}
In order to estimate the maximal process we use a similar perturbation argument as in~\eqref{eq:LocalDecomposition}. Let us introduce the operator $\overline{G}:\Omega \times I \to L_2(\mathfrak{U};L^2_x)$ 
\begin{align} \label{eq:DefinitionGBar}
\overline{G}(\omega,t) = \left\{\mathfrak{u} \mapsto \sum_{n=1}^N a_n(t) G_n\big(\omega, u_{(n-2)\vee 0}(\omega) \big)\mathfrak{u} \right\}.
\end{align}
Note that $\overline{G}$ is progressively measurable with respect to $(\mathcal{F}_t)$. We denote the stochastic integral of $\overline{G}$ by 
\begin{align} \label{eq:StochasticIntegralGlue}
\mathcal{I}_{\overline{G}}(t) := \int_0^t \overline{G}(s) \dd W(s).
\end{align}
Now we can decompose, using the locality of the weights $a_l$, cf.~\eqref{eq:Weights} and~\eqref{eq:Waver}, 
\begin{align}\label{eq:ContinuousStochasticIntegral}
\sum_{l=1}^{n} G_l(u_{(l-2)\vee 0}) \Delta_l \mathbb{W} &= \mathcal{I}_{\overline{G}}(t) - \overline{\mathrm{E}}(t),
\end{align}
where the error term is given by, for $t \in J_n$,
\begin{align} \label{eq:CompensatorIntegral}
\begin{aligned}
\overline{\mathrm{E}}(t) &= -\int_{t}^{t_{n + 1/2}} a_n(s) G_n(u_{(n-2)\vee 0}) \dd W(s) \\
&\quad + \left( \int_{ t_{n-1/2}}^{t} a_{n+1}(s)  G_{n+1}( u_{(n-1)\vee 0}) \dd W(s) \right) 1_{\{n+1 \leq N\}}.
\end{aligned}
\end{align}

Thus,
\begin{align*}
\mathbb{E}\left[ \max_{n \in [N]} \norm{\pi_{n}^{\mathrm{sto}}}_{Q_{\mathrm{sto}}}^2 \right] \lesssim \mathbb{E}\left[ \max_{n \in [N]} \norm{\mathcal{I}_{\overline{G}}(t_{n-1/2}) }_{L^2_x}^2 \right] + \mathbb{E}\left[ \max_{n \in [N]} \norm{\overline{E}_n(t_{n-1/2}) }_{L^2_x}^2 \right].
\end{align*}

Using the Burkholder-Davis-Gundy inequality, the temporal locality of~$a_n$ and $a_n^2 \leq 1$, cf.~\eqref{eq:Weights}, one concludes
\begin{align*}
 \mathbb{E}\left[ \max_{n \in [N]} \norm{\mathcal{I}_{\overline{G}}(t_{n-1/2}) }_{L^2_x}^2 \right] &\leq  \mathbb{E}\left[ \sup_{t\in (0,T-\tau/2)} \norm{\mathcal{I}_{\overline{G}}(t) }_{L^2_x}^2 \right] \\
 &\lesssim \mathbb{E}\left[ \int_0^{T-\tau/2} \norm{\overline{G}(t) }_{L_2(\mathfrak{U}; L^2_x)}^2 \dd t \right] \\
 &= \mathbb{E}\left[ \sum_{n=1}^{N} \int_{J_n} \norm{\sum_{l=1}^N a_{l}(t) G_{l}(u_{(l -2) \vee 0})}_{L_2(\mathfrak{U}; L^2_x)}^2 \dd t \right] \\
 &\lesssim \mathbb{E}\left[ \sum_{n=1}^N \tau \norm{G_{n}(u_{(n -2) \vee 0})}_{L_2(\mathfrak{U}; L^2_x)}^2 \right].
\end{align*}

Similarly,
\begin{align*}
\mathbb{E}\left[ \max_{n \in [N]} \norm{\overline{E}_n(t_{n-1/2}) }_{L^2_x}^2 \right] &\leq \mathbb{E}\left[ \sum_{n =1}^{N} \norm{\overline{E}_n(t_{n-1/2}) }_{L^2_x}^2 \right] \\
&= \mathbb{E}\left[ \sum_{n =1}^{N} \norm{\int_{J_n } a_n(s) G_n(u_{(n-2)\vee 0}) \dd W(s)}_{L^2_x}^2 \right] \\
&\leq \mathbb{E}\left[ \sum_{n=1}^{N} \tau \norm{G_{n}(u_{(n -2) \vee 0})}_{L_2(\mathfrak{U}; L^2_x)}^2 \right].
\end{align*}
The estimate~\eqref{eq:BesovFirst02} follows by using~\eqref{ass:Sublinear} and the a priori estimate~\eqref{eq:AprioriFirst}.

Finally, it remains to verify~\eqref{eq:BesovFirst01}. Analogously to~\eqref{eq:InitialPressureEstimate}
\begin{align*}
\norm{\pi_n^{\mathrm{sto}} - \pi_{n-k}^{\mathrm{sto}}}_{Q_{\mathrm{sto}}}^2 \leq \norm{\sum_{l= n-k+1}^n G_l( u_{(l-2)\vee 0}) \Delta_l \mathbb{W}}_{L^2_x}^2.
\end{align*}
Take expectation, use the It\^o isometry and the temporal locality of~$a_l$,
\begin{align*}
\mathbb{E}\left[\norm{\pi_n^{\mathrm{sto}} - \pi_{n-k}^{\mathrm{sto}}}_{Q_{\mathrm{sto}}}^2 \right] &\lesssim \sum_{l=n-k+1}^n \tau \mathbb{E}\left[ \norm{ G_l( u_{(l-2)\vee 0})}_{L_2(\mathfrak{U};L^2_x)}^2 \right]\\
&\leq \max_{l \in [N]} \mathbb{E}\left[ \norm{ G_l( u_{(l-2)\vee 0})}_{L_2(\mathfrak{U};L^2_x)}^2 \right] \tau k.
\end{align*}
Therefore, the claim follows by~\eqref{ass:Sublinear} and the a priori estimate~\eqref{eq:AprioriFirst}.
\end{proof}

\begin{remark}
The pressure components~$\pi^{\mathrm{init}}$, $\pi^{\mathrm{det}}$ and $\pi^{\mathrm{sto}}$ depend on the choice $V_{h,\Div}^\perp$. In fact, we could have chosen a different subspace $V_{h}^\circ \subset V_h \backslash V_{h,\Div}$. However, we favoured~$V_{h,\Div}^\perp$ since it simplifies some computations. 
\end{remark}

\subsection{A second scale of a priori estimates} \label{sec:SecondScale}
Strengthening the condition on the data approximation allows us to derive stronger stability results for velocity and pressure. 

\begin{assumption} \label{ass:StrongerCondition}
We assume there exists $C \geq 1$ (independent of $\tau$ and $h$) such that for all~$\mathcal{F}_{(n-2)\vee 0}^N$-measurable $v_{n} \in L^2_{\omega} L^2_x \cap L^p_{\omega} W^{1,p}_{0,x}$,

\begin{align} \label{eq:AssStrongerCondition}
\begin{aligned}
&\mathbb{E}\left[\max_{n \in [N]} \norm{G_n(v_n)}_{L_2(\mathfrak{U};L^2_x)}^2 \right] \\
&\hspace{3em} \leq C\left(\mathbb{E}\left[ \max_{n\in [N]} \norm{v_n}_{L^2_x}^2 + \sum_{n=3}^N \tau \norm{\varepsilon v_n}_{L^p_x}^p \right] + 1 \right).
\end{aligned}
\end{align}
\end{assumption}

\begin{remark}
Notice that~\eqref{ass:Sublinear} immediately implies
\begin{align*}
\max_{n \in [N]} \mathbb{E}\left[\norm{G_n(v_n)}_{L_2(\mathfrak{U};L^2_x)}^2 \right] \leq C \mathbb{E}\left[  \max_{n\in [N]} \norm{v_n}_{L^2_x}^2 \right].
\end{align*}
However, the maximum can not be taken prior to the expectation. Assumption~\ref{ass:StrongerCondition} enables a control on the maximal process of the noise coefficient in terms of the maximal process and the gradient of its input. 

We exclude the initial steps in the gradient contribution, since we do no want to prescribe gradient regularity at initial time. Asymptotically, the influence of this exclusion vanishes, i.e.,
\begin{align}
\mathbb{E}\left[ \sup_{t} \norm{G(t,v(t))}_{L_2(\mathfrak{U};L^2_x)}^2 \right] \lesssim \mathbb{E}\left[ \sup_t \norm{v(t)}_{L^2_x}^2 + \int \norm{\varepsilon v}_{L^p_x}^p \dd t  \right] + 1,
\end{align}
for all $(\mathcal{F}_t)$-adapted processes~$v \in L^2_\omega C_t L^2_x \cap L^p_\omega L^p_t W^{1,p}_{0,x}$.
\end{remark}

\subsubsection{Velocity estimate on $L^2_\omega B^{1/2}_{2,\infty} L^2_x$}
The next lemma establishes an estimate for the expectation of time differences.

\begin{lemma} \label{lem:VelocityReg}
Let the assumptions of Theorem~\ref{thm:EnergyBoundsVelocity} be satisfied. Additionally, let Assumption~\ref{ass:StrongerCondition} be true. Then it holds
\begin{align}
\mathbb{E} \left[ \max_{k \in [N]} \sum_{n = k}^N \frac{\norm{ u_n - u_{n-k}}_{L^2_x}^2}{k} \right] \lesssim  \mathbb{E}\left[ \norm{u_0}_{L^2_x}^2 \right] + 1.
\end{align}
\end{lemma}

\begin{proof}
First of all we observe that
\begin{align*}
&\mathbb{E} \left[ \max_{k \in [N]} \sum_{n = k}^N \frac{\norm{ u_n - u_{n-k}}_{L^2_x}^2}{k} \right] \\
&\hspace{3em} \leq \mathbb{E} \left[ \max_{k \in [N]} \norm{ u_k - u_{0}}_{L^2_x}^2 \right] +\mathbb{E}\left[  \max_{k \in [N]}\sum_{n = k+1}^N \frac{\norm{ u_n - u_{n-k}}_{L^2_x}^2}{k} \right].
\end{align*}
The first term is already controlled by~\eqref{eq:AprioriFirst}. Thus, it remains to establish an estimate for the second term.

We start by choosing $\xi = u_n - u_{n-k}$ in~\eqref{eq:Global} for $n > k$. Here we specifically exclude the case $n=k$, since otherwise gradient regularity of the initial condition needs to be required.

Now, H\"older's inequality shows
\begin{align*}
\norm{u_n - u_{n-k}}_{L^2_x}^2 &\leq \norm{\sum_{l=n-k+1}^n \tau S(\varepsilon u_l)}_{L^{p'}_x}\norm{ \varepsilon (u_n - u_{n-k})}_{L^p_x} \\
&\quad + \norm{\sum_{l=n-k+1}^n G_l(u_{(l-2)\vee 0}) \Delta_l \mathbb{W}}_{L^2_x} \norm{u_n - u_{n-k}}_{L^2_x}.
\end{align*}
Young's and Jensen's inequalities allow us to verify
\begin{align*}
\norm{u_n - u_{n-k}}_{L^2_x}^2 &\leq 2\sum_{l=n-k+1}^n \tau \norm{S(\varepsilon u_l)}_{L^{p'}_x}\norm{ \varepsilon (u_n - u_{n-k})}_{L^p_x} \\
&\quad + \norm{\sum_{l=n-k+1}^n G_l(u_{(l-2)\vee 0}) \Delta_l \mathbb{W}}_{L^2_x}^2.
\end{align*}

Divide by $k$, sum up for $n\in \{k+1,\ldots,N\}$, take maximum over $k \in [N]$ and expectation
\begin{align*}
&\mathbb{E}\left[ \max_{k\in[N]}  \sum_{n=k+1}^N  \frac{\norm{u_n - u_{n-k}}_{L^2_x}^2 }{k} \right] \\
&\quad \lesssim \mathbb{E}\left[ \max_{k\in[N]}  \sum_{n=k+1}^N  \frac{\sum_{l=n-k+1}^n \tau \norm{S(\varepsilon u_l)}_{L^{p'}_x}\norm{ \varepsilon (u_n - u_{n-k})}_{L^p_x} }{ k}  \right]\\
&\quad \quad + \mathbb{E}\left[ \max_{k\in[N]} \sum_{n=k+1}^N  \frac{\norm{\sum_{l=n-k+1}^n G_l(u_{(l-2)\vee 0}) \Delta_l \mathbb{W}}_{L^2_x}^2 }{k}  \right]\\
&\quad =: \mathrm{R}_1 + \mathrm{R}_2.
\end{align*} 
The term~$\mathrm{R}_1$ can be treated as in~\eqref{eq:Sestimate}.

The stochastic term needs a more sophisticated analysis. Recall the decomposition~\eqref{eq:ContinuousStochasticIntegral}.
This allows us to rewrite, for arbitrary $t \in J_n$ and $s \in J_{n-k}$,
\begin{align} \label{eq:DecompFlex}
\begin{aligned}
\sum_{l=n-k+1}^n G_l(u_{(l-2)\vee 0}) \Delta_l \mathbb{W} &= \sum_{l=1}^n G_l(u_{(l-2)\vee 0}) \Delta_l \mathbb{W} - \sum_{l=1}^{n-k} G_l(u_{(l-2)\vee 0}) \Delta_l \mathbb{W} \\
&= \mathcal{I}_{\overline{G}}(t)  - \overline{\mathrm{E}}(t) - \left( \mathcal{I}_{\overline{G}}(s) - \overline{\mathrm{E}}(s) \right).
\end{aligned}
\end{align}
In particular, if $t \in J_n$ then $s = t - k \tau \in J_{n-k}$. Therefore,
\begin{align*}
\mathrm{R}_2 &= \mathbb{E}\left[ \max_{k\in[N]} \sum_{n=k+1}^N \dashint_{J_n}  \frac{\norm{\mathcal{I}_{\overline{G}}(t) - \mathcal{I}_{\overline{G}}(t - k\tau) - \overline{\mathrm{E}}(t) + \overline{\mathrm{E}}(t - k\tau)}_{L^2_x}^2 }{k} \dd t \right] \\
&\lesssim \mathbb{E}\left[ \max_{k\in[N]} \sum_{n=k+1}^N \dashint_{J_n}  \frac{\norm{\mathcal{I}_{\overline{G}}(t) - \mathcal{I}_{\overline{G}}(t - k\tau)}_{L^2_x}^2 }{k} \dd t\right] + \mathbb{E}\left[  \sum_{n=1}^N \dashint_{J_n} \norm{\overline{\mathrm{E}}(t)}_{L^2_x}^2 \dd t \right]\\
&=: \mathrm{R}_2^{a} +  \mathrm{R}_2^{b}.
\end{align*}

Due to~\cite[Theorem~3.2(ii)]{MR4116708} together with an extrapolation argument, cf.~\cite[Remark~3.4]{MR4116708}
\begin{align*}
\mathrm{R}_2^a \leq \mathbb{E}\left[ \sup_{h}  \int_{h}^T \frac{\norm{\mathcal{I}_{\overline{G}}(t ) - \mathcal{I}_{\overline{G}}(t - h)}_{L^2_x}^2 }{h}  \dd t \right] \lesssim \mathbb{E}\left[ \norm{\overline{G}}_{L^\infty_t L_2(\mathfrak{U};L^2_x)}^2 \right].
\end{align*}
Additionally, the continuous operator $\overline{G}$ can be bounded by the discrete building blocks. Indeed, using the definition~\eqref{eq:DefinitionGBar} and locality of $a_l$, 
\begin{align*}
\norm{\overline{G}(t)}_{L_2(\mathfrak{U};L^2_x)}^2 &=\sum_{n=1}^N \sum_{m=1}^N  a_n(t) a_m(t) \left(G_n(u_{(n-2)\vee 0}), G_m(u_{(m-2)\vee 0}) \right)_{L_2(\mathfrak{U};L^2_x)} \\
&\lesssim \sum_{n=1}^N a_n^2(t)\norm{G_n(u_{(n-2)\vee 0})}_{L_2(\mathfrak{U};L^2_x)}^2 \\
&\lesssim \max_{n \in [N]} \norm{G_n(u_{(n-2)\vee 0})}_{L_2(\mathfrak{U};L^2_x)}^2 \sum_{n=1}^N a_n^2(t).
\end{align*}
Since $\sup_{t \in I}\sum_{n=1}^N a_n^2(t) \leq 1$ we find
\begin{align}\label{eq:BoundContinuousDiscrete}
\mathbb{E}\left[ \norm{\overline{G}}_{L^\infty_t L_2(\mathfrak{U};L^2_x)}^2 \right] \lesssim \mathbb{E}\left[ \max_{n \in [N]} \norm{G_n(u_{(n-2)\vee 0})}_{L_2(\mathfrak{U};L^2_x)}^2 \right].
\end{align}

The It\^o isometry implies, for $t \in J_n$,
\begin{align*}
&\mathbb{E}\left[ \norm{\overline{\mathrm{E}}(t)}_{L^2_x}^2  \right] = \int_{t}^{t_{n + 1/2}} a_n^2(s) \mathbb{E}\left[ \norm{G_n(u_{(n-2)\vee 0})}_{L_2(\mathfrak{U};L^2_x)}^2 \right] \dd s \\
&\quad + \left( \int_{ t_{n-1/2}}^{t} a_{n+1}^2(s)  \mathbb{E}\left[ \norm{G_{n+1}( u_{(n-1)\vee 0})}_{L_2(\mathfrak{U};L^2_x)}^2 \right] \dd s \right) 1_{\{n+1 \leq N\}}.
\end{align*}
Thus,
\begin{align*}
\sup_{t \in J_n} \mathbb{E}\left[ \norm{\overline{\mathrm{E}}(t)}_{L^2_x}^2  \right] \leq 2 \max_{n \in [N]} \mathbb{E}\left[ \norm{G_n(u_{(n-2)\vee 0})}_{L_2(\mathfrak{U};L^2_x)}^2 \right] \tau.
\end{align*}
This allows us to bound the perturbation
\begin{align*}
\mathrm{R}_2^b &\leq  \sum_{n=1}^N \sup_{t \in J_n}\mathbb{E}\left[  \norm{\overline{\mathrm{E}}(t)}_{L^2_x}^2  \right] \\
&\leq 2 T \max_{n \in [N]} \mathbb{E}\left[ \norm{G_n(u_{(n-2)\vee 0})}_{L_2(\mathfrak{U};L^2_x)}^2 \right].
\end{align*}

Collecting the estimates for~$\mathrm{R}_2^a$ and~$\mathrm{R}_2^b$, we have verified
\begin{align*}
\mathrm{R}_2 \lesssim \mathbb{E}\left[ \max_{n \in [N]} \norm{G_n(u_{(n-2)\vee 0})}_{L_2(\mathfrak{U};L^2_x)}^2 \right].
\end{align*}

Lastly, we use~\eqref{eq:AssStrongerCondition} with $v_n = u_{(n-2)\vee 0}$ to obtain
\begin{align} \label{eq:GtoU}
\begin{aligned}
&\mathbb{E}\left[ \max_{n \in [N]} \norm{G_n(u_{(n-2)\vee 0})}_{L_2(\mathfrak{U};L^2_x)}^2 \right] \\
&\hspace{3em} \lesssim \norm{u_0}_{L^2_\omega L^2_x}^2 +\left(\mathbb{E}\left[ \max_{n\in [N]} \norm{u_n}_{L^2_x}^2 + \sum_{n=1}^N \tau \norm{\varepsilon u_n}_{L^p_x}^p \right] + 1 \right).
\end{aligned}
\end{align}
The assertion follows by using~\eqref{eq:AprioriFirst}.
\end{proof}

\subsubsection{Pressure estimate on $L^2_\omega B^{1/2}_{\infty-,\infty} W^{1,2}_x$} \label{sec:PressureStrong}
Also the stochastic pressure enjoys regularity for time differences.

\begin{lemma} \label{lem:PressureStrong} Let the assumptions of Theorem~\ref{thm:EnergyBoundsVelocity} be satisfied and~$\pi^{\mathrm{sto}}$ be given by Theorem~\ref{thm:PressureTimeStepping}. Additionally, let Assumption~\ref{ass:StrongerCondition} be true. Then for all $r\in [2,\infty)$ there exists $C_r \geq 1$ such that it holds
\begin{align} \label{eq:EstimateStochasticPressure}
\mathbb{E}\left[ \left( \max_{k \in [N]} \sum_{n=k}^N \tau \norm{\frac{\pi^{\mathrm{sto}}_{n} - \pi^{\mathrm{sto}}_{n-k}}{\sqrt{\tau k}}}_{Q_{\mathrm{sto}}}^r  \right)^{2/r} \right]  \lesssim C_r \left( \mathbb{E}\left[ \norm{ u_0}_{L^2_x}^2 \right] + 1 \right).
\end{align}
\end{lemma}

\begin{proof}
We will only discuss the case~$r > 2$. The case~$r=2$ follows along the lines of Lemma~\ref{lem:VelocityReg}. Let $\pi_0^{\mathrm{sto}} := 0$. We split
\begin{align*}
&\mathbb{E}\left[ \left( \max_{k \in [N]} \sum_{n=k}^N \tau \norm{\frac{\pi^{\mathrm{sto}}_{n} - \pi^{\mathrm{sto}}_{n-k}}{\sqrt{\tau k}}}_{Q_{\mathrm{sto}}}^r  \right)^{2/r} \right] \leq \mathbb{E}\left[ \left( \max_{k \in [N]}\tau \norm{\frac{\pi^{\mathrm{sto}}_{k}}{\sqrt{\tau k}}}_{Q_{\mathrm{sto}}}^r  \right)^{2/r} \right]  \\
&\hspace{3em}  + \mathbb{E}\left[ \left( \max_{k \in [N]} \sum_{n=k+1}^N \tau \norm{\frac{\pi^{\mathrm{sto}}_{n} - \pi^{\mathrm{sto}}_{n-k}}{\sqrt{\tau k}}}_{Q_{\mathrm{sto}}}^r  \right)^{2/r} \right]\\
&\hspace{2em} =: \mathrm{J}_1 + \mathrm{J}_2.
\end{align*}

Due to the explicit expression~\eqref{eq:PressureB} and the representation~\eqref{eq:ContinuousStochasticIntegral},
\begin{align*}
\mathrm{J}_1 &\leq \mathbb{E}\left[ \left( \max_{k \in [N]}\tau \norm{\frac{  \sum_{l=1}^k G_l(u_{(l-2)\vee 0}) \Delta_l \mathbb{W}   }{\sqrt{\tau k}}}_{L^2_x}^r  \right)^{2/r} \right] \\
&=  \mathbb{E}\left[ \left( \max_{k \in [N]} \int_{J_k} \norm{\frac{  \mathcal{I}_{\overline{G}}(t) - \overline{E}(t)  }{\sqrt{\tau k}}}_{L^2_x}^r \dd t \right)^{2/r} \right] \\
&\lesssim \mathbb{E}\left[ \left( \max_{k \in [N]} \int_{J_k} \norm{\frac{  \mathcal{I}_{\overline{G}}(t)}{\sqrt{\tau k}}}_{L^2_x}^r \dd t  \right)^{2/r} \right] \\
&\hspace{2em} + \mathbb{E}\left[ \left( \max_{k \in [N]} \int_{J_k} \norm{\frac{ \overline{E}(t)  }{\sqrt{\tau k}}}_{L^2_x}^r \dd t \right)^{2/r} \right]\\
&=: \mathrm{J}_1^a + \mathrm{J}_1^b.
\end{align*}
Similarly, we can decompose
\begin{align*}
\mathrm{J}_2&\lesssim \mathbb{E}\left[  \max_{k \in [N]} \left( \sum_{n=k+1}^N \int_{J_n} \norm{\frac{\mathcal{I}_{\overline{G}}(t) - \mathcal{I}_{\overline{G}}(t - k\tau)}{\sqrt{k\tau }}}_{L^2_x}^r \dd t \right)^{2/r}   \right] \\
&\quad + \mathbb{E}\left[ \max_{k \in [N]} \left( \sum_{n=k+1}^N \int_{J_n} \norm{\frac{\overline{E}(t) - \overline{E}(t-k\tau)}{\sqrt{\tau k}}}_{L^2_x}^r  \dd t\right)^{2/r} \right]\\
&=: \mathrm{J}_2^a +\mathrm{J}_2^b.
\end{align*}

First we discuss the terms~$\mathrm{J}^a$. They correspond to the continuous operator~$\mathcal{I}_{\overline{G}}$ measured on time-discretized norms. We bound the discrete norms by continuous ones and use stability results for stochastic integration as presented in~\cite[Theorem~3.2]{MR4116708}.

Notice that
\begin{align*}
\mathrm{J}_1^a &= \mathbb{E}\left[ \left( \max_{k \in [N]} \frac{1}{k} \dashint_{J_k} \abs{\frac{t}{k\tau}}^{(r-2)/2}   \frac{\norm{\mathcal{I}_{\overline{G}}(t)}_{L^2_x}^r}{\abs{t}^{(r-2)/2}} \dd t  \right)^{2/r} \right] \\
&\leq \mathbb{E}\left[ \left( \sup_{t \in I} \frac{\norm{\mathcal{I}_{\overline{G}}(t)}_{L^2_x}}{\abs{t}^{1/2 - 1/r}} \right)^2 \right] \\
&\leq \norm{\mathcal{I}_{\overline{G}}}_{L^2_\omega B^{\alpha_r}_{\infty,\infty} L^2_x}^2,
\end{align*}
where $\alpha_r = 1/2 - 1/r$. Using the embedding~$B^{1/2}_{r,\infty} \hookrightarrow B^{\alpha_r}_{\infty,\infty}$ (see e.g.~\cite[Theorem~10]{MR1108473}), we conclude
\begin{align} \label{eq:InitNew}
 \mathrm{J}_1^a \lesssim \norm{\mathcal{I}_{\overline{G}}}_{L^2_\omega B^{1/2}_{r,\infty} L^2_x}^2.
\end{align}
Additionally, it holds
\begin{align*}
\mathrm{J}_2^a \leq \mathbb{E}\left[   \left(\sup_{h} h^{-r/2} \int_{h}^T \norm{\mathcal{I}_{\overline{G}}(t) - \mathcal{I}_{\overline{G}}(t - h)}_{L^2_x}^r \dd t \right)^{2/r}   \right] \leq \norm{\mathcal{I}_{\overline{G}}}_{L^2_\omega B^{1/2}_{r,\infty} L^2_x}^2.
\end{align*}
Therefore, using~\cite[Theorem~3.2(ii)]{MR4116708} together with an extrapolation argument, cf.~\cite[Remark~3.4]{MR4116708}, we get
\begin{align} \label{eq:EstimateJ}
\mathrm{J}_1^a + \mathrm{J}_2^a \lesssim \norm{\overline{G}}_{L^2_\omega L^\infty_t L_2(\mathfrak{U};L^2_x)}^2.
\end{align}

Next we turn our attention to~$\mathrm{J}^b$. These terms are generated by the correlation of the increments~$\Delta_n \mathbb{W}$. Fortunately the correlation length is bounded. In our setting only adjecent increments are correlated.

Estimating the maximum by a sum and dropping the additional factor~$1/k$ establish 
\begin{align*}
\mathrm{J}_1^b \leq \mathbb{E}\left[ \left( \sum_{n=1}^N \int_{J_n} \norm{\frac{\overline{E}(t)}{\sqrt{\tau}}}_{L^2_x}^r \dd t \right)^{2/r} \right].
\end{align*}
Similarly,
\begin{align*}
\mathrm{J}_2^b &\lesssim \mathbb{E}\left[ \left( \sum_{n=1}^N \int_{J_n} \norm{\frac{\overline{E}(t)}{\sqrt{\tau}}}_{L^2_x}^r \dd t \right)^{2/r} \right].
\end{align*}

The assertion~\eqref{eq:EstimateStochasticPressure} follows from~\eqref{eq:EstimateJ} analogously to~\eqref{eq:GtoU}  provided we can verify
\begin{align} \label{eq:EstPertur}
\mathbb{E}\left[ \left( \sum_{n=1}^N \int_{J_n} \norm{\frac{\overline{E}(t)}{\sqrt{\tau}}}_{L^2_x}^r \dd t \right)^{2/r} \right] \lesssim \mathbb{E}\left[ \max_{n \in [N]} \norm{G_n(u_{(n-2)\vee 0})}_{L_2(\mathfrak{U};L^2_x)}^2 \right].
\end{align}

We will establish~\eqref{eq:EstPertur} with the help of the discrete extrapolation as presented in Appendix~\ref{app:Appendix}. Let us define the $(\mathcal{F}_n^N)$-adapted processes
\begin{align*}
X_M &:= \left( \sum_{n=1}^M \int_{J_n} \norm{\frac{\overline{E}(t)}{\sqrt{\tau}}}_{L^2_x}^r \dd t \right) 1_{\{M \leq N\}}, \\
Y_M &:= \max_{n \leq (M+2) \wedge N} \norm{G_n(u_{(n-2)\vee 0})}_{L_2(\mathfrak{U};L^2_x)}^r.
\end{align*}
By Lemma~\ref{lem:DominationCheck} the process $(X_n)$ is dominated by~$(Y_n)$. Therefore, Corollary~\ref{cor:Domination} with $k = 2/r \in (0,1)$ establishes
\begin{align*}
\mathbb{E}\left[ \left(\sup_{M \in \mathbb{N}} X_M \right)^{2/r} \right] \lesssim \mathbb{E}\left[ \left(\sup_{M \in \mathbb{N}} Y_M \right)^{2/r} \right].
\end{align*}
Notice that
\begin{align*}
\sup_{M \in \mathbb{N}} X_M &= \sum_{n=1}^N \int_{J_n} \norm{\frac{\overline{E}(t)}{\sqrt{\tau}}}_{L^2_x}^r \dd t,\\
\sup_{M \in \mathbb{N}} Y_M &= \max_{n \leq N} \norm{G_n(u_{(n-2)\vee 0})}_{L_2(\mathfrak{U};L^2_x)}^r.
\end{align*}
Therefore, we have verified~\eqref{eq:EstPertur} and the proof is finished.
\end{proof}

\section{Velocity error for exactly divergence free approximations} \label{sec:ErrorDecomp}
In this section we decompose the velocity error for abstract exactly divergence-free approximations on the natural scale of weak solutions 
\begin{align*}
L^2_\omega L^\infty_t L^2_x \cap L^p_\omega L^p_t W^{1,p}_{0,x} \cap B^{1/2}_{2,\infty} L^2_\omega L^2_x
\end{align*}
into individual error contributions.

% Each error term provides a theoretically motivated design principle for the construction of new specific discretizations, i.e., one can try to minimize the error terms.
%
%It uses similar arguments as the derivation of the a priori estimates, cf. Theorem~\ref{thm:EnergyBoundsVelocity}. We discuss the decomposition on $L^2_\omega L^\infty_t L^2_x \cap L^{p}_\omega L^p_t W^{1,p}_{0,x}$ and $B^{1/2}_{2,\infty} L^2_\omega L^2_x$ individually.

From now on we assume that~$V_{h,\Div} \subset L^2_x \cap W^{1,p}_{0,\Div}$. In other words, the spatial discretization is conforming with respect to regularity and the divergence free constraint. As a consequence the equation for the velocity error does not see any pressure contribution.

Let $e_n := \mean{u}_n - u_n$. The error equation reads, for all $n \in [N]$, $\xi \in V_{h,\Div}$ and~$\mathbb{P}$-a.s.
\begin{align} \label{eq:ErrorEquation}
\begin{aligned}
&\int_{\mathcal{O}} d_n e \cdot \xi \dd x + \int_{\mathcal{O}}  \int a_n(s) \left( S\big(\varepsilon u(s) \big) - S(\varepsilon u_n) \right) \dd s : \varepsilon \xi \dd x \\
&\hspace{3em} = \int_{\mathcal{O}}  \int a_n(s) \left( G\big(s,u(s) \big)  - G_n(u_{(n-2)\vee 0}) \right) \dd W(s) \cdot  \xi \dd x.
\end{aligned}
\end{align}

\subsection{Decomposition on $L^2_\omega L^\infty_t L^2_x \cap L^{p}_\omega L^p_t W^{1,p}_{0,x}$}
The first decomposition is closely related to Part~A in the proof of Theorem~\ref{thm:EnergyBoundsVelocity}.

\begin{theorem} \label{thm:ErrorEstimates}
Let $u$ be a weak solution to~\eqref{eq:intro_pStokes} and $u_n$ be generated by~\eqref{eq:velocity001}. Then it holds
\begin{align}\label{eq:ErrorEstimate}
\begin{aligned}
& \mathbb{E}\left[ \max_{n} \norm{\mean{u}_n - u_n}_{L^2_x}^2 + \sum_{n} \int a_n(t) \norm{V\big( \varepsilon u(t) \big) - V(\varepsilon u_n)}_{L^2_x}^2 \dd t \right] \\
&\hspace{3em} \lesssim C_{\mathrm{init}} + C_{L^\infty} +  C_{\mathrm{best}} + C_{G},
\end{aligned}
\end{align}
where
\begin{itemize}
\item initial error
\begin{align*}
C_{\mathrm{init}} = \mathbb{E}\left[ \norm{\Pi_{\Div} u_0 - \Pi_{\Div} u_0^h}_{L^2_x}^2 \right],
\end{align*}
\item projection error in~$L^2_\omega L^\infty_t L^2_x$
\begin{align*}
C_{L^\infty} = \mathbb{E}\left[ \max_{n} \norm{\mean{u}_n - \Pi_{\Div} \mean{u}_n}_{L^2_x}^2 \right],
\end{align*}
\item best approximation in $V_{h,\Div}$
\begin{align*}
C_{\mathrm{best}} = \inf_{\eta \in [V_{h,\Div}]^N} \mathbb{E}\left[\sum_{n=1}^N \norm{\Pi_{\Div} \mean{u}_n - \eta_n}_{L^2_x}^2 + \int a_n(t)  \norm{V\big( \varepsilon u(t) \big) - V(\varepsilon \eta_n)}_{L^2_x}^2 \dd t\right],
\end{align*} 
\item data approximation
\begin{align*}
C_{G}  = \mathbb{E}\left[\sum_{n=1}^N \int a_n^2(t) \norm{G\big(t, u(t) \big)- G_n(u_{(n-2)\vee 0})}_{L_2(\mathfrak{U};L^2_x)}^2 \dd t \right].
\end{align*}
\end{itemize}

\end{theorem}

\begin{proof}
The overall aim is to test the error equation~\eqref{eq:ErrorEquation} with the error itself and to use the $V$-coercivity of~$S$. However, since $\mean{u}_n$ is not a discrete function this can not be done. Instead we test~\eqref{eq:ErrorEquation} with~$\xi = \eta_n - u_n$ where $\eta_n \in V_{h,\Div}$ is a free variable. Afterwards, we artificially introduce~$\mean{u}_n$ in the test function. We discuss each term appearing in~\eqref{eq:ErrorEquation} separately.

Due to the symmetry of~$\Pi_{\Div}$, H\"older's and Young's inequalities
\begin{align} \label{eq:ErrorLinfL2}
\begin{aligned}
&\int_{\mathcal{O}} d_n e \cdot (\eta_n - u_n) \dd x \\
&\hspace{3em} = \int_{\mathcal{O}} d_n \Pi_{\Div} e \cdot (\Pi_{\Div}e_n + \eta_n - \Pi_{\Div}\mean{u}_n) \dd x \\
&\hspace{3em}  = \frac{1}{2}\left( \norm{\Pi_{\Div}e_n}_{L^2_x}^2 -\norm{\Pi_{\Div}e_{n-1}}_{L^2_x}^2 + \norm{\Pi_{\Div}(e_n - e_{n-1})}_{L^2_x}^2 \right) \\
&\hspace{6em}  + \int_{\mathcal{O}} d_n \Pi_{\Div}e \cdot (\eta_n - \Pi_{\Div}\mean{u}_n) \dd x \\
&\hspace{3em} \geq \frac{1}{2}\left( \norm{\Pi_{\Div}e_n}_{L^2_x}^2 -\norm{\Pi_{\Div}e_{n-1}}_{L^2_x}^2 \right) + \frac{1}{4} \norm{\Pi_{\Div}(e_n - e_{n-1})}_{L^2_x}^2 \\
&\hspace{6em} - \norm{\eta_n - \Pi_{\Div}\mean{u}_n}_{L^2_x}^2.
\end{aligned}
\end{align}

Due to Lemma~\ref{lem:Relation} and Lemma~\ref{lem:GeneralizedYoung}
\begin{align*}
&\int_{\mathcal{O}}  \int a_n(s) \left( S\big(\varepsilon u(s) \big) - S(\varepsilon u_n) \right) \dd s : \varepsilon (\eta_n - u_n) \dd x \\
&\quad \geq c \int a_n(s) \norm{V\big(\varepsilon u(s) \big) - V(\varepsilon u_n)}_{L^2_x}^2 \dd s\\
&\quad \quad + \int a_n(s)  \int_{\mathcal{O}}  \left( S\big(\varepsilon u(s) \big) - S(\varepsilon u_n) \right)  : \varepsilon (\eta_n - u(s)) \dd x\dd s \\
&\quad \geq c(1-\delta) \int a_n(s) \norm{V\big(\varepsilon u(s) \big) - V(\varepsilon u_n)}_{L^2_x}^2 \dd s \\
&\quad \quad - c_\delta  \int a_n(s) \norm{V\big(\varepsilon u(s) \big) - V(\varepsilon \eta_n)}_{L^2_x}^2 \dd s.
\end{align*}

We delay the test function such that it is is uncorrelated to the $L^2_x$-valued random variable
\begin{align*}
 \int a_n(s) \left( G\big(s,u(s) \big)  - G_n(u_{(n-2)\vee 0}) \right) \dd W(s).
\end{align*}
Therefore, we split 
\begin{align*}
&\int_{\mathcal{O}}  \int a_n(s) \left( G\big(s,u(s) \big)  - G_n(u_{(n-2)\vee 0}) \right) \dd W(s) \cdot  (\eta_n - u_n) \dd x \\
&\quad = \int_{\mathcal{O}}  \int a_n(s) \left( G\big(s,u(s) \big)  - G_n(u_{(n-2)\vee 0}) \right) \dd W(s) \cdot  \Pi_{\Div}(e_n  - e_{(n-2) \vee 0} ) \dd x \\
&\quad \quad + \int_{\mathcal{O}}  \int a_n(s) \left( G\big(s,u(s) \big)  - G_n(u_{(n-2)\vee 0}) \right) \dd W(s) \cdot \Pi_{\Div} e_{(n-2) \vee 0} \dd x \\
&\quad\quad + \int_{\mathcal{O}}  \int a_n(s) \left( G\big(s,u(s) \big)  - G_n(u_{(n-2)\vee 0}) \right) \dd W(s) \cdot  ( \eta_n - \Pi_{\Div}\mean{u}_n) \dd x \\
&\quad =: \mathfrak{J}_1(n) + \mathfrak{J}_2(n) + \mathfrak{J}_3(n).
\end{align*}

Summing up, and taking maximum and expectation
\begin{align}\label{eq:ErrorEstimateAlmost}
\begin{aligned}
&\mathbb{E} \left[ \max_{n \leq N} \norm{\Pi_{\Div}e_n}_{L^2_x}^2  + (1-\delta) \sum_{n=1}^N \int a_n(s) \norm{V\big(\varepsilon u(s) \big) - V(\varepsilon u_n)}_{L^2_x}^2 \dd s \right] \\
&\hspace{2em} + \mathbb{E} \left[ \sum_{n=1}^N \norm{\Pi_{\Div}(e_n - e_{n-1})}_{L^2_x}^2  \right] \\
&\hspace{1em} \lesssim \mathbb{E} \left[  \norm{\Pi_{\Div}e_0}_{L^2_x}^2\right] +\mathbb{E} \left[ \max_{M \leq N} \sum_{n=1}^M\mathfrak{J}_1(n) + \mathfrak{J}_2(n) + \mathfrak{J}_3(n)\right]\\
&\hspace{2em} +\mathbb{E} \left[  \sum_{n=1}^N \norm{\eta_n - \Pi_{\Div}\mean{u}_n}_{L^2_x}^2  +  c_\delta \int a_n(s) \norm{V\big(\varepsilon u(s) \big) - V(\varepsilon \eta_n)}_{L^2_x}^2 \dd s\right].
\end{aligned}
\end{align}

It remains to estimate the contributions due to the stochastic integral. The first and third term are estimated by H\"older's and Young's inequalities
\begin{align*}
&\mathbb{E} \left[ \max_{M \leq N} \sum_{n=1}^M\mathfrak{J}_1(n) + \mathfrak{J}_3(n)\right] \\
&\hspace{3em} \leq c_\delta  \mathbb{E} \left[ \sum_{n=1}^N \norm{ \int a_n(s) \left( G\big(s,u(s) \big)  - G_n(u_{(n-2)\vee 0}) \right) \dd W(s)}_{L^2_x}^2 \right] \\
&\hspace{4em} +  \delta \sum_{n=1}^N \mathbb{E}\left[ \norm{\Pi_{\Div}(e_n  - e_{(n-2) \vee 0} ) }_{L^2_x}^2 \right] + \sum_{n=1}^N \mathbb{E}\left[ \norm{\Pi_{\Div} \mean{u}_n - \eta_n }_{L^2_x}^2 \right].
\end{align*}

The It\^o isometry shows
\begin{align*}
&\mathbb{E} \left[ \sum_{n=1}^N \norm{ \int a_n(s) \left( G\big(s,u(s) \big)  - G_n(u_{(n-2)\vee 0}) \right) \dd W(s)}_{L^2_x}^2 \right] \\
&\quad = \mathbb{E} \left[ \sum_{n=1}^N  \int a_n^2(s) \norm{G\big(s,u(s) \big)  - G_n(u_{(n-2)\vee 0}) }_{L_2(\mathfrak{U};L^2_x)}^2 \dd s \right].
\end{align*}

For~$\mathfrak{J}_2$ we proceed as in the derivation of the a priori estimates for the approximation~$u_n$, cf.~\eqref{eq:DiscreteStochasticIntegral}. We introduce the random variable
\begin{align*}
\mathfrak{Y}_M := \sum_{n=1}^M \mathfrak{J}_2(n).
\end{align*}
It decomposes into a martingale with respect to~$(\mathcal{F}_n^N)$ and a compensator
\begin{align} \label{eq:DecompErrorMart}
\mathbb{E}\left[\mathfrak{Y}_N \big| \mathcal{F}_{M}^N \right] - \mathfrak{Y}_{M} = \mathfrak{E}_M,
\end{align}
where
\begin{align*}
\mathfrak{E}_M =  &\left( \int_{\mathcal{O}}  \int_{J_M} a_{M+1}(s) \left( G\big(s,u(s) \big)  - G_{M+1}(u_{(M-1)\vee 0}) \right) \dd W(s) \cdot \Pi_{\Div} e_{(M-1) \vee 0} \dd x \right) \\
&\hspace{3em} 1_{\{M+1 \leq N\}}.
\end{align*}
Now we can estimate the maximal process in terms of the maximal process of its summands
\begin{align*}
\mathbb{E}\left[\max_{M \leq N} \mathfrak{Y}_M \right] \leq \mathbb{E}\left[\max_{M \leq N} \mathbb{E}\left[\mathfrak{Y}_N \big| \mathcal{F}_{M}^N \right] \right] + \mathbb{E}\left[\max_{M \leq N} \mathfrak{E}_M \right].
\end{align*}

An application of the Burkholder-Davis-Gundy inequality and~\eqref{eq:DecompErrorMart} verify
\begin{align*}
\mathbb{E}\left[\max_{M \leq N} \mathbb{E}\left[\mathfrak{Y}_N \big| \mathcal{F}_{M}^N \right] \right] &\lesssim \mathbb{E}\left[ \left( \sum_{n=1}^N \left( \mathbb{E}\left[\mathfrak{Y}_N \big| \mathcal{F}_{n}^N \right] - \mathbb{E}\left[\mathfrak{Y}_N \big| \mathcal{F}_{n-1}^N \right] \right)^2 \right)^{1/2} \right] \\
&\lesssim \mathbb{E}\left[ \left( \sum_{n=1}^N \left(\mathfrak{Y}_n - \mathfrak{Y}_{n-1} \right)^2 \right)^{1/2} \right] + \mathbb{E}\left[ \left( \sum_{n=1}^N \mathfrak{E}_n^2 \right)^{1/2} \right].
\end{align*}
Due to H\"older's and Young's inequalities and the It\^o isometry
\begin{align*}
&\mathbb{E}\left[\max_{M \leq N} \mathfrak{E}_M \right] \leq \mathbb{E}\left[ \left( \sum_{n=1}^N \mathfrak{E}_n^2 \right)^{1/2} \right] \\
&\hspace{2em} \lesssim \delta\mathbb{E}\left[ \max_{M \leq N} \norm{\Pi_{\Div} e_{(M-2) \vee 0}}_{L^2_x}^2 \right] \\
&\hspace{4em}  + c_\delta \mathbb{E}\left[ \sum_{n=1}^N \int a_n^2(s)\norm{  G\big(s,u(s) \big)  - G_n(u_{(n-2)\vee 0})}_{L_2(\mathfrak{U};L^2_x)}^2 \dd s \right].
\end{align*}
Similarly one derives
\begin{align*}
&\mathbb{E}\left[ \left( \sum_{n=1}^N \left(\mathfrak{Y}_n - \mathfrak{Y}_{n-1} \right)^2 \right)^{1/2} \right] \lesssim \delta\mathbb{E}\left[ \max_{M \leq N} \norm{\Pi_{\Div} e_{(M-2) \vee 0}}_{L^2_x}^2 \right] \\
&\hspace{4em} + c_\delta \mathbb{E}\left[ \sum_{n=1}^N \int a_n^2(s)\norm{  G\big(s,u(s) \big)  - G_n(u_{(n-2)\vee 0})}_{L_2(\mathfrak{U};L^2_x)}^2 \dd s \right].
\end{align*}
Overall, we have proved
\begin{align} \label{eq:EstimateStochasticIntegralMart}
\begin{aligned}
&\mathbb{E} \left[ \max_{M \leq N} \sum_{n=1}^M\mathfrak{J}_1(n) + \mathfrak{J}_2(n) + \mathfrak{J}_3(n)\right] \\
&\hspace{2em} \lesssim c_\delta \mathbb{E}\left[ \sum_{n=1}^N \int a_n^2(s)\norm{  G\big(s,u(s) \big)  - G_n(u_{(n-2)\vee 0})}_{L_2(\mathfrak{U};L^2_x)}^2 \dd s \right] \\
&\hspace{3em} +  \delta \sum_{n=1}^N \mathbb{E}\left[ \norm{\Pi_{\Div}(e_n  - e_{(n-2) \vee 0} ) }_{L^2_x}^2 \right] + \delta\mathbb{E}\left[ \max_{M \leq N} \norm{\Pi_{\Div} e_{(M-2) \vee 0}}_{L^2_x}^2 \right]\\
&\hspace{3em} + \sum_{n=1}^N \mathbb{E}\left[ \norm{\Pi_{\Div} \mean{u}_n - \eta_n }_{L^2_x}^2 \right].
\end{aligned}
\end{align}
Using~\eqref{eq:EstimateStochasticIntegralMart} in~\eqref{eq:ErrorEstimateAlmost} and choose~$\delta> 0$ sufficiently small
\begin{align} \label{eq:EstimateFinal01}
\begin{aligned}
&\mathbb{E} \left[ \max_{n \leq N} \norm{\Pi_{\Div}e_n}_{L^2_x}^2  + \sum_{n=1}^N \int a_n(s) \norm{V\big(\varepsilon u(s) \big) - V(\varepsilon u_n)}_{L^2_x}^2 \dd s \right] \\
&\hspace{2em} + \mathbb{E} \left[ \sum_{n=1}^N \norm{\Pi_{\Div}(e_n - e_{n-1})}_{L^2_x}^2  \right] \\
&\hspace{1em} \lesssim \mathbb{E} \left[  \norm{\Pi_{\Div}e_0}_{L^2_x}^2\right] + \mathbb{E}\left[ \sum_{n=1}^N \int a_n^2(s)\norm{  G\big(s,u(s) \big)  - G_n(u_{(n-2)\vee 0})}_{L_2(\mathfrak{U};L^2_x)}^2 \dd s \right]\\
&\hspace{2em} +\mathbb{E} \left[  \sum_{n=1}^N \norm{\eta_n - \Pi_{\Div}\mean{u}_n}_{L^2_x}^2  +  \int a_n(s) \norm{V\big(\varepsilon u(s) \big) - V(\varepsilon \eta_n)}_{L^2_x}^2 \dd s\right].
\end{aligned}
\end{align}
This concludes the error estimate for the projected error. In order to get an estimate for the full error we artificially introduce the projected error and trivially estimate
\begin{align*}
\mathbb{E} \left[ \max_{n \leq N} \norm{e_n}_{L^2_x}^2  \right] \lesssim \mathbb{E} \left[ \max_{n \leq N} \norm{\Pi_{\Div}\mean{u}_n - \mean{u}_n}_{L^2_x}^2 \right] + \mathbb{E} \left[ \max_{n \leq N} \norm{\Pi_{\Div}e_n}_{L^2_x}^2 \right].
\end{align*}
Lastly, since~$\eta \in [V_{h,\Div}]^N$ is a free parameter we can take the infimum over~$\eta \in [V_{h,\Div}]^N$ and the assertion follows.
\end{proof}

\subsection{Decomposition on $B^{1/2}_{2,\infty} L^2_\omega L^2_x$}
The decomposition of time differences is closely related to Part~B in the proof of Theorem~\ref{thm:EnergyBoundsVelocity}. 

\begin{theorem}\label{lem:ErrorEstimateBesov}
In the framework of Theorem~\ref{thm:ErrorEstimates} it holds
\begin{align}\label{eq:EstimateB12}
\begin{aligned}
&\max_{k \in [N]} \frac{1}{k} \sum_{n=k}^N \mathbb{E}\left[ \norm{(\mean{u}_n - u_n) - (\mean{u}_{n-k} - u_{n-k}) }_{L^2_x}^2 \right] \\
&\hspace{4em} \lesssim C_{\mathrm{init}} + C_{B^{1/2}_{2,\infty}}+ C_{\mathrm{best}} + C_{G}  + C_{V},
\end{aligned}
\end{align}
where 
\begin{itemize}
\item temporal oscillation of non-linear gradient
\begin{align*}
C_V &= \frac{1}{\tau} \sum_{n=1}^N \int \int a_n(s)  a_{n}(t) \mathbb{E}\left[    \norm{ V\big(\varepsilon u(s) \big)  -  V\big(\varepsilon u(t)  \big)}_{L^2_x}^2  \right]\dd t \dd s,
\end{align*}
\item projection error in~$B^{1/2}_{2,\infty} L^2_\omega L^2_x$
\begin{align*}
C_{B^{1/2}_{2,\infty}} &= \max_{k \in [N]} \frac{1}{k} \sum_{n=k}^N\mathbb{E}\left[ \norm{\Pi_{\Div}(\mean{u}_n - \mean{u}_{n-k} ) - (\mean{u}_n - \mean{u}_{n-k} ) }_{L^2_x}^2 \right].
\end{align*}
\end{itemize}
\end{theorem}

\begin{proof}
First of all notice that
\begin{align} \label{eq:Decomp002}
\begin{aligned}
&\max_{k \in [N]} \frac{1}{k} \sum_{n=k}^N \mathbb{E}\left[ \norm{e_n - e_{n-k} }_{L^2_x}^2 \right]\\
&\hspace{2em} \lesssim \max_{k \in [N]} \frac{1}{k} \sum_{n=k}^N \mathbb{E}\left[ \norm{\Pi_{\Div}(e_n - e_{n-k}) }_{L^2_x}^2 \right] + C_{B^{1/2}_{2,\infty}}\\
&\hspace{2em} \lesssim \max_{k \in [N]}\mathbb{E}\left[ \norm{\Pi_{\Div}(e_k - e_{0}) }_{L^2_x}^2 \right] \\
&\hspace{4em} + \max_{k \in [N]} \frac{1}{k} \sum_{n=k+1}^N \mathbb{E}\left[ \norm{\Pi_{\Div}(e_n - e_{n-k}) }_{L^2_x}^2 \right] + C_{B^{1/2}_{2,\infty}}.
\end{aligned}
\end{align}
We will concentrate on the estimate for the second term.

Similarly to~\eqref{eq:Global} we obtain
\begin{align*}
&\int_{\mathcal{O}} (e_n - e_{n-k}) \cdot \xi \dd x + \int_{\mathcal{O}}  \sum_{l=n-k+1}^n \int a_l(s) \left( S\big(\varepsilon u(s) \big) - S(\varepsilon u_l) \right) \dd s : \varepsilon \xi \dd x \\
&\hspace{3em} = \int_{\mathcal{O}} \sum_{l=n-k+1}^n  \int a_l(s) \left( G\big(s,u(s) \big)  - G_n(u_{(l-2)\vee 0}) \right) \dd W(s) \cdot  \xi \dd x.
\end{align*}

Choose~$\xi = \eta_n - u_n \in V_h$ and sum up for $n \in \{k+1,\ldots,N\}$
\begin{align} \label{eq:GlobalK01}
\begin{aligned}
\mathfrak{L}_1^k &:=  \sum_{n=k+1}^N \int_{\mathcal{O}} (e_n - e_{n-k}) \cdot (\eta_n - u_n) \dd x \\
&= -\sum_{n=k+1}^N\int_{\mathcal{O}}  \sum_{l=n-k+1}^n \int a_l(s) \left( S\big(\varepsilon u(s) \big) - S(\varepsilon u_l) \right) \dd s : \varepsilon (\eta_n - u_n)\dd x \\
&\quad + \sum_{n=k+1}^N\int_{\mathcal{O}} \sum_{l=n-k+1}^n  \int a_l(s) \left( G\big(s,u(s) \big)  - G_l(u_{(l-2)\vee 0}) \right) \dd W(s) \cdot  (\eta_n - u_n) \dd x \\
&=: \mathfrak{L}_2^k + \mathfrak{L}_3^k.
\end{aligned}
\end{align}

Analogously to~\eqref{eq:ErrorLinfL2} we can estimate
\begin{align*}
\mathfrak{L}_1^k &\geq \frac{1}{2} \sum_{n=k+1}^N \left( \norm{\Pi_{\Div}e_n}_{L^2_x}^2 -\norm{\Pi_{\Div}e_{n-k}}_{L^2_x}^2 \right) + \frac{1}{4} \sum_{n=k+1}^N \norm{\Pi_{\Div}(e_n - e_{n-k})}_{L^2_x}^2 \\
&\hspace{6em} - \sum_{n=k+1}^N  \norm{\eta_n - \Pi_{\Div}\mean{u}_n}_{L^2_x}^2.
\end{align*}

Applying generalized Young's inequality
\begin{align*}
&\mathfrak{L}_2^k \lesssim \sum_{n=k+1}^N \sum_{l=n-k+1}^n \int a_l(s) \norm{V\big(\varepsilon u(s)\big) - V(\varepsilon u_l)}_{L^2_x}^2 \dd s \\
& +\sum_{n=k+1}^N \sum_{l=n-k+1}^n \int a_l(s) \left( \norm{V\big(\varepsilon u(s)\big) - V(\varepsilon u_n)}_{L^2_x}^2 + \norm{V\big(\varepsilon u(s)\big) - V(\varepsilon \eta_n)}_{L^2_x}^2\right) \dd s.
\end{align*}
The last term compares a time-average on~$\cup_{l=n-k} J_l$ of the non-linear gradient~$V(\varepsilon u)$ to a fixed approximation at time instance~$n$. We correct this miss-match by artificially introducing another averaged version of~$V(\varepsilon u)$. Define
\begin{align*}
\mean{V}_n^k := \frac{1}{\tau k} \sum_{l=n-k+1}^n \int a_l(t) V\big(\varepsilon u(t) \big) \dd t
\end{align*}
or equivalently
\begin{align*}
\mean{V}_n^k = \argmin_{Z \in L^2_x} \frac{1}{k}  \sum_{l=n-k+1}^n \int a_l(t) \norm{ V\big(\varepsilon u(t) \big) - Z }_{L^2_x}^2 \dd t.
\end{align*}
Most importantly,~$\mean{V}_n^k$ is close to all its point-values on $\{n-k+1,\ldots, n\}$ simultaneously. Indeed, let $l \in \{n-k+1,\ldots, n\}$. Then by an application of Jensen's inequality and the local support of the weights~$a_l$,
\begin{align*}
&\int a_l(s) \mathbb{E}\left[  \norm{V\big(\varepsilon u(s)\big) - \mean{V}_n^k}_{L^2_x}^2 \right]\dd s \\
&\hspace{2em}\leq \frac{1}{k \tau}\sum_{l'=n-k+1}^n \int \int a_l(s)  a_{l'}(t) \mathbb{E}\left[    \norm{ V\big(\varepsilon u(s) \big)  -  V\big(\varepsilon u(t)  \big)}_{L^2_x}^2  \right]\dd t \dd s \\
&\hspace{2em}\lesssim  \frac{1}{k \tau} \sum_{l' = l-1}^{l+1} \int \int a_{l'}(s)  a_{l'}(t) \mathbb{E}\left[    \norm{ V\big(\varepsilon u(s) \big)  -  V\big(\varepsilon u(t)  \big)}_{L^2_x}^2  \right]\dd t \dd s
\end{align*}
Therefore,
\begin{align*}
&\sum_{n=k+1}^N  \sum_{l=n-k+1}^n \int a_l(s) \norm{V\big(\varepsilon u(s)\big) - V(\varepsilon u_n)}_{L^2_x}^2 \dd s \\
&\quad \lesssim \sum_{n=k+1}^N \sum_{l=n-k+1}^n \int a_l(s) \norm{V\big(\varepsilon u(s)\big) - \mean{V}_n^k \big)}_{L^2_x}^2 \dd s \\
&\quad+ \sum_{n=k+1}^N k \int a_n(s) \norm{\mean{V}_n^k  - V\big(\varepsilon u(s) \big)}_{L^2_x}^2 \dd s \\
&\quad + \sum_{n=k+1}^N k \int a_n(s) \norm{V\big( \varepsilon u(s) \big) - V(\varepsilon u_n)}_{L^2_x}^2 \dd s \\
&\lesssim \frac{1}{\tau} \sum_{n=1}^N \int \int a_n(s)  a_{n}(t) \mathbb{E}\left[    \norm{ V\big(\varepsilon u(s) \big)  -  V\big(\varepsilon u(t)  \big)}_{L^2_x}^2  \right]\dd t \dd s \\
&\quad + \sum_{n=k+1}^N k \int a_n(s) \norm{V\big( \varepsilon u(s) \big) - V(\varepsilon u_n)}_{L^2_x}^2 \dd s.
\end{align*}

The stochastic term~$\mathcal{L}_3$ is estimated in expectation. We artificially introduce the analytic solution and increments 
\begin{align*}
&\mathcal{L}_3^k =\mathbb{E} \left[ \sum_{n=k+1}^N \int_{\mathcal{O}} \sum_{l=n-k+1}^n  \int a_l(s) \left( G\big(s,u(s) \big)  - G_l(u_{(l-2)\vee 0}) \right) \dd W(s) \cdot  (\eta_n - \Pi_{\Div} \mean{u}_n) \dd x \right]\\
&+ \mathbb{E} \left[ \sum_{n=k+1}^N \int_{\mathcal{O}} \sum_{l=n-k+1}^n  \int a_l(s) \left( G\big(s,u(s) \big)  - G_l(u_{(l-2)\vee 0}) \right) \dd W(s) \cdot  \Pi_{\Div} (e_n - e_{(n-k-1)\vee 0} ) \dd x \right].
\end{align*}
H\"older's and Young's inequalities and the It\^o isometry show
\begin{align*}
&\mathbb{E} \left[ \sum_{n=k+1}^N \int_{\mathcal{O}} \sum_{l=n-k+1}^n  \int a_l(s) \left( G\big(s,u(s) \big)  - G_l(u_{(l-2)\vee 0}) \right) \dd W(s) \cdot  \Pi_{\Div} (e_n - e_{(n-k-1)\vee 0} ) \dd x \right]\\
&\leq \delta \mathbb{E} \left[ \sum_{n=k+1}^N \norm{  \Pi_{\Div} (e_n - e_{(n-k-1)\vee 0} )}_{L^2_x}^2 \right] \\
&\quad + c_\delta \mathbb{E} \left[ \sum_{n=k+1}^N \norm{\sum_{l=n-k+1}^n  \int a_l(s) \left( G\big(s,u(s) \big)  - G_l(u_{(l-2)\vee 0}) \right) \dd W(s)}_{L^2_x}^2 \right] \\
&\lesssim \delta \mathbb{E} \left[ \sum_{n=k+1}^N \norm{  \Pi_{\Div} (e_n - e_{n-k} )}_{L^2_x}^2 \right] + \mathbb{E} \left[ \sum_{n=1}^N \norm{  \Pi_{\Div} (e_n - e_{n-1} )}_{L^2_x}^2 \right]\\
&\quad + c_\delta \mathbb{E} \left[ \sum_{n=k+1}^N \sum_{l=n-k+1}^n  \int a_l^2(s) \norm{G\big(s,u(s) \big)  - G_l(u_{(l-2)\vee 0})}_{L_2(\mathfrak{U}; L^2_x)}^2 \dd s \right].
\end{align*}
Similarly,
\begin{align*}
&\mathbb{E} \left[ \sum_{n=k+1}^N \int_{\mathcal{O}} \sum_{l=n-k+1}^n  \int a_l(s) \left( G\big(s,u(s) \big)  - G_l(u_{(l-2)\vee 0}) \right) \dd W(s) \cdot  (\eta_n - \Pi_{\Div} \mean{u}_n) \dd x \right] \\
&\hspace{3em} \lesssim \mathbb{E} \left[ \sum_{n=k+1}^N \sum_{l=n-k+1}^n  \int a_l^2(s) \norm{G\big(s,u(s) \big)  - G_l(u_{(l-2)\vee 0})}_{L_2(\mathfrak{U}; L^2_x)}^2 \dd s \right] \\
&\hspace{4em} + \mathbb{E} \left[ \sum_{n=k+1}^N   \norm{\eta_n - \Pi_{\Div} \mean{u}_n)}_{L^2_x}^2 \right].
\end{align*}

Let us go back to~\eqref{eq:GlobalK01}. Take expectation, divide by~$k$ and take the maximum over~$k \in [N]$, using the previously established estimates,
\begin{align*}
&(1-\delta) \max_{k\in [N]} \frac{1}{k} \sum_{n=k+1}^N \mathbb{E}\left[\norm{\Pi_{\Div}(e_n - e_{n-k})}_{L^2_x}^2 \right] \\
&\lesssim \mathbb{E}\left[\max_{n \in [N]}\norm{\Pi_{\Div} e_n}_{L^2_x}^2 \right] + \mathbb{E} \left[ \sum_{n=1}^N \norm{  \Pi_{\Div} (e_n - e_{n-1} )}_{L^2_x}^2 \right] \\
&\hspace{3em} + \mathbb{E}\left[ \sum_{n=1}^N \int a_n(s) \norm{V\big(\varepsilon u(s)\big) - V(\varepsilon u_n)}_{L^2_x}^2 \dd s \right]  \\
&\quad+  c_\delta \mathbb{E} \left[ \sum_{n=1}^N  \int a_n^2(s) \norm{G\big(s,u(s) \big)  - G_n(u_{(n-2)\vee 0})}_{L_2(\mathfrak{U}; L^2_x)}^2 \dd s \right]\\
&\quad + \mathbb{E}\left[  \sum_{n=1}^N  \norm{\Pi_{\Div}\mean{u}_n - \eta_n}_{L^2_x}^2 + \int a_n(s) \norm{V\big(\varepsilon u(s)\big) - V(\varepsilon \eta_n)}_{L^2_x}^2 \dd s \right] \\
&\quad +\frac{1}{\tau} \sum_{n=1}^N \int \int a_n(s)  a_{n}(t) \mathbb{E}\left[    \norm{ V\big(\varepsilon u(s) \big)  -  V\big(\varepsilon u(t)  \big)}_{L^2_x}^2  \right]\dd t \dd s.
\end{align*}
Next, we choose~$\delta>0$ sufficiently small and use~\eqref{eq:EstimateFinal01}. Additionally, we may take the infimum over~$\eta \in [V_{h,\Div}]^N$ to find
\begin{align*} 
&\max_{k\in [N]} \frac{1}{k} \sum_{n=k+1}^N \mathbb{E}\left[\norm{\Pi_{\Div}(e_n - e_{n-k})}_{L^2_x}^2 \right] \lesssim C_{\mathrm{init}} + C_G + C_{\mathrm{best}} + C_{V}.
\end{align*}
This together with~\eqref{eq:Decomp002} establish~\eqref{eq:EstimateB12} and the proof is finished.
\end{proof}

\begin{remark}
We expect that similar to the second scale of the a priori estimates, cf. Section~\ref{sec:SecondScale}, it is possible to establish a decomposition of the error in $L^2_\omega B^{1/2}_{2,\infty} L^2_x$. The data error needs to be measured on $L^2_\omega L^\infty_t L_2(\mathfrak{U};L^2_x)$, e.g.,
\begin{align*}
\tilde{C}_G = \mathbb{E}\left[\max_{n \in [N]} \int \frac{a_n^2(t)}{\tau} \norm{G\big(t, u(t) \big)- G_n(u_{(n-2)\vee 0})}_{L_2(\mathfrak{U};L^2_x)}^2 \dd t \right].
\end{align*}   
\end{remark}

The temporal oscillation of the non-linear gradient can be controlled by temporal regularity of the non-linear gradient. The most important ingredient is a quantified decay of increments. It can be measured in terms of Besov spaces.

\begin{lemma} \label{lem:RegulartiyV}
Let $\alpha \in (0,1)$ and $r \in [1,\infty]$. Moreover, assume  $V(\varepsilon u) \in B^{\alpha}_{2,r} L^2_\omega L^2_x$. Then
\begin{align}\label{eq:EstimateTemporalOscillation}
C_V \lesssim \tau^{2\alpha} \seminorm{V(\varepsilon u)}_{ B^{\alpha}_{2,r} L^2_\omega L^2_x}^2.
\end{align}
\end{lemma}
\begin{proof}
We rewrite the integration over $(s,t) \in \mathrm{supp}(a_n) \times \mathrm{supp}(a_n)$ to $(h,t) \in [0,2\tau] \times I_n^h$ where $I_h^n = \{ t \in \mathrm{supp}(a_n) | \, t+h \in \mathrm{supp}(a_n) \}$. Moreover, notice that $\sum_{n=1}^N 1_{\{I_h^n\}} \leq 2\, 1_{\{I_h\}}$ for $I_h = \{t \in I | \, t+h \in I\}$. Thus, Fubini's theorem together with $a_n \leq 1$ show
\begin{align*}
C_V &= \frac{1}{\tau} \sum_{n=1}^N  \int \int  a_n(s)  a_{n}(t) \mathbb{E}\left[    \norm{ V\big(\varepsilon u(s) \big)  -  V\big(\varepsilon u(t)  \big)}_{L^2_x}^2  \right]\dd t \dd s\\
&\lesssim \frac{1}{\tau}  \int_0^{2\tau} \int_{I_h} \mathbb{E}\left[    \norm{ V\big(\varepsilon u(t+h) \big)  -  V\big(\varepsilon u(t)  \big)}_{L^2_x}^2  \right]\dd t \dd h.
\end{align*}
H\"older's inequality implies
\begin{align*}
&\int_0^{2\tau} \int_{I_h} \mathbb{E}\left[    \norm{ V\big(\varepsilon u(t+h) \big)  -  V\big(\varepsilon u(t)  \big)}_{L^2_x}^2  \right]\dd t \dd h\\
&\hspace{3em} \leq  \left( \int_0^{2\tau} \left( \int_{I_h}    \frac{\mathbb{E}\left[    \norm{ V\big(\varepsilon u(t+h) \big)  -  V\big(\varepsilon u(t)  \big)}_{L^2_x}^2  \right]}{h^{2\alpha}}\dd t \right)^r \frac{\dd h}{h} \right)^{1/r} \\
&\hspace{5em}  \left( \int_0^{2\tau} h^{(2 \alpha + 1/r){r'}} \dd h \right)^{1/r'} \\
&\hspace{3em} \lesssim \seminorm{V(\varepsilon u)}_{B^{\alpha}_{2,r} L^2_\omega L^2_x}^2 \tau^{2\alpha +1}.
\end{align*}
This finishes the proof.
\end{proof}

\section{Convergence rates} \label{sec:Convergence}
In this section we give an example for a possible choice of~$(V_h,Q_h)$ and a sufficient condition on the data approximation of~$G$ such that the error estimates of Theorem~\ref{thm:ErrorEstimates} and Theorem~\ref{lem:ErrorEstimateBesov} provide a quantified convergence rate. 

\subsection{Finite elements}
Let $\mathcal{T}_h$ denote a regular partition (triangulation)
of~$\mathcal{O}$ (no hanging nodes), which consists of closed $n$-simplices called \emph{elements}. For each element ($n$-simplex) $K\in \mathcal{T}_h$ we denote by $h_K$ the diameter of $T$, and by $\rho_K$ the supremum of the diameters of inscribed balls. The neighborhood of an element~$K$ is denoted by $\omega_K$ and is the union of all elements that share a facet with element~$K$.

We assume that $\mathcal{T}_h$ is \emph{shape regular}, that is there exists a constant~$\gamma$ (the shape regularity constant) such that
\begin{align}
  \label{eq:nondeg}
  \max_{K \in \mathcal{T}_h} \frac{h_K}{\rho_K} \leq \gamma.
\end{align}
We define the maximal mesh-size by
\begin{align*}
  h &:= \max_{K \in \mathcal{T}_h} h_K.
\end{align*}
We assume further that our triangulation is \emph{quasi-uniform}, i.e.
\begin{align}
  \label{eq:quasi-uniform}
  h_K \eqsim h \qquad \text{for all $K \in \mathcal{T}_h$}.
\end{align}
For $\ell\in \setN _0$ we denote by $\mathscr{P}_\ell(\mathcal{O})$ the
polynomials on $\mathcal{O}$ of degree less than or equal to
$\ell$. 

For fixed $r \in \setN$ we define the vector valued finite element spaces $X_h$ and $V_h$ by
\begin{subequations} \label{def:velocity-FEM}
\begin{align}\label{def:Xh}
    X_h &:= \set{v \in W^{1,\infty}_{x} \,:\, v|_K \in \mathscr{P}_r(K) 
      \,\,\forall K\in \mathcal{T}_h},\\ \label{def:Vh}
   V_h &:= X_h \cap W^{1,1}_{0,x}.
\end{align} 
\end{subequations}
Similarly, we define the scalar valued finite element spaces
\begin{subequations} \label{def:pressure-FEM}
\begin{align}\label{def:Yh}
    Y_h &:= \set{q \in L^\infty_x \,:\, q|_K \in \mathscr{P}_r(K) \,\,\forall K\in \mathcal{T}_h},\\ \label{def:Qh}
   Q_h &:= Y_h \cap L^1_{0,x}.
\end{align} 
\end{subequations}

In order to control best approximations one needs to find a systematic way to construct discrete functions that are close to the target function.  Additionally, so far we have measured the pressure components in its natural norms, cf.~\eqref{eq:StochPresBoth}. However, these norms depend on the discretization parameter~$h$. It is customary to find a uniform regularity measure. 

Fortunately, both concerns can be addressed with the help of projection operators. Following~\cite{BelenkiBerselliDieningRozicka2012} it is sufficient to ask for the following abstract conditions:
\begin{assumption}\label{ass:VelocityProj}
We assume that $\mathscr{P}_1(\mathcal{T}_h) \subset X_h$ and there exists a linear projection $\Theta^h_{\Div} : L^2_x \cap W^{1,p}_{x} \to X_h$ which
\begin{enumerate}
\item preserves divergence in the $Y_h^*$-sense, i.e., for all $u \in L^2_x \cap W^{1,p}_{x}$ and $q \in Y_h$,
\begin{align}
\left( \Div u, q \right) = \left( \Div \Theta_{\Div}^h u, q \right)
\end{align}
\item preserves zero boundary values, i.e., $\Theta_{\Div}^h\big( L^2_x \cap W^{1,p}_{0,x} \big) \subset V_h$,
\item is locally $W^{1,1}_x$-stable in the sense that, for all $u \in L^2_x \cap W^{1,p}_x$ and $K \in \mathcal{T}_h$,
\begin{align}
\dashint_{K} \abs{\Theta_{\Div}^h u} \dd x \leq c \dashint_{\omega_K} \abs{u} \dd x + c \dashint_{\omega_K} h_K\abs{\nabla u} \dd x.
\end{align}
\end{enumerate}
\end{assumption}

\begin{assumption}\label{ass:PressureProj}
We assume that~$Y_h$ contains the constant functions, i.e., $\mathbb{R} \subset Y_h$, and there exists a linear projection~$\Theta_{\pi}^h: L^1_x \to Y_h$ which is locally $L^1_x$-stable in the sense that for all $q\in L^1_x$ and $K \in \mathcal{T}_h$ 
\begin{align}
\dashint_{K} \abs{\Theta_{\pi}^h q} \dd x \leq c \int_{\omega_K} \abs{q} \dd x.
\end{align}
\end{assumption}

\begin{remark}
There are many finite element spaces~$(V_h,Q_h)$ that satisfy Assumption~\ref{ass:VelocityProj} and Assumption~\ref{ass:PressureProj}, e.g., Taylor-Hood, MINI and conforming Crouzeix-Raviart element spaces. However, these spaces are not exactly divergence-free. 

Fortunately, there are also exactly divergence-free finite elements spaces that satisfy Assumption~\ref{ass:VelocityProj} and Assumption~\ref{ass:PressureProj}, e.g., Guzmán-Neilan~\cite{MR3120580,MR3269433} and Scott-Vogelius~\cite{MR813691,MR3882274} element spaces.  

For more details we refer to~\cite[Example~2.26 and Example~2.28]{tscherpel2018finite}.
\end{remark}

The next lemma can be derived along the lines of~\cite[Lemma~4.1]{BelenkiBerselliDieningRozicka2012}.
\begin{lemma}[Inf-sup condition]
Let Assumption~\ref{ass:VelocityProj} be satisfied. Then there exists a constant $c > 0$ such that for all $q \in Q_h$
\begin{align} \label{eq:Infsup}
\norm{q}_{L^{p'}_x} &\leq c \norm{q}_{Q_{\mathrm{det}}} \quad \text{ and } \quad 
\norm{q}_{L^2_x}  \leq c \norm{q}_{Q_{\mathrm{sto}}}.
\end{align}
\end{lemma}

%\begin{remark}
%Considering the regularity of the stochastic pressure it is more reasonable to require a stronger inf-sup condition than~\eqref{eq:Infsup} for the stochastic norm.
%
%If discrete pressures have gradients, i.e., $Q_h \subset W^{1,\infty} \cap L^\infty_0$, a natural choice is 
%\begin{align} \label{eq:InfsupGrad}
%\norm{\nabla q}_{L^2_x} \leq c \norm{q}_{\mathrm{Q}_{\mathrm{sto}}}.
%\end{align}
%However, due to the alternative representation of the stochastic norm, cf.~\eqref{eq:PressureNormAlter}, inequality~\eqref{eq:InfsupGrad} can only be satisfied if $ \nabla Q_h \subset V_{h,\Div}^\perp $. In this case one finds
%\begin{align*}
%\norm{q}_{\mathrm{Q}_{\mathrm{sto}}} = \norm{\Pi_{\Div}^\perp \nabla q}_{L^2_x} = \norm{\nabla q}_{L^2_x}. 
%\end{align*}
%\end{remark}

\begin{lemma} \label{lem:StabilityProj}
Let Assumption~\ref{ass:VelocityProj} and Assumption~\ref{ass:PressureProj} be satisfied. Then
\begin{enumerate}
\item \label{it:linearEstimate}(linear estimates) for all $u \in W^{2,2}_x$
\begin{align} \label{eq:linearInterpolationEst}
\norm{u - \Theta_{\Div}^h u}_{L^2_x} + h\norm{\nabla( u - \Theta_{\Div}^h u)}_{L^2_x} &\lesssim h^2 \norm{\nabla^2 u}_{L^2_x},
\end{align}
and for all $r \in [1,\infty)$ and $q \in L^r_x$
\begin{align}
\norm{q - \Theta_\pi^h q}_{L^r_x} \lesssim h\norm{\nabla q}_{L^r_x}.
\end{align}
\item \label{it:nonlinearEstimate}(non-linear estimate) for all $V(\varepsilon u) \in W^{1,2}_x$
\begin{align}\label{eq:nonlinearInterpolationEst}
\norm{V(\varepsilon u) -  V(\varepsilon \Theta_{\Div}^h u)}_{L^2_x} \lesssim h \norm{\nabla V(\varepsilon u)}_{L^2_x}.
\end{align}
\end{enumerate}
\end{lemma}
\begin{proof}
Part~\ref{it:linearEstimate} is a standard result. The nonlinear estimate~\eqref{eq:nonlinearInterpolationEst} is derived in~\cite[Theorem~5.1]{BelenkiBerselliDieningRozicka2012}.
\end{proof}

From now on we will assume that Assumption~\ref{ass:VelocityProj} and Assumption~\ref{ass:PressureProj} are satisfied.

\subsection{Initial error}
From a mathematical perspective the optimal choice as an initial condition of algorithm~\eqref{eq:velocity001} is the $L^2_x$-projection of the analytic initial condition, i.e., $u_0^h = \Pi_{\Div} u_0$. Then the initial error does not influence the error at later times since $C_{\mathrm{init}} = 0$. In general, it is a non-trivial task to compute the $L^2_x$-projection since one needs to evaluate integrals exactly. The discretization of the integrals, e.g. by quadrature rules or random evaluations, introduces an additional error. Whether the error can be bounded in terms of regularity of the initial condition needs to be determined. 

\begin{lemma} \label{lem:Initproj}
Let $\Pi_{\mathrm{init}}^h: L^2_x \to V_h$ be a linear projection such that for all $v\in L^2_x$
\begin{align*}
\norm{\Pi_{\mathrm{init}}^h v}_{L^2_x} \leq C \norm{v}_{L^2_x}.
\end{align*}
Then it holds
\begin{align*}
\norm{\Pi_{\mathrm{init}}^h v - \Pi_{\Div} v}_{L^2_x} \leq 2C \inf_{q \in V_{h,\Div}} \norm{v - q}_{L^2_x}.
\end{align*}
\end{lemma}
\begin{proof}
Since $\Pi_{\Div}$ and $\Pi_{\mathrm{init}}^h $ are linear projections, for arbitrary $q \in V_{h,\Div}$,
\begin{align*}
\norm{\Pi_{\mathrm{init}}^h v - \Pi_{\Div} v}_{L^2_x}  &= \norm{ \Pi_{\mathrm{init}}^h\left[v - q -  \Pi_{\Div}(v - q) \right]}_{L^2_x} \\
&\leq C \norm{v - q -  \Pi_{\Div}(v - q) }_{L^2_x} \\
&\leq 2C \norm{v - q}_{L^2_x}.
\end{align*}
The assertions follows by taking the infimum over~$q \in V_{h,\Div}$.
\end{proof}

\begin{corollary}
Let Assumption~\ref{ass:VelocityProj} be satisfied, $u_0 \in L^2_\omega W^{1,2}_{0,\Div}$ and $u_0^h = \Pi_{\mathrm{init}}^h u_0$ for some~$\Pi_{\mathrm{init}}^h$ given by Lemma~\ref{lem:Initproj}. Then
\begin{align*}
C_{\mathrm{init}} \lesssim h^2 \mathbb{E} \left[ \norm{\nabla u_0}_{L^2_x}^2 \right].
\end{align*}
\end{corollary}

\subsection{Noise coefficient}
Convergence of the velocity can only be guaranteed if the approximated noise coefficient~$G_n$ converges to the limiting noise coefficient~$G$. We quantify the convergence in the following sense.

\begin{assumption}[Asymptotic data exactness]
 \label{ass:DataLimit}
 We assume there exists $C >0$ (independent of $\tau$ and equivalently of $N$) such that for all progressively measurable $v \in L^2_\omega C_t L^2_x$ it holds
\begin{enumerate}
\item (averaged Lipschitz continuity)
\begin{align} \label{eq:AveragedLip}
\begin{aligned}
&\mathbb{E}\left[ \sum_{n=1}^N \int a_n^2(t) \norm{G\big(t, v(t) \big) - G\big(t, \mean{v}_{(n-2)\vee 0} \big)}_{L_2(\mathfrak{U};L^2_x)}^2 \dd t \right] \\
&\hspace{3em} \leq C \mathbb{E}\left[ \sum_{n=1}^N \int a_n^2(t) \norm{v(t) - \mean{v}_{(n-2)\vee 0}}_{L^2_x}^2 \dd t \right],
\end{aligned}
\end{align}
\item (averaged time regularity)
\begin{align}\label{eq:AveragedTime}
\begin{aligned}
&\mathbb{E}\left[ \sum_{n=1}^N \int a_n^2(t) \norm{G\big(t, \mean{v}_{(n-2)\vee 0} \big) - G_n\big(\mean{v}_{(n-2)\vee 0} \big)}_{L_2(\mathfrak{U};L^2_x)}^2 \dd t \right] \\
&\hspace{3em} \leq C \tau \left( 1 + \norm{v}_{L^2_\omega L^\infty_t L^2_x}^2 \right).
\end{aligned}
\end{align}
\end{enumerate}
\end{assumption}

\begin{lemma}
Let Assumption~\ref{ass:DataApprox} and Assumption~\ref{ass:DataLimit} be satisfied. Then it holds
\begin{align*}
C_{G}  &\lesssim \tau \left( \norm{u}_{B^{1/2}_{2,\infty} L^2_\omega L^2_x}^2 + \norm{u}_{L^2_\omega L^\infty_t L^2_x}^2 + \norm{u_0}_{L^2_\omega L^2_x}^2  + C_{\mathrm{init}} \right) + \mathbb{E}\left[\sum_{n=1}^N \tau \norm{e_n}_{L^2_x}^2  \right].
\end{align*}
\end{lemma}

\begin{proof}
Recall that~$G$ is a function in the variables~$(\omega,t,v) \in \Omega \times I \times \mathbb{R}^n$. Additionally, we compose $G$ with a stochastic process in its third variable. Measuring closeness of the data approximation jointly with respect to the variables is a hard task. Instead, we split the error into an approximation error with respect to~$t$ and~$v$ separately,
\begin{align*}
C_{G} &\lesssim \mathbb{E}\left[\sum_{n=1}^N \int a_n^2(t) \norm{G\big(t, u(t) \big)-G\big(t, \mean{u}_{(n-2)\vee 0} \big) }_{L_2(\mathfrak{U};L^2_x)}^2 \dd t \right]  \\
&\quad + \mathbb{E}\left[\sum_{n=1}^N \int a_n^2(t) \norm{G\big(t, \mean{u}_{(n-2)\vee 0} \big) - G_n(\mean{u}_{(n-2)\vee 0} )}_{L_2(\mathfrak{U};L^2_x)}^2 \dd t \right] \\
&\quad + \mathbb{E}\left[\sum_{n=1}^N \tau \norm{G_n(\mean{u}_{(n-2)\vee 0} ) - G_n(u_{(n-2)\vee 0} )}_{L_2(\mathfrak{U};L^2_x)}^2  \right] \\
&=: C_G^1 + C_G^2 + C_G^3.
\end{align*}

The first and last term measure closeness of $G$ respectively $G_n$ with respect to the third variable. The second term quantifies the distance of $G_n(\cdot) $ and $\int a_n(t) G(t, \cdot) \dd t$ in the second variable.

Due to~\eqref{eq:AveragedLip} and a bound on the temporal oscillations as presented in~\eqref{eq:EstimateTemporalOscillation} we arrive at
\begin{align*}
C_G^1 &\lesssim \mathbb{E}\left[ \sum_{n=1}^N \int a_n^2(t) \norm{u(t) - \mean{u}_{(n-2)\vee 0}}_{L^2_x}^2 \dd t \right]\\
&\lesssim \tau \left( \norm{u}_{B^{1/2}_{2,\infty} L^2_\omega L^2_x}^2 + \norm{u}_{L^2_\omega L^\infty_t L^2_x}^2 + \norm{u_0}_{L^2_\omega L^2_x}^2  \right).
\end{align*}

Notice that $\mean{u}_{(n-2)\vee 0}$ is $\mathcal{F}_{(n-2) \vee 0}^N$-measurable. Thus, we can apply~\eqref{ass:Lipschitz} to obtain
\begin{align*}
C_G^3 &\lesssim \mathbb{E}\left[\sum_{n=1}^N \tau \norm{\mean{u}_{(n-2)\vee 0}  -u_{(n-2)\vee 0} }_{L_2(\mathfrak{U};L^2_x)}^2  \right] \\
&\lesssim \tau\left( \norm{u_0}_{L^2_\omega L^2_x}^2 +  C_{\mathrm{init}} \right) + \mathbb{E}\left[\sum_{n=1}^N \tau \norm{\mean{u}_{n}  -u_{n} }_{L_2(\mathfrak{U};L^2_x)}^2  \right].
\end{align*}

Using~\eqref{eq:AveragedTime} we find
\begin{align*}
C_G^2 \lesssim \tau \left( 1 + \norm{u}_{L^2_\omega L^\infty_t L^2_x}^2 \right).
\end{align*}
Combining the estimates finishes the proof.
\end{proof}

\subsubsection{Example}
In this example we present a sufficient condition for the noise coefficient~$G$ and a possible choice for $G_n$ such that the Assumptions~\ref{ass:DataApprox} and~\ref{ass:DataLimit} are satisfied.

\begin{example} \label{ex:GandGn}
 Let $G$ satisfy~\eqref{ex:Sublinear},~\eqref{ex:Lipschitz} and additionally, for all $\omega \in \Omega$, $t,s \in I$ and $v \in L^2_x$,
\begin{align}\label{ex:timeReg}
\norm{G(\omega,t, v) - G(\omega,s, v)}_{L_2(\mathfrak{U};L^2_x)}^2 \lesssim \abs{t-s} \left( \norm{v}_{L^2_x}^2 + 1 \right).
\end{align}
Next, we construct a time discrete approximation of $G$. Let $\omega \in \Omega$, $v \in L^2_x$ and set
\begin{align*}
n = 1,2: \quad G_n(\omega,v) &:= 0, \\
n \geq 3: \quad G_n(\omega,v) &:= \dashint_{J_{n-2}} G(\omega,t,v) \dd t.
\end{align*}
\end{example}

\begin{lemma}
$G$ and $G_n$ satisfy Assumption~\ref{ass:DataApprox} and Assumption~\ref{ass:DataLimit}.
\end{lemma}
\begin{proof}
Ad Assumption~\ref{ass:DataApprox}: 
We restrict ourselves to the case $n \geq 3$ since the verification for $n = 1,2$ is trivial. 

Clearly, $G_n$ is $\mathcal{F}_{n-2}^N$-measurable as $G$ is progressively measurable with respect to~$(\mathcal{F}_t)$.

 Let $v,w \in L^2_\omega L^2_x$. For fixed $\omega \in \Omega$ we find, using Jensen's inequality and~\eqref{ex:Lipschitz},
\begin{align*}
&\norm{G_n\big(\omega, v(\omega) \big) - G_n\big(\omega, w(\omega) \big)}_{L_2(\mathfrak{U};L^2_x)}^2 \\
&\hspace{3em} \leq \dashint_{J_{n-2}} \norm{G\big(\omega,t, v(\omega)\big) - G\big(\omega,t, w(\omega)\big)}_{L_2(\mathfrak{U};L^2_x)}^2 \dd t\\
&\hspace{3em} \lesssim  \norm{v(\omega) -  w(\omega)}_{L^2_x}^2 .
\end{align*}
The Lipschitz estimate~\eqref{ass:Lipschitz} follows by applying expectation. In particular, it is satisfied for $\mathcal{F}_{n-2}^N$-measurable $v$ and $w$.

The sublinear growth estimate~\eqref{ass:Sublinear} follows analogously to the Lipschitz estimate.

Ad Assumption~\ref{ass:DataLimit}: The averaged (in probability and time) Lipschitz continuity~\eqref{eq:AveragedLip} follows immediately from the pointwise Lipschitz continuity~\eqref{ex:Lipschitz}.

Let $v \in L^2_\omega C_t L^2_x$. Then
\begin{align*}
&\mathbb{E}\left[ \sum_{n=1}^N \int a_n^2(t) \norm{G\big(t, \mean{v}_{(n-2)\vee 0} \big) - G_n\big(\mean{v}_{(n-2)\vee 0} \big)}_{L_2(\mathfrak{U};L^2_x)}^2 \dd t \right] \\
&\hspace{2em} = \mathbb{E}\left[ \sum_{n=1}^2 \int a_n^2(t) \norm{G\big(t, v(0) \big)}_{L_2(\mathfrak{U};L^2_x)}^2 \dd t \right] \\
&\hspace{3em} + \mathbb{E}\left[ \sum_{n=3}^N \int a_n^2(t) \norm{G\big(t, \mean{v}_{n-2} \big) - G_n\big(\mean{v}_{n-2} \big)}_{L_2(\mathfrak{U};L^2_x)}^2 \dd t \right] \\
&\hspace{2em} =: \mathfrak{H}_1 + \mathfrak{H}_2.
\end{align*}
Using~\eqref{ex:Sublinear} and the continuity of $v$
\begin{align*}
\mathfrak{H}_1 \lesssim \mathbb{E}\left[ \sum_{n=1}^2 \int a_n^2(t) \left( \norm{v(0)}_{L^2_x}^2 + 1\right) \dd t \right] \lesssim \tau \left( 1 + \norm{v}_{L^2_\omega L^\infty_t L^2_x}^2 \right).
\end{align*}
By Jensen's inequality and~\eqref{ex:timeReg}
\begin{align*}
\mathfrak{H}_2 &\leq  \mathbb{E}\left[ \sum_{n=3}^N \int a_n^2(t) \dashint_{J_{n-2}} \norm{G\big(t, \mean{v}_{n-2} \big) - G\big(s, \mean{v}_{n-2} \big)}_{L_2(\mathfrak{U};L^2_x)}^2 \dd s \dd t \right] \\
&\lesssim \mathbb{E}\left[ \sum_{n=3}^N \int a_n^2(t) \dashint_{J_{n-2}} \abs{t-s} \left( \norm{\mean{v}_{n-2}}_{L^2_x}^2 + 1 \right) \dd s \dd t \right].
\end{align*}
Notice that
\begin{align*}
\mathbb{E}\left[ \sum_{n=3}^N \int a_n^2(t) \dashint_{J_{n-2}} \abs{t-s} \dd s \dd t \right] \lesssim \tau.
\end{align*}
Moreover, another application of Jensen's inequality shows,
\begin{align*}
&\mathbb{E}\left[ \sum_{n=3}^N \int a_n^2(t) \dashint_{J_{n-2}} \abs{t-s} \norm{\mean{v}_{n-2}}_{L^2_x}^2 \dd s \dd t \right] \\
&\hspace{3em}\leq \tau \mathbb{E}\left[ \sum_{n=3}^N \int a_n^2(t) \dashint_{J_{n-2}}  \norm{v(s)}_{L^2_x}^2 \dd s \dd t \right] \\
&\hspace{3em}\lesssim \tau \norm{v}_{L^2_\omega L^\infty_t L^2_x}^2.
\end{align*}
Collecting the estimates for~$\mathfrak{H}_1$ and~$\mathfrak{H}_2$ allows us to conclude~\eqref{eq:AveragedTime}.

\end{proof}

\begin{remark}
We want to stress that~$G_n$ additionally satisfies Assumption~\ref{ass:StrongerCondition}.
\end{remark}

\subsection{Projection error}
The projection error is not an obstacle in the derivation of convergence rates and can be treated easily.
\begin{lemma} \label{lem:ProjectionErrorBound}
It holds
\begin{align} \label{eq:estimateProjection01}
C_{L^\infty} \lesssim h^2 \norm{\nabla u}_{L^2_\omega L^\infty_t L^2_x}^2
\end{align}
and
\begin{align}\label{eq:estimateProjection02}
C_{B^{1/2}_{2,\infty}} \lesssim h^2 \seminorm{\nabla u}_{B^{1/2}_{2,\infty} L^2_\omega L^2_x}^2.
\end{align}
\end{lemma}
\begin{proof}
Since $\Pi_{\Div}$ is the orthogonal projection onto $V_{h,\Div}$,~\eqref{eq:linearInterpolationEst} and Jensen's inequality, we find
\begin{align*}
\norm{\mean{u}_n - \Pi_{\Div} \mean{u}_n}_{L^2_x}^2 &\leq \norm{\mean{u}_n - \Theta_{\Div}^h \mean{u}_n}_{L^2_x}^2 \\
&\lesssim h^2 \norm{\nabla \mean{u}_n}_{L^2_x}^2 \lesssim h^2 \sup_{t \in I} \norm{\nabla u}_{L^2_x}^2.
\end{align*}
Similarly,
\begin{align*}
\norm{\mean{u}_n  - \mean{u}_{n-k} - \Pi_{\Div}(\mean{u}_n - \mean{u}_{n-k})}_{L^2_x}^2 \lesssim h^2 \dashint_{J_n} \norm{\nabla u(t) - \nabla u(t-k\tau) }_{L^2_x}^2 \dd t.
\end{align*}
The assertions~\eqref{eq:estimateProjection01} and~\eqref{eq:estimateProjection02} follow immediately.

\end{proof}

\subsection{Best approximation}
The distance measure for the best-approximation consists of two contributions, 
\begin{align*}
d^1(\eta) &=  \mathbb{E}\left[\sum_{n=1}^N \norm{\Pi_{\Div} \mean{u}_n - \eta_n}_{L^2_x}^2\right], \\
 d^2(\eta) &= \mathbb{E}\left[\sum_{n=1}^N \int a_n(t)  \norm{V\big( \varepsilon u(t) \big) - V(\varepsilon \eta_n)}_{L^2_x}^2 \dd t\right].
\end{align*}
The best-approximation optimize both contributions simultaneously. We choose specific discrete velocities in order to understand how the best-approximation behaves as a function of the discretization parameters. There are essentially two approaches
\begin{enumerate}
\item ($L^2_x$-projection) $\eta_n = \Pi_{\Div} \mean{u}_n$,
\item (local projection) $\eta_n = \Theta_{\Div}^h \mean{u}_n$.
\end{enumerate}
The $L^2_x$-projection has the clear advantage that the first contribution vanishes. However, it is by far non-trivial to understand the approximation quality of the $L^2_x$-projection for the non-linear gradient. In the context of the $p$-Laplace system, we recently proved a corresponding gradient estimate, cf.~\cite[Theorem~7]{MR4286257}. The result heavily relies on weighted gradient stability for the $L^2_x$-projection, which itself relies on the underlying structure of the finite element space and its triangulation. For more details we refer to~\cite{MR4320894} and the references therein. 

On the other hand, local projections satisfy a non-linear gradient stability as presented in~\eqref{eq:nonlinearInterpolationEst}. Therefore, the second contribution can be controlled. However, now we need to additionally estimate the first contribution.

We discuss the estimates for each choice separately.

\subsubsection{$L^2_x$-projection}
Provided that the $L^2_x$-projection satisfies a non-linear gradient stability it is straight forward to bound the best-approximation error. The difficult part is to verify the stability~\eqref{eq:GradStabL2}.

\begin{lemma} \label{lem:L2ProjEst}
Let there exists a constant $C$ such that for all $ V(\varepsilon u) \in W^{1,2}_x$ it holds
\begin{align}\label{eq:GradStabL2}
\norm{V(\varepsilon u) - V(\varepsilon \Pi_{\Div} u)}_{L^2_x} \leq C h\norm{\nabla V(\varepsilon u)}_{L^2_x}.
\end{align}
Then
\begin{align*}
C_{\mathrm{best}} \lesssim h^2 \left( \norm{\nabla V(\varepsilon u)}_{L^2_\omega L^2_t L^2_x}^2 \right) + \tau \seminorm{V(\varepsilon u)}_{B^{1/2}_{2,\infty} L^2_\omega L^2_x}^2. 
\end{align*}
\end{lemma}

\begin{proof}
Clearly, $d^1\big( \Pi_{\Div} \mean{u}\big) = 0$. Thus, it remains to bound~$d^2$.

The relation between~$V$ and $S$, cf. Lemma~\ref{lem:Relation}, and generalized Young's inequality, cf. Lemma~\ref{lem:GeneralizedYoung}, imply
\begin{align*}
&d^2\big( \Pi_{\Div} \mean{u}\big) \\
&\eqsim \mathbb{E}\left[\sum_{n=1}^N \int a_n(t) \dashint_{J_n} \int_{\mathcal{O}} \left(S\big( \varepsilon u(t) \big) - S(\varepsilon \Pi_{\Div} \mean{u}_n) \right) :\left( \varepsilon u(t) - \varepsilon \Pi_{\Div} u(s) \right) \dd x \dd s\dd t\right] \\
&= \mathbb{E}\left[\sum_{n=1}^N \int a_n(t) \dashint_{J_n} \int_{\mathcal{O}} \left(S\big( \varepsilon u(t) \big) - S(\varepsilon \Pi_{\Div} \mean{u}_n) \right) :\left( \varepsilon u(t) - \varepsilon u(s)  \right) \dd x \dd s\dd t\right] \\
&\quad + \mathbb{E}\left[\sum_{n=1}^N \int a_n(t) \dashint_{J_n} \int_{\mathcal{O}} \left(S\big( \varepsilon u(t) \big) - S\big( \varepsilon u(s) \big) \right) :\left( \varepsilon u(s) - \varepsilon \Pi_{\Div} u(s) \right) \dd x \dd s\dd t\right] \\
&\quad+ \mathbb{E}\left[\sum_{n=1}^N \int_{J_n} \int_{\mathcal{O}} \left(S\big( \varepsilon u(s) \big) - S(\varepsilon \Pi_{\Div} \mean{u}_n) \right) :\left( \varepsilon u(s) - \varepsilon \Pi_{\Div} u(s) \right) \dd x \dd s \right] \\
&\leq \delta \mathbb{E}\left[\sum_{n=1}^N \int a_n(t)  \norm{V\big( \varepsilon u(t) \big) - V(\varepsilon \Pi_{\Div} \mean{u}_n)}_{L^2_x}^2 \dd t\right] \\
&\quad + c_\delta \mathbb{E}\left[\sum_{n=1}^N \int a_n(t) \dashint_{J_n} \norm{V\big( \varepsilon u(t) \big) - V\big( \varepsilon u(s) \big)}_{L^2_x}^2 \dd s \dd t\right] \\
&\quad + c_\delta \mathbb{E}\left[ \int_{I} \norm{V\big( \varepsilon u(s) \big) - V\big( \varepsilon \Pi_{\Div} u(s) \big)}_{L^2_x}^2 \dd s\right].
\end{align*}
The first term can be absorbed to the left hand side. The second term can be treated as in Lemma~\ref{lem:RegulartiyV}. The last term is bounded by assumption. 
\end{proof}

\subsubsection{Local projection}
Local projections are stable even in the non-linear setting, cf.~\eqref{eq:nonlinearInterpolationEst}. This immediately provides an upper bound for the non-linear gradient error. However, since~$d^1$ no longer vanishes, it needs to be considered. It leads to a severe increase of regularity requirements for the velocity as well as a coupling condition between the temporal and spatial discretization parameter. 

\begin{lemma} It holds
\begin{align*}
C_{\mathrm{best}} \lesssim h^2 \left( \frac{h^2}{\tau} \norm{\nabla^2 u}_{L^2_\omega L^\infty_t L^2_x}^2 + \norm{\nabla V(\varepsilon u)}_{L^2_\omega L^2_t L^2_x}^2 \right) + \tau \seminorm{V(\varepsilon u)}_{B^{1/2}_{2,\infty} L^2_\omega L^2_x}^2. 
\end{align*}
\end{lemma}
\begin{proof}
The proof follows along the lines of Lemma~\ref{lem:ProjectionErrorBound} and Lemma~\ref{lem:L2ProjEst}.
\end{proof}

\subsection{Main result}
A combination of all the previous results finally allows us to conclude the following quantified convergence rate for specific discretizations.

\begin{theorem}[Convergence rates] \label{thm:Convergence}
Let the following be true:
\begin{itemize}
\item $(V_h,Q_h)$ satisfy Assumption~\ref{ass:VelocityProj} and Assumption~\ref{ass:PressureProj},
\item $V_{h,\Div}$ contains only exactly divergence free vector fields,
\item $G$ and $G_n$ satisfy Assumption~\ref{ass:DataApprox} and Assumption~\ref{ass:DataLimit},
\item $u_0 \in L^2_\omega W^{1,2}_{x}$ and $u_0^h = \Pi_{\mathrm{init}}^h u_0$ for some~$\Pi_{\mathrm{init}}^h$ given by Lemma~\ref{lem:Initproj},
\item $u \in L^2_\omega L^\infty_t W^{1,2}_x \cap B^{1/2}_{2,\infty} L^2_\omega W^{1,2}_x$,
\item $V(\varepsilon u) \in L^2_\omega L^2_t W^{1,2}_x \cap B^{1/2}_{2,\infty} L^2_\omega L^2_x$.
\end{itemize}
Additionally, let either of the following be true:
\begin{enumerate}
\item \label{it:athm} There exists a constant $C$ such that for all $ V(\varepsilon u) \in W^{1,2}_x$ it holds
\begin{align*}
\norm{V(\varepsilon u) - V(\varepsilon \Pi_{\Div} u)}_{L^2_x} \leq C h\norm{\nabla V(\varepsilon u)}_{L^2_x},
\end{align*}
\item \label{it:bthm} or $h^2 \lesssim \tau$ and $u \in L^2_\omega L^\infty_t W^{2,2}_x$.
\end{enumerate}
 Then the velocity approximation~$u_n$ generated by~\eqref{eq:velocity001} satisfies
\begin{align}
\begin{aligned}
&\mathbb{E}\left[ \max_{n} \norm{\mean{u}_n - u_n}_{L^2_x}^2 + \sum_{n} \int a_n(t) \norm{V\big( \varepsilon u(t) \big) - V(\varepsilon u_n)}_{L^2_x}^2 \dd t \right] \\
&\hspace{3em} +\max_{k \in [N]} \frac{1}{k} \sum_{n=k}^N \mathbb{E}\left[ \norm{(\mean{u}_n - u_n) - (\mean{u}_{n-k} - u_{n-k}) }_{L^2_x}^2 \right] \lesssim h^2 + \tau.
\end{aligned}
\end{align}
\end{theorem}

\begin{remark}
It is shown in~\cite{MR4022286} (actually they consider more complex systems but the arguments carry over) and~\cite{2022arXiv220902796W} that the velocity, for $p \geq 2$, satisfies
\begin{subequations}
\begin{align}
u &\in L^2_\omega L^\infty_t W^{1,2}_x \cap L^2_\omega B^{1/2}_{\Phi_2,\infty} L^2_x, \\
V(\varepsilon u) &\in L^2_\omega L^2_t W^{1,2}_x \cap L^2_\omega B^{1/2}_{2,\infty} L^2_x.
\end{align}
\end{subequations}

Moreover, if additionally $\kappa > 0$, it is possible to transfer regularity from~$V(\varepsilon u)$ to~$\nabla u$. Indeed, similarly to~\eqref{eq:RelationGeq2} one finds, for all $A, B \in \mathbb{R}^{n \times n}$,
\begin{align*}
\abs{A - B}^2 \leq \kappa^{2-p}\abs{V(A) - V(B)}^2.
\end{align*} 
This together with Korn's inequality show
\begin{align*}
V(\varepsilon u ) \in L^2_\omega B^{1/2}_{2,\infty} L^2_x \quad  \Rightarrow \quad \nabla u \in L^2_\omega B^{1/2}_{2,\infty} L^2_x.
\end{align*}
Thus, the first regularity requirements of Theorem~\ref{thm:Convergence} are met. 

Unfortunately, it is non-trivial to verify either~\ref{it:athm} or~\ref{it:bthm}. Both require a sophisticated analysis which is not yet available.
\end{remark}
%
%\section{Conclusion}
%List of open problems
%\begin{enumerate}
%\item Non-linear stability of $L^2$-projection onto discretely divergence free vector fields
%\item question on regularity for pressure 
%\item approximation of pressure
%\item $B^{1/2}_{2,\infty} W^{1,2}_x$ discussion; why it is satisfied in the linear case; remark to dual estimates for time derivative.
%\end{enumerate}
%\red{do we need conclusion? } 

\section{Numerical simulations} \label{sec:Numerical-simulation}
In this section we report on numerical experiments that test the need of our theoretically imposed assumptions. We are mainly interested on the influence of (non)-exactly divergence free velocity approximations on the time stability and convergence of our algorithm:
\begin{quote}
Do non-exactly divergence free velocity approximations behave worse than exactly divergence free ones?
\end{quote}
We investigate this question by choosing specific spatial discretisations that are either exactly or discretely divergence-free while keeping all other parameters fixed. 

The code is available at \href{https://github.com/joernwichmann/Stochastic_pStokes}{https://github.com/joernwichmann/Stochastic\_pStokes}. The implementation uses the open-source finite element package \textit{Firedrake}~\cite{FiredrakeUserManual}, which itself heavily relies on \textit{PETSc}~\cite{petsc-web-page}. 

\subsection{Method}
In this section we explain the methodology of our numerical simulation. It contains model configuration, discretisation aspects and details on the data processing of solutions. The section closes with a summary, presented in terms of pseudo-code, of all steps used in our numerical study.

\subsubsection{Model}
We used the following model configuration:
\begin{itemize}
\item two-dimensional unit square $\mathcal{O} = (0,1)^2$;
\item final time $T=1$;
\item initial condition $u_0(x,y) =  \begin{pmatrix}
x^2(1-x)^2(2-6y+4y^2)y \\
-y^2(1-y)^2(2-6x+4x^2)x
\end{pmatrix} $;
\item one-dimensional Brownian motion $\beta$ which generates the cylindrical Wiener process $W = \beta$ for $\mathfrak{U} = \mathbb{R}$;  
\item noise coefficient $G(u) = \lambda u + g$ with intensity $\lambda = 1$ and shift~$g = u_0$;
\item viscous stress tensor $ S(A) = (0.1 + \abs{A}^2)^{(p-2)/2} A$.
\end{itemize}

\subsubsection{Spatial discretisation}
\paragraph{Discretely divergence free}~We choose a mixed spatial discretisation consisting of $\mathscr{P}_2$--$\mathscr{P}_1$ continuous velocity and continuous pressure approximate spaces, so-called Taylor--Hood elements~\cite{Taylor1973}; that is velocity and pressure approximate spaces are given by~\eqref{def:Xh} and~\eqref{def:Yh} with $r = 2$ and $r = 1$, respectively. Taylor-Hood elements are merely discretely divergence free; thus, generally violating the divergence free constraint. 

\paragraph{Exactly divergence free}~We choose Scott--Vogelius elements~\cite{MR813691}. Here, the velocity approximate space consists of $\mathscr{P}_k$ continuous elements; the pressure approximate space is defined by $\mathscr{P}_{k-1}$ discontinuous elements that additionally satisfy a constraint on singular vertices. This pairing is \textit{inf-sup} stable for $k\geq 4$,~\cite{MR3882274}; and $k=3$ on non-singular meshes~\cite{Guzmn2018}. We won't elaborate on the notion of singular meshes. Instead, we refer to~\cite{Guzmn2018} for details. Scott--Vogelius elements are exactly divergence free; in other words, they capture the analytic constraint on the discrete level. We choose $k=2$ which enables the fair comparison between Taylor--Hood and Scott--Vogelius elements. This choice is unproven of being stable but we couldn't detect any numerical instability in our simulations.  

\begin{figure}
\centering
\includegraphics[scale=0.25]{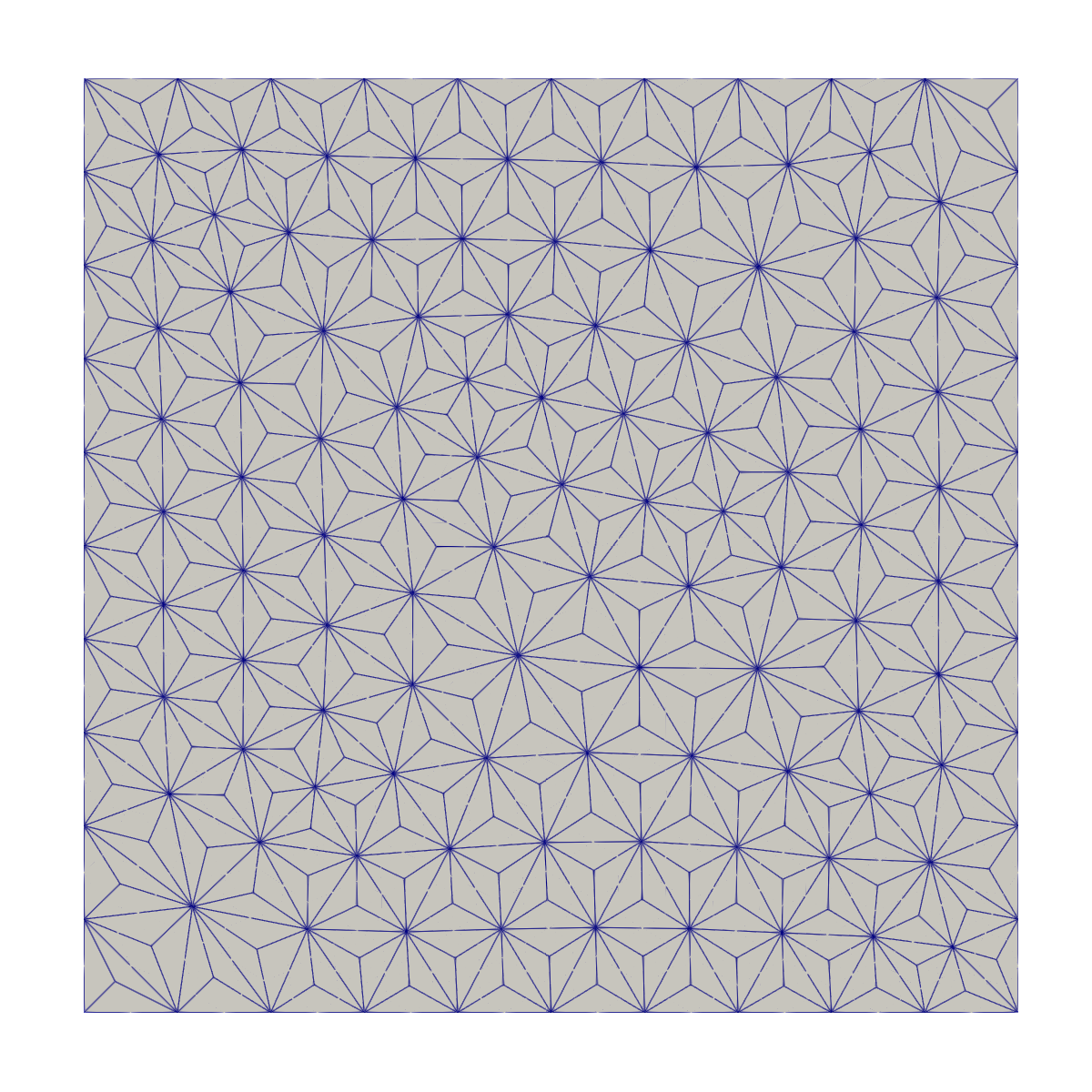}
\vspace*{-3em}
\caption{Non-singular mesh of the unit square.}
\label{fig:non-singularMesh}
\end{figure}

In Figure~\ref{fig:non-singularMesh} we present the mesh that is used in our simulations for both, Scott--Vogelius and Taylor--Hood elements. We remark that the mesh is non-singular. In particular, no singular vertex exists which removes the additional constraint of the Scott--Vogelius pressure space; it reduces to standard discontinuous, piecewise linear and mean-value free pressure approximates.   

\subsubsection{Temporal discretisation}
Let $N \in \mathbb{N}$ be the time resolution. We use an equi-distant time-grid and a time-stepping scheme. The details are described in Section~\ref{sec:time-disc}. 

We comment on some implementation aspects: the viscous stress tensor~$S$ is evaluated implicitly. This choice enabled the stability analysis of our algorithm but makes solving of the system~\eqref{eq:CoupledAlgo} more difficult. Our algorithm doesn't specify how to solve~\eqref{eq:CoupledAlgo} since we rely on the blackbox \textit{firedrake.solve}, which uses Newton's method for finding the solution. Efficient algorithms that find solutions to these non-linear problems are of individual interest.

\subsubsection{Sampling strategy}
We employ a Monte-Carlo approach for the discretisation of the probability space.

Let $M \in \mathbb{N}$ be the number of samples. Realisations of the random vector $(\Delta_n \mathbb{W}(\omega_m) )_{n=1}^{N}$, $m \in [M]$, where $\Delta_n \mathbb{W} = \mean{W}_n - \mean{W}_{n-1}$, are replaced by 
\begin{align} \label{eq:sample-random-inc}
Z_n^m \approx \Delta_n \mathbb{W}(\omega_m), \qquad n \in [N],\, m \in [M],
\end{align}
which are generated by a pseudo-random number generator and the sampling strategy proposed in~\cite[Section~4.3]{Diening2022}.

\subsubsection{Implemented algorithm}
In total, we implement the fully discrete system of equations:
\begin{itemize}
\item (Initialise) For all $m \in [M]$, we define $u_0^m = \Theta_h u_0 \in V_h$ by calling the \textit{firedrake.project} function~$\Theta_h$ that projects the analytic initial condition into the discrete velocity space~$V_h$;
\item (Time-stepping) For all $n \in [N]$ and $m \in [M]$, define $u_n^m \in V_h$ and $p_n^m \in Q_h$ by solving
\begin{subequations} \label{algo:Experiment-stepping}
\begin{align}
&\forall \xi \in V_h: \, \hspace{1em} \left( u_n^m, \xi \right)\hspace{-1pt} +\hspace{-1pt} \tau \left[ \left( S(\varepsilon u_n^m), \varepsilon \xi \right) - \left( p_n^m, \Div \xi\right) \right]\\ \nonumber
&\hspace{10.1em}  = \left( u_{n-1}^m, \xi \right)\hspace{-1pt} +\hspace{-1pt} \left[ \lambda \left(u_{(n-2)\vee 0}, \xi \right) + \left(\Theta_h g, \xi \right) \right] Z_n^m, \\
&\forall q \in Q_h: \, \hspace{1em} \left( \Div u_n^m, q \right) = 0.
\end{align}
\end{subequations}
\end{itemize}
Notice that, contrary to~\eqref{eq:CoupledAlgo}, we use the time-differentiated pressure variable $p$ instead of $\pi$; they are related by the identity $\pi_n^m - \pi_{n-1}^m = \tau p_n^m$. The pressure $\pi$ can be reconstructed from $p$ up to a constant: the initial pressure.

\subsubsection{Monitored statistics}
We exclusively trace the dependence of our algorithm on the time resolution. We monitor two statistics: stability and convergence.

Since we expect qualitative differences between Taylor--Hood and Scott--Vogelius elements for the evolution of velocity's divergence and pressure we investigate stability of the algorithm with respect to 
\begin{align} \label{eq:disc-to-stab}
N \mapsto \mathbb{E}\left[ \frac{1}{N} \sum_{n=1}^N \left( \norm{\Div u_n^N}_{L^2_x}^2 + \norm{p_n^N}_{L^2_x}^2 \right) \right],
\end{align}
where the super-script $N$ is used to indicate dependence on the time resolution.

Time convergence of the algorithm is measured with respect to
\begin{align} \label{eq:disc-to-error}
\begin{aligned}
N \mapsto &\mathbb{E}\left[ \max_{n\leq N} \norm{\mean{u}^N_n - u_n^N}_{L^2_x}^2 + \sum_{n=1}^N \int a^N_n(t) \norm{V\big( \varepsilon u(t) \big) - V(\varepsilon u_n^N)}_{L^2_x}^2 \dd t \right] \\
&\hspace{3em} +\max_{k \leq N} \frac{1}{k} \sum_{n=k}^N \mathbb{E}\left[ \norm{(\mean{u}_n^N - u_n^N) - (\mean{u}_{n-k}^N - u_{n-k}^N) }_{L^2_x}^2 \right],
\end{aligned}
\end{align}
and 
\begin{align}  \label{eq:disc-to-error-pres}
N \mapsto \mathbb{E}\left[\max_{ n\leq N} \norm{\pi(t_n^N) -  \pi_n^N}_{L^2_x}^2 \right].
\end{align}
The first mapping is motivated by Theorem~\ref{thm:Convergence}, which provides a rigorous analysis of the algorithm for Scott-Vogelius elements. Its evolution for Taylor-Hood elements is unclear. Identifying the natural function space for pressure convergence as well as quantifying its speed of convergence is an open field of research. Observation of the mapping~\eqref{eq:disc-to-error-pres} provides us with numerical evidence for pressure convergence.

Numerically,~\eqref{eq:disc-to-stab},~\eqref{eq:disc-to-error} and~\eqref{eq:disc-to-error-pres} are inaccessible for two reasons: 
\begin{itemize}
\item the expectation can only be computed approximately;
\item the analytic velocity~$u$ and pressure~$\pi$ are unknown;
\end{itemize}
Additionally, we want to remark that the time-discrete $B^{1/2}_{2,\infty} L^2_\omega L^2_x$-norm is computationally demanding, which motivated us to discard any further study of this quantity.

\subsubsection{Lifting vectors to functions}
The distance measures \eqref{eq:disc-to-error} and~\eqref{eq:disc-to-error-pres} additionally dependent on the time resolution~$N$. The dependence of the function spaces on the resolution can be removed by first lifting vectors to functions and afterwards comparing these functions.
 
Let $E$ be a vector space and $N \in \mathbb{N}$. Let $U=(u_n)_{n=0}^N$ with $u_n \in E$ for all $n \in [N_0]$. We define the piecewise constant interpolation by 
\begin{align*}
\mathscr{L}[U](t) &:= u_N + \sum_{n=1}^N (u_{n-1} - u_N) \chi_{J_n^N}(t),
\end{align*}
where $\chi$ denotes the indicator function, $J_n^N = [t_{n-1}^N,t_n^N)$ and $t_n^N = T \, n/N$.

\subsubsection{Comparison on different time scales} Instead of studying the distance of our approximation at a fixed time resolution and the (unknown) analytic solution, we replace the analytic solution by a more refined approximation.

Let $N_c$, $N_f \in \mathbb{N}$ be coarse, respectively, fine time discretisations; and $U(N_c) = (u^{N_c})$, $ U(N_f) = (u^{N_f})$ and $\Pi(N_c) = (\pi^{N_c})$, $\Pi(N_f) = (\pi^{N_f})$ be the corresponding coarse, respectively, fine velocity and pressure vectors. We define the following distances on vectors:
\begin{subequations} \label{def:choice-of-distance}
\begin{align}
\mathrm{d}_{L^\infty_t L^2_x}^{\mathrm{vel}}\big( U(N_c), U(N_f) \big) &:= \norm{\mathscr{L}[U(N_c)] - \mathscr{L}[U(N_f)] }_{L^\infty_t L^2_x}, \\
\mathrm{d}_{L^2_t V}^{\mathrm{vel}}\big( U(N_c), U(N_f) \big) &:= \norm{V(\varepsilon \mathscr{L}[U(N_c)] ) - V(\varepsilon \mathscr{L}[U(N_f)] ) }_{L^2_t L^2_x}, \\
\mathrm{d}_{L^\infty_t L^2_x}^{\mathrm{pre}}\big( \Pi( N_c) , \Pi(N_f )\big) &:= \norm{\mathscr{L}[\Pi(N_c)] - \mathscr{L}[\Pi(N_f)] }_{L^\infty_t L^2_x}.
\end{align}
\end{subequations}

\subsubsection{Empirical approximation of expectation}
We replace incomputable probabilistic mean-values by computable empirical approximations. 

Recall that $M$ denotes the number of samples. We define empirical stability measures by
\begin{subequations} \label{eq:statistical-stability}
\begin{align}
\mathrm{K}_{L^2_t L_{\Div}}^{\mathrm{vel},M}(N_c) &:= \left( \frac{1}{M} \sum_{m=1}^M \frac{1}{N_c} \sum_{n=1}^{N_c}  \norm{\Div U_n^m }_{L^2_x}^2 \right)^{1/2}, \\
\mathrm{K}_{W^{1,2}_t L^2_x}^{\mathrm{pre},M}(N_c) &:= \left( \frac{1}{M} \sum_{m=1}^M \frac{1}{N_c} \sum_{n=1}^{N_c} \norm{ p_n^m }_{L^2_x}^2 \right)^{1/2}.
\end{align}
\end{subequations}
For a given sample index $m \in [M]$ and time resolution $N \in \{N_c, N_f\}$, let $U_m(N)$ and $\Pi_m(N)$ denote approximate velocity and pressure vectors with resolution~$N$ generated by the $m$-th sample, respectively. We define empirical convergence measures by
\begin{subequations} \label{eq:statistical-mean}
\begin{align}
\mathrm{E}^{\mathrm{vel},M}_{L^\infty_t L^2_x}(N_c,N_f) &:= \left( \frac{1}{M} \sum_{m=1}^M \left( \mathrm{d}_{L^\infty_t L^2_x}^{\mathrm{vel}}\big( U_m(N_c), U_m(N_f) \big) \right)^2 \right)^{1/2}, \\
\mathrm{E}^{\mathrm{vel},M}_{L^2_t V}(N_c,N_f) &:= \left( \frac{1}{M} \sum_{m=1}^M \left( \mathrm{d}_{L^2_t V}^{\mathrm{vel}}\big( U_m(N_c), U_m(N_f) \big) \right)^2 \right)^{1/2}, \\
\mathrm{E}^{\mathrm{pre},M}_{L^\infty_t L^2_x}(N_c,N_f) &:= \left( \frac{1}{M} \sum_{m=1}^M \left( \mathrm{d}_{L^\infty_t L^2_x}^{\mathrm{pre}}\big( \Pi_m( N_c) , \Pi_m(N_f )\big) \right)^2 \right)^{1/2}.
\end{align}
\end{subequations}

\subsubsection{Simultaneous sampling on different time scales}
We need to ensure that coarse and fine approximates are computed on the same pseudo-random event. Notice that the pseudo-random numbers defined in~\eqref{eq:sample-random-inc} depend on the time resolution. In particular, different time resolutions cannot be sampled independently, as can be seen on the analytic level: the random vectors 
\begin{align*}
(\Delta_n \mathbb{W})_{n=1}^{N_c} \hspace{3em} \text{ and }  \hspace{3em} (\Delta_n \mathbb{W})_{n=1}^{N_f} 
\end{align*}
are correlated, requiring us to modify the generation of the pseudo-random numbers to account for dependencies on different time scales. If $N_f = r N_c$ for some $r \in \mathbb{N}$, we can write the coarse increments in terms of the fine ones, yielding a simple reconstruction rule~\cite[Lemma~37]{Diening2022}: For all $m \in [M]$
\begin{subequations}\label{eq:sample-reconst}
\begin{align} 
Z_1^{m,\mathrm{c}} &= \sum_{\ell=1}^r \left( 1 - \frac{\ell - 1}{r} \right) Z_{\ell}^{m,f}, \\ 
Z_n^{m,\mathrm{c}} &= \sum_{\ell=0}^{r-1} \frac{\ell+1}{r} Z_{rn - \ell}^{m,\mathrm{f}} +  \sum_{\ell=0}^{r-2} \left( 1 - \frac{\ell+1}{r} \right) Z_{r(n-1) -\ell}^{m,\mathrm{f}}, \qquad n \in \{2, \ldots, N_c\}.
\end{align}
\end{subequations}
In other words, generating the fine pseudo-random numbers fully determines the values of the coarse ones.

\subsubsection{Summary}
We are ready to present the algorithm used for the sample generation of approximate velocity and pressure, and their comparison on various function spaces. 
\begin{framed}
\begin{algorithm}[H]
\KwData{samples, time scales, space discretisation}
\KwResult{empirical study of stability and time convergence}
\For{sample in samples}{
generate pseudo-random vector on finest time scale by~\eqref{eq:sample-random-inc}\;
\For{time scale in time scales}{
reconstruct coarse pseudo-random vector from finest one using~\eqref{eq:sample-reconst}\;
initialise time-stepping relative to time scale\;
\For{time step in time steps}{
solve~\eqref{algo:Experiment-stepping}\;
}
}
compute sample stability\;
compute sample distance of coarse and fine approximates\;
}
compute empirical mean of sample stability by~\eqref{eq:statistical-stability}\;
compute empirical mean of sample distances by~\eqref{eq:statistical-mean}\;
\caption{Pseudo-code used for our numerical simulation.}
\label{algo:Pseudo-code}
\end{algorithm}
\end{framed}

We conducted four different experiments:
\begin{itemize}
\item we choose $p= 1.5$ or $p= 3$ in the definition of the viscous stress tensor~$S$;
\item we choose Taylor--Hood or Scott--Vogelius elements for the spatial discretisation. 
\end{itemize}
All other parameters are fixed to the values $N_f =  2^9$, $N_c \in \{2^2, 2^3, \ldots, 2^8 \}$ and $M = 1000$.

\subsection{Results}
In this section we present the results of Algorithm~\ref{algo:Pseudo-code}. 

In general we couldn't detect any differences in the approximation quality of Taylor--Hood elements compared to Scott--Vogelius elements. Both methods behave equally well with respect to time convergence (Figure~\ref{fig:ConvergencePlot}) and time stability (Figure~\ref{fig:StabilityPlot}) independently of $p = 1.5$ or $p=3$. 

Velocity and pressure convergence on $L^\infty_t L^2_x$ initially converge slower than the theoretically predicted rate~$1/2$. Only after a few refinements this rate can be observed. The symmetric gradient follows the theoretically predicted rate~$1/2$ consistently.

We observe stability of velocity's divergence on $L^2_t L^2_x$ for Taylor--Hood and Scott--Vogelius elements. Scott--Vogelius elements resolve the divergence-free constraint exactly; the computed velocity is divergence free up to the machine precision. The approximation obtained by Taylor--Hood elements isn't exactly divergence; but importantly, its $L^2_t L^2_x$ is uniformly bounded with respect to the time discretisation. 

Stability of pressure depends on the time discretisation. Initially, pressure measured on $W^{1,2}_t L^2_x$ is inversely proportional to the time resolution. In the super-quadratic case, $p=3$, a saturation effect is observed after a few time refinements, which eventually shows stability independently of the time discretisation. A similar effect is visible for the sub-quadratic case, $p=1.5$; however, the number of refinements needed for the saturation increases compared to the super-quadratic case.

\begin{figure}
\begin{center}
\includegraphics[scale=1]{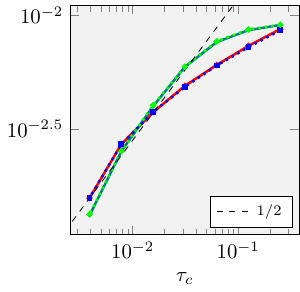}
\includegraphics[scale=1]{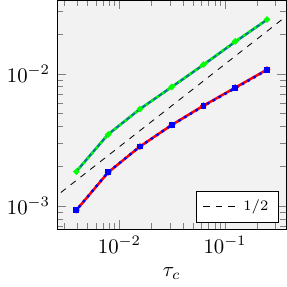}
\includegraphics[scale=1]{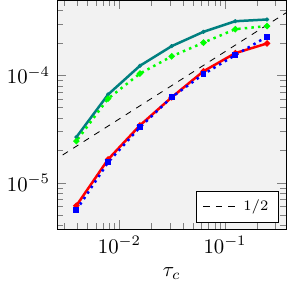}
\caption{Evolution of $\tau_c \mapsto \mathrm{E}^{\bullet,M}_{\bullet}(1/\tau_c,N_f)$ for velocity distances $\mathrm{E}^{\mathrm{vel},M}_{L^\infty_t L^2_x}$~(top left), $\mathrm{E}^{\mathrm{vel},M}_{L^2_t V}$~(top right) and pressure distance $\mathrm{E}^{\mathrm{pre},M}_{L^\infty_t L^2_x}$~(bottom) for Scott--Vogelius elements with $p=3$~(red solid) and $p=1.5$~(teal solid); and Taylor--Hood elements with $p=3$~(blue dotted) and $p=1.5$~(green dotted). }
\label{fig:ConvergencePlot}
\end{center}
\end{figure}

\begin{figure}
\begin{center}
\includegraphics[scale=1]{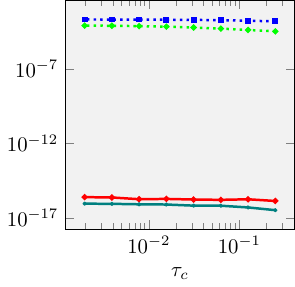}
\includegraphics[scale=1]{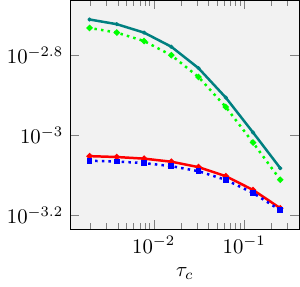}
\caption{Evolution of $\tau_c \mapsto \mathrm{K}^{\bullet,M}_{\bullet}(1/\tau_c)$ for velocity on $L^\infty_t L_{\Div}$~(left) and pressure on  $W^{1,2}_t L^2_x$~(right) for Scott--Vogelius elements with $p=3$~(red solid) and $p=1.5$~(teal solid); and Taylor--Hood elements with $p=3$~(blue dotted) and $p=1.5$~(green dotted).}
\label{fig:StabilityPlot}
\end{center}
\end{figure}

\subsection{Conclusions}
In the beginning of this section we asked: 'do non-exactly divergence free velocity approximations behave worse than exactly divergence free ones?'
Our numerical simulations provide a negative answer. We couldn't detect any qualitative difference between exactly and non-exactly divergence free approximations. Thus, from an applied point of view, there is no reason to favour exactly divergence free approximations. Contrary, one might favour non-exactly divergence free approximations, considering the additional effort need for the construction of exactly divergence free ones: we had to construct a mesh without singular vertices. Instead of constructing a non-singular mesh, we could alternatively restrict the pressure space at singular vertices to restore the inf-sup stability of the Scott--Vogelius pair. Both steps, the construction of a non-singular mesh and the restriction of the pressure space, are non-trivial ingredients in the algorithm. If the resulting simulations are qualitatively indistinguishable from algorithms that are easier to implement, e.g. Taylor--Hood elements, there is no reason (besides the theoretical justification as given by Theorem~\ref{thm:Convergence}) to favour Scott--Vogelius elements over Taylor--Hood elements. Moreover, our simulations suggest that Theorem~\ref{thm:Convergence} (with suitable modifications) might be true for non-exactly divergence free approximations. 

Another point is the numerically observed convergence of pressure on $L^\infty_t L^2_x$. A rigorous investigation of the pressure error is missing even in the unsteady deterministic situation. Related results for the steady case can be found in e.g.~\cite{MR4613234} and the references therein. It will be content of future research to explain this convergence for deterministic and stochastic generalised Stokes systems.

\appendix
\section{Pressure norms, tensor relations and extrapolation}
\label{app:Appendix}

\subsection{Pressure norms}
\label{sec:PressureNorms}
The inf-sup condition~\eqref{eq:InfSupAbst} is trivial if one works with norms defined by duality. 
\begin{lemma}
Let $\norm{\cdot}_V$ be a norm on $V_h$. Define the relation 
\begin{align} \label{def:Relation}
q_1 \sim q_2 \in Q_h \quad \Leftrightarrow \quad \forall v \in V_{h,\Div}^\perp: (q_1 - q_2, \Div v) =0.
\end{align}
Additionally, we set
\begin{align} \label{eq:DualDefinition}
\norm{q}_{V^*} := \sup_{v \in V_{h,\Div}^\perp} \frac{(q,\Div v)}{\norm{v}_V}.
\end{align}
Then $\norm{\cdot}_{V^*}$ defines a norm on $Q_h^{\sim}:= Q_h \slash \sim$.  Moreover, $\beta_h = 1$.
\end{lemma}
\begin{proof}
We need to restrict to the equivalence classes of~$\sim$ in order to recover the definiteness of~$\norm{\cdot}_{V^*}$. 
All other properties are easily checked. Thus, $\norm{\cdot}_{V^*}$ defines a norm on $Q_h^{\sim}$. 

The inf-sup constant can be computed by~\eqref{eq:DualDefinition}.
\end{proof}

\begin{lemma}\label{lem:Equi}
The following statements are equivalent:
\begin{enumerate}
\item \label{it:a Equiv} $\nabla^h :Q_h \to V_{h,\Div}^\perp$ is one to one,
\item \label{it:b Equiv} the relation~$\sim$ defined in~\eqref{def:Relation} is trivial,
\item \label{it:c Equiv}$(q,\Div v) = 0$ for all $v \in V_{h,\Div}^\perp$ implies $q = 0$,
\item \label{it:d Equiv}for all $q \in Q_h \backslash\{0\}$ exists $v \in V_{h,\Div}^\perp$ such that $(q,\Div v) \neq 0$.
\end{enumerate}
\end{lemma}

\begin{remark}
If $Q_h \subset \Div V_h$, then~\ref{it:c Equiv} is satisfied. Moreover,~\ref{it:d Equiv} is the complementary non-degeneracy condition to~\eqref{eq:NonDeg} (see also Lemma~\ref{lem:NonDegenerate}).
\end{remark}

\begin{proof}[Proof of Lemma~\ref{lem:Equi}]
$\ref{it:a Equiv} \Rightarrow \ref{it:b Equiv}$: Let $q_1 \sim q_2$, i.e., for all $v \in V_{h,\Div}^\perp$
\begin{align*}
0 &= (q_1 - q_2, \Div v) = -(\nabla^h(q_1 - q_2), v).
\end{align*}
Choosing~$v = \nabla^h(q_1 - q_2) \in V_{h,\Div}^\perp$ yields
\begin{align*}
\norm{\nabla^h(q_1 -q_2)}_{L^2_x}^2 = 0.
\end{align*}
Since $\nabla^h$ is one to one we conclude $q_1 = q_2$.

$\ref{it:b Equiv} \Rightarrow \ref{it:c Equiv}$: Let $(q,\Div v) = 0$ for all $v \in V_{h,\Div}^\perp$. Since 
\begin{align*}
0 &= (q,\Div v) = (q - 0,\Div v)
\end{align*}
for all $v \in V_{h,\Div}^\perp$ we find $q \sim 0$. Thus, due to triviality of~$\sim$, it holds $q = 0$.

$\ref{it:c Equiv} \Rightarrow \ref{it:a Equiv}$: Let $q_1, q_2 \in Q_h$ such that $\nabla^h q_1 = \nabla^h q_2$. In particular, for all $v \in V_{h,\Div}^\perp$
\begin{align*}
0 = -(\nabla^h q_1- \nabla^h q_2,v) = (q_1 - q_2, \Div v).
\end{align*}
Thus $q_1 = q_2$.

$\ref{it:c Equiv} \Leftrightarrow \ref{it:d Equiv}$: Notice that
\begin{align*}
&\big\{ [\forall v \in V_{h,\Div}^\perp: (q,\Div v) = 0] \Rightarrow [q=0] \big\}\\
\Leftrightarrow \hspace{2em} &\big\{ [q\neq 0] \Rightarrow [\exists v \in V_{h,\Div}^\perp: (q,\Div v) \neq 0]   \big\}.
\end{align*}
\end{proof}

\begin{lemma} \label{lem:NonDegenerate}
For all $v \in V_{h,\Div}^\perp \backslash \{ 0\}$ exists $q \in Q_h$ such that~\eqref{eq:NonDeg} holds.
\end{lemma} 
\begin{proof}
We show that its negation is false. Let us assume there exists $v \in V_{h,\Div}^\perp \backslash \{ 0\}$ such that for all $q \in Q_h$
\begin{align*}
(q,\Div v) = 0.
\end{align*}
Therefore $v \in V_{h,\Div}$. But this contradicts $v \in V_{h,\Div}^\perp \backslash \{ 0\} \cap V_{h,\Div} = \emptyset$.
\end{proof}

The discrete gradient~\eqref{eq:DiscreteGrad} allows us to represent the stochastic norm in a different form.
\begin{lemma}
Let $q \in L^1$. Then it holds
\begin{align} \label{eq:PressureNormAlter}
\norm{q}_{Q_\mathrm{sto}} = \norm{\Pi_{\Div}^\perp \nabla^h q}_{L^2_x}.
\end{align}
\end{lemma}

The next lemma addresses well posedness of the initial projection~\eqref{eq:InitPressure02}.
\begin{lemma} \label{lem:InitialPressure}
Let $u_0 \in V_h$. Then there exists a unique $\pi^{\mathrm{init}} \in Q_h^{\sim}$ solving~\eqref{eq:InitPressure01}. Moreover, it holds
\begin{align} \label{eq:InitialPressureEstimate}
\norm{\pi^{\mathrm{init}}}_{Q_\mathrm{sto}} = \norm{\Pi_{\Div}^\perp u_0}_{L^2_x}.
\end{align}
\end{lemma}
\begin{proof}
The existence and uniqueness of~$\pi^{\mathrm{init}} \in Q_h^{\sim}$ that solves~\eqref{eq:InitPressure01} follows from Babuska's lemma. 

Next, we prove the norm equality. Let $0 \neq \xi \in V_{h,\Div}^\perp$. Divide both sides of~\eqref{eq:InitPressure01} by $\norm{\xi}_{L^2_x}$ and take the supremum over $\xi \in V_{h,\Div}^\perp$
\begin{align*}
\norm{\pi^{\mathrm{init}}}_{Q_{\mathrm{sto}}}= \sup_{\xi \in V_{h,\Div}^\perp} \frac{\left( u_0, \xi \right)}{\norm{\xi}_{L^2_x}}.
\end{align*}
Using the projection property and symmetry of $\Pi_{\Div}^\perp$ we derive
\begin{align*}
\sup_{\xi \in V_{h,\Div}^\perp} \frac{\left( u_0, \xi \right)}{\norm{\xi}_{L^2_x}} = \sup_{\xi \in V_{h,\Div}^\perp} \frac{\left(\Pi_{\Div}^\perp u_0, \xi \right)}{\norm{\xi}_{L^2_x}}.
\end{align*}
The supremum is attained by choosing $\xi = \Pi_{\Div}^\perp u_0 \in V_{h,\Div}^\perp$.
\end{proof}

\subsection{Uniformly convex $N$-functions}
The following results are standard in the context of uniformly convex $N$-functions. For more details see e.g.~\cite[Appendix~B]{DieForTomWan20}.

Let $p \in (1,\infty)$ and $\kappa \geq 0$. Define $S, V : \mathbb{R}^{n\times n} \to \mathbb{R}^{n\times n}$ by 
\begin{align}
S(A) := \left( \kappa + \abs{A} \right)^{p-2} A \quad \text{ and } \quad V(A) := \left( \kappa + \abs{A} \right)^{(p-2)/2} A.
\end{align}

\begin{lemma}[Relation of tensors]\label{lem:Relation}
Let $A, B \in \mathbb{R}^{n\times n}$. Then
\begin{align}\label{eq:Monotone}
\big( S(A) - S(B) \big):(A-B) \eqsim \abs{V(A) - V(B)}^2. 
\end{align}
\begin{enumerate}
\item If $ p \geq 2$. Then
\begin{align} \label{eq:RelationGeq2}
\abs{A - B}^p \lesssim \abs{V(A) - V(B)}^2 \lesssim \abs{S(A) - S(B)}^{p'}.
\end{align}
\item If $p \leq 2$. Then 
\begin{align} \label{eq:RelationLeq2}
\abs{S(A) - S(B)}^{p'} \lesssim \abs{V(A) - V(B)}^2 \lesssim \abs{A - B}^p.
\end{align}
\end{enumerate}
\end{lemma}

\begin{lemma}[Generalized Young's inequality] \label{lem:GeneralizedYoung}
Let $\delta > 0$. Then there exists $c_\delta \geq 1$ such that for all $A,B, C \in \mathbb{R}^{n\times n}$ 
\begin{align} \label{eq:GeneralizedYoung}
\big( S(A) - S(B) \big): (C - B) &\leq \delta \abs{V(A) - V(B)}^2 + c_\delta \abs{V(C) - V(B)}^2.
\end{align}
\end{lemma}

\subsection{Discrete extrapolation} 
In this appendix we extend the continuous extrapolation result presented in~\cite[Chapter~4]{MR1725357} to the time-discrete framework. It is then used to transfer regularity results for continuous stochastic integration as presented in~\cite{MR4116708} to the time-discrete case. This is the foundation for the estimates presented in Section~\ref{sec:PressureStrong}.

\subsubsection{Domination}

\begin{definition}[Domination relation]
Let $\big( \Omega; \mathcal{F}, (\mathcal{F}_n), \mathbb{P} \big)$ be a filtered probability space. Moreover, let $(X_n)$ and $(Y_n)$ be two real-valued, non-negative, $(\mathcal{F}_n)$-adapted processes. We call~$(X_n)$ dominated by $(Y_n)$ if there exists a constant~$C\geq 1$ such that for all stopping times~$\mathfrak{n}$ with respect to the shifted filtration~$(\mathcal{H}_n) := (\mathcal{F}_{n+1})$ it holds 
\begin{align} \label{eq:DominationCondition}
\mathbb{E} \left[ X_\mathfrak{n} \right] \leq C \mathbb{E} \left[ Y_\mathfrak{n} \right].
\end{align}
\end{definition}

\begin{lemma} \label{lem:Domination}
Let $(X_n)$ be dominated by~$(Y_n)$. Then for all $x, y > 0$ it holds
\begin{align}
\mathbb{P}\left( X^*  > x ; Y^* \leq y \right) <  C \frac{1}{x} \mathbb{E} \left[ Y^* \wedge y \right],
\end{align}
where $Z^* = \sup_{n} Z_n$, $Z \in \{X,Y\}$.
\end{lemma}

\begin{proof}
We define the $(\mathcal{F}_n)$-stopping times
\begin{align*}
\mathfrak{n}_y &:= \inf\{ n: \, Y_n > y \} \quad \text{ and } \quad 
\mathfrak{m}_x := \inf\{ n: \, X_n > x \},
\end{align*}
where we use the convention $\inf \emptyset = \infty$. This allows us to rewrite the condition on the supremal process in terms of stopping times. Indeed, 
\begin{align*}
\{ Y^* \leq y \} &= \{ \mathfrak{n}_y = \infty\}, \\
\{ X^* > x \} &= \{ \mathfrak{m}_x < \infty\} =  \{ X_{\mathfrak{m}_x} > x; \mathfrak{m}_x < \infty\}.
\end{align*}
Thus,
\begin{align*}
\mathbb{P}\left( X^*  > x ; Y^* \leq y \right) &= \mathbb{P}\left( X_{\mathfrak{m}_x}  > x ; \mathfrak{m}_x < \infty; \mathfrak{n}_y = \infty \right) \\
&=  \mathbb{P}\left( X_{\mathfrak{m}_x \wedge (\mathfrak{n}_y-1)}  > x ; \mathfrak{m}_x < \infty ; \mathfrak{n}_y = \infty \right) \\
&\leq \mathbb{P}\left( X_{\mathfrak{m}_x \wedge (\mathfrak{n}_y-1)}  > x\right).
\end{align*}
Next, we remark that 
\begin{align*}
\mathfrak{\delta}:= \mathfrak{m}_x \wedge (\mathfrak{n}_y - 1)
\end{align*}
defines an $(\mathcal{H}_n)$-stopping time. Chebycheff's inequality and~\eqref{eq:DominationCondition} show
\begin{align*}
 \mathbb{P}\left( X_{\mathfrak{m}_x \wedge (\mathfrak{n}_y-1)}  > x \right) < \frac{1}{x} \mathbb{E}\left[X_{\mathfrak{m}_x \wedge (\mathfrak{n}_y-1)} \right] \leq \frac{C}{x} \mathbb{E}\left[Y_{\mathfrak{m}_x \wedge (\mathfrak{n}_y-1)} \right] .
\end{align*}

For all $k \leq \mathfrak{n}_y - 1$ we have $Y_k \leq y$. Hence, $Y_{\mathfrak{m}_x \wedge (\mathfrak{n}_y-1)} \leq y$.
Additionally, it holds trivially $Y_{\mathfrak{m}_x \wedge (\mathfrak{n}_y-1)} \leq Y^*$. Overall, we find $\mathbb{P}$-a.s.
\begin{align*}
Y_{\mathfrak{m}_x \wedge (\mathfrak{n}_y-1)}  \leq Y^* \wedge y.
\end{align*}
The assertion follows by taking expectation.
\end{proof}

Extrapolation allows us to relate lower moments of maximal processes for processes that are in a domination relation.
\begin{corollary}[Extrapolation] \label{cor:Domination}
Let the conditions of Lemma~\ref{lem:Domination} be satisfied and $k \in (0,1)$. Then
\begin{align*}
\mathbb{E}\left[ (X^*)^k \right] < \frac{1 + C - k}{1-k} \mathbb{E}\left[( Y^*)^k \right].
\end{align*}
\end{corollary}
\begin{proof}
The proof follows along the lines of the one for~\cite[Proposition~4.7]{MR1725357}.
\end{proof}

\subsubsection{Application} \label{app:Application}
Let $r \in (0,\infty)$. We define the time discrete filtration
\begin{align}
(\mathcal{F}_n^N) &= (\mathcal{F}_{t_{n}+\tau/2})
\end{align}
and the $(\mathcal{F}_n^N)$-adapted processes
\begin{align} \label{eq:XExtra}
X_M &:= \left( \sum_{n=1}^M \int_{J_n} \norm{\frac{\overline{E}(t)}{\sqrt{\tau}}}_{L^2_x}^r \dd t \right) 1_{\{M \leq N\}}, \\ \label{eq:YExtra}
Y_M &:= \max_{n \leq (M+2) \wedge N} \norm{G_n(u_{(n-2)\vee 0})}_{L_2(\mathfrak{U};L^2_x)}^r,
\end{align}
where~$\overline{E}$ is defined by~\eqref{eq:CompensatorIntegral}.

\begin{lemma} \label{lem:DominationCheck}
$(X_n)$ is dominated by~$(Y_n)$.
\end{lemma}
\begin{proof}
Define $(\mathcal{H}_n^N) = (\mathcal{F}_{n+1}^N)$. Let $\mathfrak{n}$ be an~$(\mathcal{H}_n^N)$ stopping time. Then
\begin{align*}
\mathbb{E}\left[ X_{\mathfrak{n}} \right] &= \mathbb{E}\left[ \left( \sum_{n=1}^{\mathfrak{n}} \int_{J_n} \norm{\frac{\overline{E}(t)}{\sqrt{\tau}}}_{L^2_x}^r \dd t\right) 1_{\{ \mathfrak{n}\leq N\}} \right] \\
&\leq \mathbb{E}\left[  \sum_{n=1}^{N} \int_{J_n} 1_{\{\mathfrak{n} \geq n\} }  \norm{\frac{\overline{E}(t)}{\sqrt{\tau}}}_{L^2_x}^r \dd t \right].
\end{align*}

Notice that $\set{\mathfrak{n} \geq n} \subset \set{\mathfrak{n} \geq n-1}$ and therefore $1_{\{\mathfrak{n} \geq n\} }  \leq 1_{\{\mathfrak{n} \geq n-1\} } $. Moreover, for $n \geq 2$ we find $\set{\mathfrak{n} \geq n-1} = \set{ \mathfrak{n} \leq n-2 }^c \in \mathcal{H}_{n-2}^N = \mathcal{F}_{n-1}^N$. Similarly, $\set{\mathfrak{n} \geq 0} = \Omega \in \mathcal{F}_0^N$. Overall, $1_{\{\mathfrak{n} \geq n-1\} } $ is $\mathcal{F}_{(n-1)\vee 0}^N$-measurable. The tower property of conditional expectation shows

\begin{align*}
\mathbb{E}\left[  \sum_{n=1}^{N} 1_{\{\mathfrak{n} \geq n\} }   \int_{J_n} \norm{\frac{\overline{E}(t)}{\sqrt{\tau}}}_{L^2_x}^r \dd t \right] &\leq  \sum_{n=1}^{N}  \mathbb{E}\left[  1_{\{\mathfrak{n} \geq n-1\} } \int_{J_n} \norm{\frac{\overline{E}(t)}{\sqrt{\tau}}}_{L^2_x}^r \dd t\right]  \\
&= \sum_{n=1}^{N} \mathbb{E}\left[  \mathbb{E}\left[  1_{\{\mathfrak{n} \geq n-1\} }  \int_{J_n} \norm{\frac{\overline{E}(t)}{\sqrt{\tau}}}_{L^2_x}^r \dd t \bigg| \mathcal{F}_{(n-1)\vee 0}^N \right] \right] \\
&= \mathbb{E}\left[ \sum_{n=1}^{N} 1_{\{\mathfrak{n} \geq n-1\} }  \int_{J_n}    \mathbb{E}\left[  \norm{\frac{\overline{E}(t)}{\sqrt{\tau}}}_{L^2_x}^r \bigg| \mathcal{F}_{(n-1)\vee 0}^N \right]  \dd t \right].
\end{align*}

Recall the definition of $\overline{E}$, cf.~\eqref{eq:CompensatorIntegral}. An application of~\cite[Lemma~3.1]{MR4116708} shows $\mathbb{P}$-a.s. for all $n \in [N]$ and $t \in J_n$
\begin{align*}
&\mathbb{E}\left[ \norm{\int_{t}^{t_{n + 1/2}} a_n(s) G_n(u_{(n-2)\vee 0}) \dd W(s) }_{L^2_x}^r \bigg| \mathcal{F}_{(n-1)\vee 0}^N \right] \\
&\hspace{3em} \leq C_r \mathbb{E}\left[ \left( \int_{t}^{t_{n + 1/2}}  a_n^2(s) \norm{G_n(u_{(n-2)\vee 0}) }_{L_2(\mathfrak{U}; L^2_x)}^2 \dd s \right)^{r/2} \right] \\
&\hspace{3em} \leq C_r (\tau)^{r/2} \mathbb{E}\left[ \norm{G_n(u_{(n-2)\vee 0}) }_{L_2(\mathfrak{U};L^2_x)}^r \right] .
\end{align*}

Similarly, for $n \in [N-1]$ and $t \in J_n$,
\begin{align*}
&\mathbb{E}\left[ \norm{\int_{ t_{n-1/2}}^{t} a_{n+1}(s)  G_{n+1}( u_{(n-1)\vee 0}) \dd W(s)  }_{L^2_x}^r \bigg|  \mathcal{F}_{(n-1)\vee 0}^N \right]\\
&\hspace{3em} \leq C_r (\tau)^{r/2} \mathbb{E}\left[ \norm{G_{n+1}(u_{(n-1)\vee 0}) }_{L_2(\mathfrak{U};L^2_x)}^r \right].
\end{align*}
Therefore, $\mathbb{P}$-a.s. for $n \in [N]$
\begin{align} \label{eq:EstimateCondition}
\begin{aligned}
&\sup_{t\in J_n} \mathbb{E}\left[  \norm{\frac{\overline{E}(t)}{\sqrt{\tau}}}_{L^2_x}^r \bigg| \mathcal{F}_{(n-1)\vee 0}^N \right]  \lesssim C_r \Big( \mathbb{E}\left[ \norm{G_n(u_{(n-2)\vee 0}) }_{L_2(\mathfrak{U};L^2_x)}^r \right] \\
&\hspace{4em} + \mathbb{E}\left[ \norm{G_{n+1}(u_{(n-1)\vee 0}) }_{L_2(\mathfrak{U};L^2_x)}^r \right] 1_{\{n+1 \leq N\} } \Big).
\end{aligned}
\end{align}

This allows us to conclude
\begin{align*}
&\mathbb{E}\left[ \sum_{n=1}^{N} 1_{\{\mathfrak{n} \geq n-1\} }  \int_{J_n}    \mathbb{E}\left[  \norm{\frac{\overline{E}(t)}{\sqrt{\tau}}}_{L^2_x}^r \bigg| \mathcal{F}_{(n-1)\vee 0}^N \right]  \dd t \right]\\
&\quad \lesssim \mathbb{E}\left[  \sum_{n=1}^{N-1} \tau   1_{\{\mathfrak{n} \geq n-1\} } C_r \mathbb{E}\left[ \norm{G_n(u_{(n-2)\vee 0})}_{L_2(\mathfrak{U};L^2_x)}^r  +  \norm{G_{n+1}(u_{(n-1)\vee 0})}_{L_2(\mathfrak{U};L^2_x)}^r \right] \right] \\
&\quad \quad +\mathbb{E}\left[   \tau   1_{\{\mathfrak{n} \geq N-1\} } C_r \mathbb{E}\left[ \norm{G_N(u_{(N-2)\vee 0})}_{L_2(\mathfrak{U};L^2_x)}^r  \right] \right] \\
&\quad \lesssim C_r \mathbb{E}\left[ \max_{n \leq (\mathfrak{n}+ 2) \wedge N} \norm{G_n(u_{(n-2)\vee 0})}_{L_2(\mathfrak{U};L^2_x)}^r \right] \mathbb{E}\left[ \sum_{n=1}^{N} \tau   1_{\{\mathfrak{n} \geq n-1\} } \right] \\
&\quad \lesssim C_r T \mathbb{E}\left[ \max_{n \leq (\mathfrak{n}+ 2) \wedge N} \norm{G_n(u_{(n-2)\vee 0})}_{L_2(\mathfrak{U};L^2_x)}^r  \right] \leq\tilde{C}_r \mathbb{E} \left[ Y_{\mathfrak{n}} \right].
\end{align*}

Overall, we have verified
\begin{align}
\mathbb{E}\left[ X_{\mathfrak{n}} \right]  \leq \tilde{C}_r \mathbb{E} \left[ Y_{\mathfrak{n}} \right].
\end{align}
The assertion follows since~$\mathfrak{n}$ was arbitrary.
\end{proof}

\begin{remark}
Since $C_r \sim r^{r/2}$ one can estimate $ \tilde{C}_r \leq C 2^r r^{r/2}$ for large $r$.
\end{remark}

%% New biber style  (Remove \i!!!)
\printbibliography 

%\bibliographystyle{amsalpha}
%\bibliography{numerics}

\end{document}